\documentclass[11pt,a4paper,reqno]{amsart}

\title[Quantum Chevalley formula for partial flag manifolds]
{Chevalley formula for anti-dominant minuscule fundamental weights 
in the equivariant quantum $K$-group of partial flag manifolds}
\author{Takafumi Kouno, Satoshi Naito, and Daisuke Sagaki}
\address[Takafumi Kouno]{Department of Mathematics, Tokyo Institute of Technology, 
2-12-1 Oh-okayama, Meguro-ku, Tokyo 152-8551, Japan.}
\email{kouno.t.ab@m.titech.ac.jp}
\address[Satoshi Naito]{Department of Mathematics, Tokyo Institute of Technology, 
2-12-1 Oh-okayama, Meguro-ku, Tokyo 152-8551, Japan.}
\email{naito@math.titech.ac.jp}
\address[Daisuke Sagaki]{Department of Mathematics, 
Faculty of Pure and Applied Sciences, University of Tsukuba, 
1-1-1 Tennodai, Tsukuba, Ibaraki 305-8571, Japan.}
\email{sagaki@math.tsukuba.ac.jp}
\date{}
\subjclass[2010]{Primary 17B37; Secondary 14N15, 14M15, 33D52, 81R10}
\keywords{quantum Chevalley formula, quantum LS path, semi-infinite 
flag manifold, Grassmannian, (quantum) Schubert calculus.}

\usepackage[truedimen,margin=25truemm]{geometry} % margin size
\usepackage{amsmath,amssymb,amsthm,amscd}
\usepackage{bm, time}
\usepackage{ytableau}

\allowdisplaybreaks
\numberwithin{equation}{section}

\newcommand{\BZ}{\mathbb{Z}}
\newcommand{\BC}{\mathbb{C}}

\newcommand{\SD}{\mathsf{D}}
\newcommand{\SB}{\mathsf{B}}
\newcommand{\SQ}{\mathsf{Q}}
\newcommand{\SE}{\mathsf{E}}

\newcommand{\be}{\mathbf{e}}
\newcommand{\bp}{\mathbf{p}}
\newcommand{\bq}{\mathbf{q}}
\newcommand{\bx}{\mathbf{x}}
\newcommand{\bsig}{\bm{\sigma}}

\newcommand{\q}{\mathsf{v}}

\newcommand{\bA}{\mathbf{A}}
\newcommand{\bB}{\mathbf{B}}
\newcommand{\bG}{\mathbf{G}}
\newcommand{\bH}{\mathbf{H}}
\newcommand{\bX}{\mathbf{X}}

\newcommand{\bGx}[2]{ (\bG_{#1}^{\lhd})_{ #2 } }
\newcommand{\bHx}[2]{ (\bH_{#1}^{\lhd})_{ #2 } }
\newcommand{\bXx}[2]{ (\bX_{#1}^{\lhd})_{ #2 } }

\newcommand{\gq}{\gamma_{\SQ}}

\DeclareMathOperator{\gch}{gch}
\DeclareMathOperator{\wt}{wt}
\DeclareMathOperator{\Deg}{deg}
\DeclareMathOperator{\Hom}{Hom}
\DeclareMathOperator{\Inv}{Inv}
\DeclareMathOperator{\ed}{end}
\DeclareMathOperator{\bgn}{start}

\newcommand{\QBG}{\mathrm{QBG}}
\newcommand{\BG}{\mathrm{BG}}
\newcommand{\QLS}{\mathrm{QLS}}
\newcommand{\LS}{\mathrm{LS}}

\newcommand{\Fg}{\mathfrak{g}}
\newcommand{\Fh}{\mathfrak{h}}

\newcommand{\vpi}{\varpi}

\newcommand{\lng}{w_{\circ}}

\newcommand{\af}{\mathrm{af}}
\newcommand{\lo}{\mathrm{long}}
\newcommand{\sh}{\mathrm{short}}
\newcommand{\qtm}{\mathrm{quantum}}

\newcommand{\edge}[1]{ \xrightarrow{\hspace{2pt}#1\hspace{2pt}} }

\newcommand{\mcr}[1]{\lfloor #1 \rfloor}

\newcommand{\ti}[1]{\widetilde{#1}}
\newcommand{\ha}[1]{\widehat{#1}}

\newcommand{\dtb}[1]{\le_{#1}^{\ast}}
\newcommand{\tbmax}[3]{\max(#1W_{#2},\le_{#3}^{\ast}\nobreak)}
\newcommand{\kap}[2]{\kappa(#1,#2)}
\newcommand{\zet}[2]{\zeta(#1,#2)}

\newcommand{\rhdeq}{\trianglerighteq}

\newcommand{\Bpair}[2]{\Bigl\langle #1,\,#2 \Bigr\rangle}
\newcommand{\pair}[2]{\langle #1,\,#2 \rangle}

\newcommand{\J}{J}
\newcommand{\WJs}{W_{\J}}
\newcommand{\WJ}{W^{\J}}

\newcommand{\DJp}{\Delta^{+} \setminus \Delta_{\J}^{+}}
\newcommand{\DJs}{\Delta_{\J}^{+}}

\newcommand{\yJ}[1]{\Phi^{\J}(#1)}
\newcommand{\yJs}[1]{\Phi_{\J}(#1)}
\newcommand{\xJ}[1]{\Psi^{\J}(#1)}
\newcommand{\xJs}[1]{\Psi_{\J}(#1)}

\newcommand{\WJx}{\WJ_{\ge s_{n}s_{n-1}s_{n}}}
\newcommand{\WJe}{\WJ \setminus \{e\}}

\newcommand{\ls}{w_{\J,\circ}}

\newcommand{\QBa}{\mathrm{QBG}_{\sigma\mu}(\WJ)}
\newcommand{\QBb}[1]{\mathrm{QBG}_{#1\mu}(\WJ)}
\newcommand{\Ba}{\mathrm{BG}_{\sigma\mu}(\WJ)}
\newcommand{\Bb}[1]{\mathrm{BG}_{#1\mu}(\WJ)}

\newcommand{\bBG}[1]{\mathbf{BG}_{#1}^{\lhd}}
\newcommand{\bQBG}[1]{\mathbf{QBG}_{#1}^{\lhd}}

\newcommand{\Lie}{\mathrm{Lie}}
\newcommand{\Fun}{\mathrm{Fun}}
\newcommand{\ngt}{\mathrm{neg}}
\newcommand{\ess}{\mathrm{ess}}
\newcommand{\CO}{\mathcal{O}}
\newcommand{\CB}{\mathcal{B}}
\newcommand{\sdot}{\,\cdot\,}
\newcommand{\Gr}{\mathop{\rm Gr}\nolimits}
\newcommand{\PrJ}{P_{\J}}
\newcommand{\bQG}{\mathbf{Q}}
\newcommand{\bQGr}{\mathbf{Q}^{\mathrm{rat}}}
\newcommand{\bQGJ}{\mathbf{Q}_{\J}}
\newcommand{\bQGJr}{\mathbf{Q}_{\J}^{\mathrm{rat}}}
\newcommand{\Bv}[1]{\CB_{\J}^{#1}}

\newcommand{\pra}[1]{(\hspace{-1pt}(#1)\hspace{-1pt})}

\theoremstyle{definition}
\newtheorem{thm}{Theorem}[section]
\newtheorem{dfn}[thm]{Definition}
\newtheorem{lem}[thm]{Lemma}
\newtheorem{prop}[thm]{Proposition}

\newtheorem{thmintro}{Theorem}

\theoremstyle{remark}
\newtheorem{rem}[thm]{Remark}
\newtheorem*{rem*}{Remark}

\newenvironment{enu}{%
 \begin{enumerate}%
}{\end{enumerate}}

\begin{document}
%%%%%%%%%%%%%%%%

%=======================%
%     START ABSTRACT    %
%=======================%
%
\begin{abstract}
In this paper, we give an explicit formula of Chevalley type, 
in terms of the Bruhat graph, 
for the quantum multiplication with the class of 
the line bundle associated to the anti-dominant minuscule fundamental weight $- \vpi_{k}$ 
in the torus-equivariant quantum $K$-group of the partial flag manifold $G/\PrJ$ 
(where $\J=I \setminus \{k\}$) corresponding to the maximal (standard) parabolic subgroup $\PrJ$ 
of minuscule type in type $A$, $D$, $E$, or $B$.
This result is obtained by proving a similar formula in 
a torus-equivariant $K$-group of 
the semi-infinite partial flag manifold $\bQGJ$ 
of minuscule type, and then 
by making use of the isomorphism between 
the torus-equivariant quantum $K$-group of $G/\PrJ$ and 
the torus-equivariant $K$-group of $\bQGJ$, 
recently established by Kato.
\end{abstract}

\maketitle

%=========================%
%     START SECTION 01    %
%=========================%
%
\section{Introduction.} 
\label{sec:intro}
Let $\bQGr$ denote the (whole) semi-infinite flag manifold, 
which is the reduced ind-scheme whose set of 
$\BC$-valued points is $G(\BC\pra{z})/(T \cdot N(\BC\pra{z}))$ (see \cite{K2} for details), 
where $G$ is a simply-connected simple algebraic group over $\BC$ 
with Borel subgroup $B = T N$, $T$ maximal torus and $N$ unipotent radical.
In this paper, we concentrate on the semi-infinite Schubert (sub)variety 
$\bQG := \bQG(e) \subset \bQGr$ associated to the identity element $e$ of 
the affine Weyl group $W_{\af} = W \ltimes Q^{\vee}$, 
with $W = \langle s_{i} \mid i \in I \rangle$ the Weyl group and 
$Q^{\vee} = \sum_{i \in I} \BZ \alpha_{i}^{\vee}$ the coroot lattice of $G$;
we also call $\bQG$ the semi-infinite flag manifold.
The study of an equivariant $K$-group of $\bQG$ was started in \cite{KNS}, 
in which a Chevalley formula for dominant weights was obtained.
Shortly afterward, in \cite{NOS}, we proved a Chevalley formula 
for anti-dominant weights in a $T$-equivariant $K$-group $K_{T}^{\prime}(\bQG)$ of $\bQG$.

A breakthrough in the study of the equivariant $K$-group of $\bQG$ was 
achieved in \cite{K1} (see also \cite{K3}), in which Kato established a $\BC[P]$-module isomorphism 
from the (small) $T$-equivariant quantum $K$-group $QK_{T}(G/B)$ of 
the finite-dimensional flag manifold $G/B$ 
onto the $T$-equivariant $K$-group $K_{T}^{\prime}(\bQG)$ of $\bQG$, 
where $P = \sum_{i \in I} \BZ \vpi_{i}$ is 
the weight lattice of $G$ and $\BZ[P] (\subset \BC[P])$ is identified with the representation ring of $T$.
This $\BC[P]$-module isomorphism sends each (opposite) Schubert class in $QK_{T}(G/B)$ 
to the corresponding semi-infinite Schubert class in $K_{T}^{\prime}(\bQG)$.
Moreover, it respects the quantum multiplication $\star$ in $QK_{T}(G/B)$ and 
the tensor product in $K_{T}^{\prime}(\bQG)$; to be more precise, 
it respects the quantum multiplication $\star$ with the class of 
the line bundle $[\CO_{G/B}(- \vpi_{k})]$ and the tensor product 
with the class of the line bundle $[\CO_{\bQG}(- \vpi_{k})]$ for each $k \in I$.
In view of this result, the formula for the quantum multiplication 
with $[\CO_{G/B}(- \vpi_{k})]$, $k \in I$, in $QK_{T}(G/B)$ is 
immediately obtained from a Chevalley formula 
in $K_{T}^{\prime}(\bQG)$ obtained in \cite{NOS}; see \cite{LNS} for details. 

Let $k \in I$ be such that the fundamental weight $\vpi_{k}$ is minuscule, 
and set $\J:=I \setminus \{k\}$. 
The purpose of this paper is to give an explicit formula of Chevalley type, 
in terms of the Bruhat graph, for the quantum multiplication $\star$ 
with the class of the line bundle $[\CO_{G/\PrJ}(- \vpi_{k})]$ 
in the (small) $T$-equivariant quantum $K$-group 
$QK_{T}(G/\PrJ) = K_{T}(G/\PrJ) \otimes \BC[Q_{k}]$, 
where $K_{T}(G/\PrJ)$ is the $T$-equivariant $K$-group of 
the (finite-dimensional) partial flag manifold $G/\PrJ$, with $\PrJ \supset B$ 
the maximal (standard) parabolic subgroup of $G$ associated 
to the subset $\J=I \setminus \{k\}$, and $\BC[Q_{k}]$ is the polynomial ring 
in the (Novikov) variable $Q_{k}$ corresponding to the simple coroot $\alpha_{k}^{\vee}$. 
In this paper, we deal with the cases that $G$ 
(or its Lie algebra $\Fg := \Lie(G)$) is of types $A$, $D$, $E$, and $B$; 
in our forthcoming paper \cite{KoNS2}, 
we deal with the case that $G$ is of type $C$
but $\vpi_{k}$ is an arbitrary fundamental weight, and also
the case that $G$ is of type $B$ and $\vpi_{k}$ is a cominuscule weight.
Let us state the main result (Theorem~\ref{thm:introThm1} below) of this paper.
Let $\WJ = W^{I \setminus \{k\}}$ denote 
the set of minimal(-length) representatives for $W/\WJs$, 
with $\WJs = \langle s_{i} \mid i \in \J = I \setminus \{k\} \rangle$ 
the stabilizer of $\vpi_{k}$ in $W$; 
for $w \in W$, we denote by $\mcr{w} \in \WJ$ 
the representative of the coset $w \WJs$.
For $x \in \WJ$, we denote by $\bBG{x}$ 
the set of all directed paths 
$\bp:y_{0} \edge{\gamma_{1}} \cdots \edge{\gamma_{s}} y_{s}$
in the Bruhat graph $\BG(W)$ such that 
$y_{0}=x$, and $\gamma_{1},\,\dots,\,\gamma_{s} \in \DJp$ with 
$\gamma_{1} \lhd \cdots \lhd \gamma_{s}$, where 
$\Delta^{+}$ is the set of positive roots, 
$\DJs:= \Delta^{+} \cap \sum_{i \in \J} \BZ \alpha_{i}$, and 
$\lhd$ is a reflection (convex) order on $\Delta^{+}$ 
satisfying the condition that $\beta \lhd \gamma$ for all 
$\beta \in \DJs$ and $\gamma \in \DJp$. 
Also, we set $\ed(\bp):=y_{s} \in \WJ$ 
if $\bp \in \bBG{x}$ is of the form above, and 
$\ed(\bBG{x}):=\bigl\{ \ed(\bp) \mid 
\bp \in \bBG{x} \bigr\}$;
note that the set $\ed(\bBG{x})$ does not depend on the choice of a reflection order $\lhd$ above (see Lemma~\ref{lem:bBGy}).
We denote by $\theta \in \Delta^{+}$ the highest root.
%
%%%%%%%%%%%%%%%%%%%%%
%%% thm:introThm1 %%%
%%%%%%%%%%%%%%%%%%%%%
%
\begin{thmintro} \label{thm:introThm1}
Assume that $\Fg = \Lie(G)$ is a simple Lie algebra 
of type $A$, $D$, $E$, or $B$. Let $k \in I$ be such that 
$\vpi_{k}$ is a minuscule fundamental weight, 
and set $\J=I \setminus \{k\}$. Let $x \in \WJ$.
Then the following hold in $QK_{T}(G/\PrJ)$:
\begin{enu}
\item If $x \ge \mcr{s_{\theta}}$, then 
\begin{equation} \label{eq:introThm_Chevalley1}
\begin{split}
& [\CO_{\Bv{x}}] \star [\CO_{G/\PrJ}(- \vpi_{k})] = \\
& \be^{x\vpi_{k}}
  \sum_{ y \in \ed(\bBG{x}) } 
  (-1)^{\ell(y) - \ell(x)}[ \CO_{\Bv{y}} ] +
  \be^{x\vpi_{k}}
  \sum_{ y \in \ed(\bBG{x}) }
  (-1)^{\ell(y) - \ell(x) + 1} 
  [ \CO_{\Bv{ \mcr{ ys_{\gq} } }} ]
  Q_{k}.
\end{split}
\end{equation}
Moreover, in the second sum on the right-hand side of \eqref{eq:introThm_Chevalley1}, no cancellations occur.

\item If $x \not \ge \mcr{s_{\theta}}$, then 
\begin{equation} \label{eq:introThm_Chevalley2}
[\CO_{\Bv{x}}] \star [\CO_{G/\PrJ}(- \vpi_{k})] 
= \be^{x\vpi_{k}}
  \sum_{ y \in \ed(\bBG{x}) } 
  (-1)^{\ell(y) - \ell(x)}[ \CO_{\Bv{y}} ]. 
\end{equation}
\end{enu}
Here, for $y \in \WJ$, $[\CO_{ \Bv{y} }]$ denotes 
the opposite Schubert class in $K_{T}(G/\PrJ)$ associated to $y$,
with $\CB_{\J}^{e} = \CB_{\J}:=G/\PrJ$ the partial flag manifold, and 
\begin{equation*}
\gq :=
\begin{cases}
\alpha_{k} & \text{if $\Fg$ is of type $A$, $D$, or $E$}, \\
s_{n} \alpha_{n-1} & \text{if $\Fg$ is of type $B_{n}$ and $k =n$}. 
\end{cases}
\end{equation*}
\end{thmintro}

\begin{rem*}
We know the formula
%
%%%%%%%%%%%%%%%%%%%%
%%% rem:introRem %%%
%%%%%%%%%%%%%%%%%%%%
%
\begin{equation} \label{rem:introRem}
[ \CO_{ \Bv{s_{k}} } ] = 
[ \CO_{G/\PrJ} ] - \be^{-\vpi_{k}}[\CO_{G/\PrJ}(-\vpi_{k})]
\end{equation}
in $K_T(G/\PrJ)$, and hence in $QK_T(G/\PrJ)$; 
this formula is also obtained as the special case 
that $x = e$ of Theorem~\ref{thm:introThm1}.
\end{rem*}

We should mention that a formula for 
the quantum multiplication with $[\CO_{G/\PrJ}(- \vpi_{k})]$ 
in $QK_{T}(G/\PrJ)$ is obtained by \cite{BCMP} 
in the case that $\vpi_{k}$ is a cominuscule fundamental weight; 
in types $A$, $D$, $E$, a fundamental weight $\vpi_{k}$ is cominuscule 
if and only if it is minuscule.
In type $B_{n}$, a formula for the minuscule (but not cominuscule) fundamental
weight $\varpi_{n}$ seems not to be explicitly written in \cite{BCMP};
however, it is obtained from one in \cite{BCMP} for the maximal orthogonal
Grassmannian corresponding to one of the two ``forked'' nodes of the Dynkin diagram of type $D_{n+1}$ (see \cite[Sect.~3.5]{IMN} for details).
However, in their description and the proof of the formula, 
we can hardly see the relation with the quantum Bruhat graph 
introduced in \cite{BFP}, while in ours it is transparent.

The proof of our formula is based on the $\BC[P]$-module isomorphism,
established in \cite{K3}, from the $T$-equivariant quantum 
$K$-group $QK_{T}(G/\PrJ)$ onto the $T$-equivariant $K$-group 
$K_{T}^{\prime}(\bQGJ)$ of the semi-infinite partial flag manifold 
$\bQGJ$ corresponding to $\vpi_{k}$; the $T$-equivariant $K$-group
$K_{T}^{\prime}(\bQGJ)$ has a $\BC[P]$-basis of the semi-infinite Schubert classes $[\CO_{\bQGJ(yt_{\beta^{\vee}})}]$ for $y \in \WJ$ and $\beta^{\vee} \in \BZ_{\geq 0} \alpha_{k}^{\vee}$.
This $\BC[P]$-module isomorphism sends 
the (opposite) Schubert class $[\CO_{\Bv{y}}] Q_{k}$ in $QK_{T}(G/\PrJ)$ to 
the semi-infinite Schubert class $[ \CO_{\bQGJ(yt_{\alpha_{k}^{\vee}})} ]$
in $K_{T}^{\prime}(\bQGJ)$ for each $y \in \WJ$, 
and also respects the quantum multiplication $\star$ 
with the class of the line bundle $[\CO_{G/\PrJ}(- \vpi_{k})]$ 
and the tensor product (denoted by $\Xi(- \vpi_{k})$) 
with the class of the line bundle $[\CO_{\bQGJ}(- \vpi_{k})]$.
Namely, the following diagram commutes:
\begin{equation*}
\begin{CD}
QK_{T}(G/\PrJ) @>{\simeq}>> K_{T}^{\prime}(\bQGJ) \\
@V{\cdot\,\star[\CO_{G/\PrJ}(- \vpi_{k})]}VV
@VV{\Xi(- \vpi_{k})}V \\
QK_{T}(G/\PrJ) @>>{\simeq}> K_{T}^{\prime}(\bQGJ).
\end{CD}
\end{equation*}
By this $\BC[P]$-module isomorphism, the proof of Theorem~\ref{thm:introThm1} is reduced to the proof of a Chevalley formula 
(Theorem~\ref{thm:introThm2} below) for $- \vpi_{k}$ 
in a $T \times \BC^*$-equivariant $K$-group 
$K_{T \times \BC^*}^{\prime}(\bQGJ)$ of 
$\bQGJ$; it suffices to specialize this formula at $q = 1$.
%
%%%%%%%%%%%%%%%%%%%%%
%%% thm:introThm2 %%%
%%%%%%%%%%%%%%%%%%%%%
%
\begin{thmintro} \label{thm:introThm2}
Assume that $\Fg = \Lie(G)$ is a simple Lie algebra of type $A$, $D$, $E$, or $B$. 
Let $k \in I$ be such that $\vpi_{k}$ is a minuscule fundamental weight, and set $J=I \setminus \{k\}$. 
Then, for $x \in \WJ$, the following hold 
in $K_{T \times \BC^*}^{\prime}(\bQGJ)$ 
(and hence in $K_{T}^{\prime}(\bQGJ)$):
\begin{enu}
\item If $x \ge \mcr{s_{\theta}}$, then 
\begin{align}
 [\CO_{\bQGJ(x)} \otimes \CO_{\bQGJ}(- \vpi_{k})] 
  & := \Xi(-\vpi_{k}) ([\CO_{\bQGJ(x)}]) \nonumber \\
& = \be^{-x\vpi_{k}}
  \sum_{ y \in \ed(\bBG{x}) } 
  (-1)^{\ell(y) - \ell(x)}
  [\CO_{\bQGJ(y)}] \nonumber \\
& \quad + 
  \be^{-x\vpi_{k}}
  \sum_{ y \in \ed(\bBG{x}) } 
  (-1)^{\ell(y) - \ell(x) + 1}
  [\CO_{ \bQGJ(\mcr{ ys_{\gq} } t_{\alpha_{k}^{\vee}}) }]. \label{eq:intro2a}
\end{align}
Moreover, in the second sum on the right-hand side of \eqref{eq:intro2a}, no cancellations occur.

\item If $x \not\ge \mcr{s_{\theta}}$, then 
\begin{align}
[\CO_{\bQGJ(x)} \otimes \CO_{\bQGJ}(- \vpi_{k})] 
 & := \Xi(-\vpi_{k}) ([\CO_{\bQGJ(x)}]) \nonumber \\
 & = \be^{-x\vpi_{k}}
  \sum_{ y \in \ed(\bBG{x}) } 
  (-1)^{\ell(y) - \ell(x)}
  [\CO_{\bQGJ(y)}]. \label{eq:intro2b}
\end{align}
\end{enu}
\end{thmintro}

\begin{rem*}
As the special case that $x = e$, we obtain the formula
$[\CO_{\bQGJ}(-\vpi_{k})] = 
 \be^{-\vpi_{k}} ([\CO_{\bQGJ}] - [\CO_{\bQGJ(s_{k})}])$,
or equivalently,
$[\CO_{\bQGJ(s_{k})}] = [\CO_{\bQGJ}] - \be^{\vpi_{k}} [\CO_{\bQGJ}(-\vpi_{k})]$,
where $[\CO_{\bQGJ}]$ can be thought of 
as the identity element of $K'_{T \times \BC^{\ast}}(\bQGJ)$ 
(or $K'_T(\bQGJ)$) with respect to the tensor product. 
\end{rem*}

Now, we explain how to prove our results above.
Recall that $\vpi_{k}$ is minuscule, and $\J = I \setminus \{k\}$.
Let $\bQGJr$ denote the (whole) the semi-infinite partial flag manifold, 
which is the reduced ind-scheme whose set of $\BC$-valued points is 
$G(\BC\pra{z})/(T \cdot [\PrJ,\,\PrJ](\BC\pra{z}))$ (see \cite{K2} for details); in this paper, we concentrate on the semi-infinite Schubert (sub)variety 
$\bQGJ := \bQGJ(e) \subset \bQGJr$ associated to the identity element $e \in \WJ$, which we also call the semi-infinite partial flag manifold.
Following \cite{K3}, we define a $T \times \BC^*$-equivariant 
$K$-group $K_{T \times \BC^*}^{\prime}(\bQGJ)$ of $\bQGJ$ to be  
the $\BC[q, q^{-1}][P]$-module consisting of all finite $\BC[q, q^{-1}][P]$-linear 
combinations of the semi-infinite Schubert classes $[\CO_{\bQGJ(x)}]$ 
for $x = vt_{\beta^{\vee}} \in W_{\af}$, 
with $v \in \WJ$ and $\beta^{\vee} \in \BZ_{\geq 0} \alpha_{k}^{\vee}$;
the semi-infinite Schubert classes $[\CO_{\bQGJ(x)}]$ for $x = vt_{\beta^{\vee}} \in W_{\af}$, 
with $v \in \WJ$ and $\beta^{\vee} \in \BZ_{\geq 0} \alpha_{k}^{\vee}$, 
turn out to form a $\BC[q, q^{-1}][P]$-basis of $K_{T \times \BC^*}^{\prime}(\bQGJ)$.
Also, let $\Fun_{\BZ \vpi_{k}}(\BC\pra{q^{-1}}[P])$ denote 
the $\BC[q, q^{-1}][P]$-module of all functions on $\BZ \vpi_{k}$ 
with values in $\BC\pra{q^{-1}}[P]$, and set
\begin{equation*}
\Fun_{\BZ \vpi_{k}}^{\ess} (\BC\pra{q^{-1}}[P]) := 
\Fun_{\BZ \vpi_{k}}(\BC\pra{q^{-1}}[P])/\Fun_{\BZ \vpi_{k}}^{\ngt}(\BC\pra{q^{-1}}[P]),
\end{equation*}
where $\Fun_{\BZ \vpi_{k}}^{\ngt}(\BC\pra{q^{-1}}[P])$ is 
the $\BC[q, q^{-1}][P]$-submodule of $\Fun_{\BZ \vpi_{k}}(\BC\pra{q^{-1}}[P])$ 
consisting of those $f \in \Fun_{\BZ \vpi_{k}} (\BC\pra{q^{-1}}[P])$ 
such that there exists some $\gamma \in \BZ \vpi_{k}$ 
for which $f(\mu) = 0$ for all $\mu \in \gamma + \BZ_{\geq 0} \vpi_{k}$.
Then, for each $x = vt_{\beta^{\vee}} \in W_{\af}$, 
with $v \in \WJ$ and $\beta^{\vee} \in \BZ_{\geq 0} \alpha_{k}^{\vee}$, 
the assignment
\begin{equation*}
\BZ \vpi_{k} \owns \mu \mapsto 
\gch H^{0}(\bQGJ,\,\CO_{\bQGJ(x)} \otimes \CO_{\bQGJ}(\mu)) \in \BC\pra{q^{-1}}[P]
\end{equation*}
defines an element of $\Fun_{\BZ \vpi_{k}}^{\ess}(\BC\pra{q^{-1}}[P])$, 
which we denote by $f^{x}(\sdot)$; here, we denote by 
$\gch H^{0}(\bQGJ, \,\CO_{\bQGJ(x)} \otimes \CO_{\bQGJ}(\mu))$
the graded character of the $T \times \BC^*$-module 
$H^{0}(\bQGJ,\,\CO_{\bQGJ(x)} \otimes \CO_{\bQGJ}(\mu))$, 
which is identical to the graded character of 
the Demazure submodule $V_{x}^{-}(\mu)$ of 
the level-zero extremal weight module $V(\mu)$ 
over the quantum affine algebra $U_{\q}(\Fg_{\af})$ 
if $\mu \in \BZ_{\geq 0} \vpi_{k}$, and is zero 
if $\mu \notin \BZ_{\geq 0} \vpi_{k}$ (see \cite{K2} and also \cite{K3} for details),
where $\Fg_{\af}$ is the (untwisted) affine Lie algebra whose underlying simple Lie algebra is $\Fg$.
Here we warn the reader that the line bundles $\CO_{\bQGJ}(\mu)$ 
associated to $\mu \in \BZ \vpi_{k}$ are normalized (as in \cite{K1}) in such a way that 
$\gch H^{0}(\bQGJ, \CO_{\bQGJ}(\mu)) = \gch V_{e}^{-}(\mu)$ holds for $\mu \in \BZ_{\ge 0} \vpi_{k}$; 
this convention differs from that of \cite{KNS} 
by the twist coming from the involution $- w_{\circ}$.
Thus we obtain a $\BC[q,q^{-1}][P]$-linear map:
\begin{equation*}
\Phi : K_{T \times \BC^*}^{\prime}(\bQGJ) \rightarrow 
\Fun_{\BZ \vpi_{k}}^{\ess} (\BC\pra{q^{-1}}[P])
\end{equation*}
given by $\Phi([\CO_{\bQGJ(x)}]) = f^{x}(\sdot)$ 
for each $x = vt_{\beta^{\vee}} \in W_{\af}$, 
with $v \in \WJ$ and $\beta^{\vee} \in \BZ_{\geq 0} \alpha_{k}^{\vee}$, 
which is injective since the graded characters $\gch V_{vt_{ \beta^{\vee} }}^{-}(\mu)$, 
$v \in \WJ$ and $\beta^{\vee} \in \BZ_{\ge 0}\alpha_{k}^{\vee}$, are linearly independent 
over $\BC[q,q^{-1}][P]$ when they are regarded as functions of sufficiently large 
$\mu \in \BZ_{\ge 1}\vpi_{k}$ (see \cite{K3}).

From the explicit identities obtained in \cite{NOS} 
(in the case of anti-dominant weights) 
for the graded characters of Demazure submodules of 
level-zero extremal weight modules, it can be shown (see \cite{K3}) 
that there exist $\BC[q, q^{-1}][P]$-module endomorphisms 
$\Xi(-\lambda)$, $\lambda \in \BZ_{\ge 1} \vpi_{k}$, of 
$K_{T \times \BC^*}^{\prime}(\bQGJ)$ such that 
$\Xi(-(\lambda + \lambda^{\prime})) = 
\Xi(-\lambda) \circ \Xi(-\lambda^{\prime})$ 
for $\lambda,\,\lambda^{\prime} \in \BZ_{\ge 1} \vpi_{k}$ and 
such that the following diagram commutes for all $\lambda \in \BZ_{\ge 1}\vpi_{k}$:
\begin{equation*}
\begin{CD}
K_{T \times \BC^*}^{\prime}(\bQGJ) @>\Phi>> 
\Fun_{\BZ \vpi_{k}}^{\ess} (\BC\pra{q^{-1}}[P]) \\
@V{\Xi(-\lambda)}VV @VV{\Theta(-\lambda)}V \\
K_{T \times \BC^*}^{\prime}(\bQGJ) @>>\Phi> 
\Fun_{\BZ \vpi_{k}}^{\ess} (\BC\pra{q^{-1}}[P]),
\end{CD}
\end{equation*}
where $\Theta(-\lambda) f(\sdot) = f(\sdot - \lambda)$ for 
$f(\sdot) \in \Fun_{\BZ \vpi_{k}}^{\ess}(\BC\pra{q^{-1}}[P])$;
the $\BC[q,q^{-1}][P]$-module endomorphism $\Xi(-\lambda)$ 
can be thought of as the tensor product with the class of 
the line bundle $[\CO_{\bQGJ}(-\lambda)]$ in $K_{T \times \BC^*}^{\prime}(\bQGJ)$.
In view of the commutativity of the diagram above and 
the injectivity of the $\BC[q, q^{-1}][P]$-linear map $\Phi$, 
the proof of our Chevalley formula (Theorem~\ref{thm:introThm2}) 
for $- \vpi_{k}$ in $K_{T \times \BC^*}^{\prime}(\bQGJ)$ is 
reduced to the proof of the corresponding identity of 
Chevalley type (Theorem~\ref{thm:introThm3} below) 
for the graded characters of Demazure submodules of 
level-zero extremal weight modules over 
the quantum affine algebra $U_{\q}(\Fg_{\af})$; 
we derive this identity 
from the results in \cite{NOS} 
through a detailed analysis of the quantum Bruhat graph. 
Indeed, since the left-hand side of \eqref{eq:intro2a} or \eqref{eq:intro2b}
is $\Xi(-\vpi_{k}) ([\CO_{\bQGJ(x)}])$, its image under $\Phi$ is 
identical to $\Theta(-\vpi_{k})(\Phi([\CO_{\bQGJ(x)}]))$ by the 
commutativity of the diagram above; by the definitions, this is 
identical to the graded character $\gch V_{x}^{-}(\mu-\vpi_{k})$ 
(regarded as a function of $\mu \in \BZ_{\ge 1}\vpi_{k}$), 
which is just the left-hand side of \eqref{eq:intro3a} or \eqref{eq:intro3b} below. 
Also, the image under $\Phi$ of the right-hand side of \eqref{eq:intro2a}
(resp., \eqref{eq:intro2b}) is, by the definitions, identical to the right-hand 
side of \eqref{eq:intro3a} (resp., \eqref{eq:intro3b}) 
(regarded as a function of $\mu \in \BZ_{\ge 1}\vpi_{k}$). 
Because these two functions of $\mu \in \BZ_{\ge 1}\vpi_{k}$ coincide 
by Theorem~\ref{thm:introThm3} below, we deduce \eqref{eq:intro2a} 
(resp., \eqref{eq:intro2b}) from the injectivity of 
the $\BC[q,q^{-1}][P]$-linear map $\Phi$. 
%
%%%%%%%%%%%%%%%%%%%%%
%%% thm:introThm3 %%%
%%%%%%%%%%%%%%%%%%%%%
%
\begin{thmintro}[= Theorem \ref{thm:main}] \label{thm:introThm3}
Assume that $\Fg$ is a simple Lie algebra of type $A$, $D$, $E$, or $B$. 
Let $k \in I$ be such that $\vpi_{k}$ is a minuscule fundamental weight, 
and set $\mu := N \vpi_{k}$, with $N \in \BZ_{\ge 1}$. 
Then, for $x \in \WJ$ with $J = I \setminus \{k\}$, the following identities 
for the graded characters of Demazure submodules of 
level-zero extremal weight modules hold:
\begin{enu}
\item If $x \ge \mcr{ s_{\theta} }$, then 
\begin{align}
\gch V_{x}^{-}(\mu-\vpi_{k}) & = \be^{-x\vpi_{k}} 
\sum_{y \in \ed(\bBG{x})} (-1)^{\ell(y)-\ell(x)} \gch V_{y}^{-}(\mu) \nonumber \\[3mm]
& \quad + 
\be^{-x\vpi_{k}} 
\sum_{y \in \ed(\bBG{x})} (-1)^{\ell(y)-\ell(x)+1} 
\gch V_{ \mcr{ys_{\gq} } t_{\alpha_{k}^{\vee}} }^{-}(\mu). \label{eq:intro3a}
\end{align}
Moreover, in the second sum on the right-hand side of \eqref{eq:intro3a}, no cancellations occur.

\item If $x \not\ge \mcr{ s_{\theta} }$, then
\begin{equation} \label{eq:intro3b}
\gch V_{x}^{-}(\mu-\vpi_{k}) = \be^{-x\vpi_{k}} 
\sum_{y \in \ed(\bBG{x})} (-1)^{\ell(y)-\ell(x)} \gch V_{y}^{-}(\mu). 
\end{equation}
\end{enu}
\end{thmintro}

This paper is organized as follows. In Section~\ref{sec:not}, 
we first fix basic notation used throughout this paper. Then
we recall some basic facts about the quantum Bruhat graph and 
quantum Lakshmibai-Seshadri paths. Also, we review a character identity 
of Chevalley type in \cite{NOS}, from which a Chevalley formula for anti-dominant weights follows. 
In Section~\ref{sec:main}, we restate Theorem~\ref{thm:introThm3} 
above as Theorem~\ref{thm:main}. Also, we show Theorem~\ref{thm:introThm3} 
in the case that $x = e$. 
In Section~\ref{sec:pre}, we first show some lemmas 
on quantum Lakshmibai-Seshadri paths of shape $\vpi_{k}$, a minuscule fundamental weight, 
and Bruhat or quantum edges in the quantum Bruhat graph. 
Then, after reviewing some basic facts about Demazure operators, 
we show recurrence relations for coefficients in the character identity 
of Chevalley type for anti-dominant minuscule fundamental weights, 
which are needed in an (inductive) argument in the proof of Theorem~\ref{thm:introThm3}.
In Section~\ref{sec:prfa}, we prove Theorem~\ref{thm:introThm3} (with $x \ne e$)
in the case that $\Fg = \Lie(G)$ is of type $A$, $D$, or $E$. 
In Section~\ref{sec:prfb}, we prove Theorem~\ref{thm:introThm3} (with $x \ne e$)
in the case that $\Fg = \Lie(G)$ is of type $B$.
In Appendix~\ref{sec:ex}, we give an example of Theorem~\ref{thm:introThm1} 
in type $A_{6}$. 
%
%
%=========================%
%     START SECTION 02    %
%=========================%
%
\section{Character identity of Chevalley type for general anti-dominant weights.}
\label{sec:not}
%
%==============================%
%     START SUBSECTION 0201    %
%==============================%
%
\subsection{Basic notation.}
\label{subsec:lie}
Let $\Fg$ be an (arbitrary) finite-dimensional simple Lie algebra over $\BC$ 
with Cartan subalgebra $\Fh$; we denote by 
$\pair{\cdot\,}{\cdot}:\Fh^{\ast} \times \Fh \rightarrow \BC$
the canonical pairing of $\Fh^{\ast}:=\Hom_{\BC}(\Fh,\,\BC)$ and $\Fh$. 
Denote by $\{ \alpha_{i}^{\vee} \}_{i \in I} \subset \Fh$ and 
$\{ \alpha_{i} \}_{i \in I} \subset \Fh^{\ast}$ 
the set of simple coroots and simple roots of $\Fg$, respectively, 
and set $Q := \sum_{i \in I} \BZ \alpha_{i}$, 
$Q^{\vee} := \sum_{i \in I} \BZ \alpha_{i}^{\vee}$. 
Let $\Delta$, $\Delta^{+}$, and $\Delta^{-}$ be 
the set  of roots, positive roots, and negative roots of $\Fg$, respectively, 
and denote by $\Delta^{+}_{\lo}$ and $\Delta^{+}_{\sh}$ the set of 
positive long roots and positive short roots of $\Fg$, respectively; 
if $\Fg$ is simply-laced, then $\Delta^{+}_{\lo}=\Delta^{+}$ and 
$\Delta^{+}_{\sh}=\emptyset$ by our convention. 
Let $\theta \in \Delta^{+}$ denote the highest root of $\Fg$; 
recall that $\theta \in \Delta^{+}_{\lo}$. 
We set $\rho:=(1/2) \sum_{\alpha \in \Delta^{+}} \alpha$. 
Also, let $\vpi_{i}$, $i \in I$, denote the fundamental weights for $\Fg$, and set
%
%%%%%%%%%%%%%%%%
%%% eq:P-fin %%%
%%%%%%%%%%%%%%%%
%
\begin{equation} \label{eq:P-fin}
P:=\sum_{i \in I} \BZ \vpi_{i} \qquad \text{and} \qquad 
P^{+} := \sum_{i \in I} \BZ_{\ge 0} \vpi_{i}. 
\end{equation} 
Let $W := \langle s_{i} \mid i \in I \rangle$ be the (finite) Weyl group of $\Fg$, 
where $s_{i}$ is the simple reflection with respect to $\alpha_{i}$ for $i \in I$, 
with $e \in W$ the identity element and $\lng \in W$ the longest element. 
Let us denote by $\ge$ the Bruhat order on $W$, 
and by $\ell:W \rightarrow \BZ_{\ge 0}$ 
the length function on $W$.
For $x \in W$, we set $\Inv(x):=\Delta^{+} \cap x^{-1}\Delta^{-}= 
\bigl\{ \alpha \in \Delta^{+} \mid x\alpha \in \Delta^{-} \bigr\}$; 
recall that $\ell(x) = \# \Inv(x)$. 
For $\beta \in \Delta$, we denote by $\beta^{\vee} \in \Fh$ 
its dual root, and by $s_{\beta} \in W$ the corresponding reflection; 
remark that for $\beta,\gamma \in \Delta^{+}$, 
$s_{\beta}=s_{\gamma}$ if and only if $\beta = \gamma$. 
Note that 
%
%%%%%%%%%%%%%
%%% eq:qr %%%
%%%%%%%%%%%%%
%
\begin{equation} \label{eq:qr}
\ell(s_{\beta}) \le 2\pair{\rho}{\beta^{\vee}} -1 \quad 
\text{for all $\beta \in \Delta^{+}$};
\end{equation}
if the equality holds in \eqref{eq:qr}, then $\beta$ is called 
a (positive) quantum root. Denote by $\Delta^{+}_{\qtm}$ the set of 
quantum roots. 
%
%%%%%%%%%%%%%%
%%% lem:qr %%%
%%%%%%%%%%%%%%
%
\begin{lem}[{\cite[Lemma 7.2]{BMO}}] \label{lem:qr}
A positive root $\beta \in \Delta^{+}$ is a quantum root if and only if 
either of the following holds\,{\rm:} 
{\rm (a)} $\beta \in \Delta^{+}_{\lo}$\,{\rm;}
{\rm (b)} $\beta \in \Delta^{+}_{\sh}$, and $\beta$ is a $\BZ$-linear combination of
short simple roots. 
\end{lem}

Let $\J \subset I$ be a subset of $I$. We set 
\begin{equation*}
\DJs:=\Delta^{+} \cap \sum_{i \in \J} \BZ \alpha_{i}, \qquad 
\rho_{\J}:=\frac{1}{2} \sum_{\alpha \in \DJs} \alpha, \qquad 
\WJs:=\langle s_{i} \mid i \in \J \rangle.
\end{equation*}
Denote by $\ls \in \WJs$ the longest element of $\WJ$. 
For $\xi=\sum_{i \in I} c_{i}\alpha_{i}^{\vee} \in Q^{\vee}$, 
we set $[\xi]=[\xi]^{\J}:=\sum_{i \in I \setminus \J} c_{i}\alpha_{i}^{\vee}$. 
Let $\WJ$ denote the set of minimal(-length) coset representatives 
for the cosets in $W/\WJs$; we know from \cite[Sect.~2.4]{BB} that 
%
%%%%%%%%%%%%%%
%%% eq:mcr %%%
%%%%%%%%%%%%%%
%
\begin{equation} \label{eq:mcr}
\WJ = \bigl\{ w \in W \mid 
\text{$w \alpha \in \Delta^{+}$ for all $\alpha \in \DJs$}\bigr\}.
\end{equation}
For $w \in W$, we denote by $\mcr{w}=\mcr{w}^{\J} \in \WJ$ 
the minimal coset representative for the coset $w \WJs$ in $W/\WJs$; 
note that $\Inv(\mcr{\lng}) = \DJp$. The following lemma is well-known. 
%
%%%%%%%%%%%%%%
%%% lem:WJ %%%
%%%%%%%%%%%%%%
%
\begin{lem} \label{lem:WJ}
Let $\Lambda \in P^{+}$ be such that 
$\J_{\Lambda}:=\bigl\{i \in I \mid \pair{\Lambda}{\alpha_{i}^{\vee}}=0\bigr\}$ 
is identical to $\J$. Let $w \in \WJ$ and $j \in I$. 
\begin{enu}
\item If $\pair{w\Lambda}{\alpha_{j}^{\vee}} > 0$, 
then $w^{-1}\alpha_{j} \in \DJp$. In this case, 
$s_{j}w \in \WJ$, and $s_{j}w > w$. 

\item If $\pair{w\Lambda}{\alpha_{j}^{\vee}} < 0$, 
then $- w^{-1}\alpha_{j} \in \DJp$. In this case, 
$s_{j}w \in \WJ$, and $s_{j}w < w$. 

\item If $\pair{w\Lambda}{\alpha_{j}^{\vee}} = 0$, 
then $w^{-1}\alpha_{j} \in \DJs$. In this case, 
$s_{j}w = ws_{p} \in \WJ$ for some $p \in \J$, and 
$\mcr{s_{j}w} = w$. 
\end{enu}
\end{lem}
%
%===============================%
%     START SUBSECTION 0202    %
%===============================%
%
\subsection{Reflection orders.}
\label{subsec:ro}

In this subsection, we review some basic facts about
reflection orders on $\Delta^{+}$; for details, 
see \cite{Dy}. 
%
%%%%%%%%%%%%%%
%%% dfn:ro %%%
%%%%%%%%%%%%%%
%
\begin{dfn} \label{dfn:ro}
A total order $\lhd$ on $\Delta^{+}$ is called a reflection (convex) order 
if for each $\alpha,\,\beta \in \Delta^{+}$ such that 
$\alpha + \beta \in \Delta^{+}$, either 
$\alpha \lhd \alpha + \beta \lhd \beta$ or 
$\beta \lhd \alpha + \beta \lhd \alpha$ holds.
\end{dfn} 

Let $\lng = s_{j_{p}}s_{j_{p-1}} \cdots s_{j_{2}}s_{j_{1}}$ be 
a reduced expression of the longest element $\lng$ of $W$. If we set 
\begin{equation*}
\beta_{q}:= s_{j_{1}} \cdots s_{j_{q-1}}\alpha_{j_{q}} 
  \quad \text{for $1 \le q \le p$},
\end{equation*}
then $\Delta^{+} = \bigl\{ \beta_{q} \mid 1 \le q \le p \bigr\}$. 
Moreover, if we define a total order $\lhd$ by 
$\beta_{p} \lhd \cdots \lhd \beta_{2} \lhd \beta_{1}$, 
then $\lhd$ is a reflection order on $\Delta^{+}$. 
Thus we have a map from the set of reduced expressions of $\lng$ 
to the set of reflection orders on $\Delta^{+}$; 
in fact, this map is bijective (see \cite[(2.13) Proposition]{Dy}, 
and also \cite[Theorem on page 662 and Corollary on page 663]{Pa}). 

Let $w \in W$. Then there exists $v \in W$ such that 
$\lng = vw$ and $\ell(\lng) = \ell(v) + \ell(w)$. 
The set of reflection orders $\lhd$ on $\Delta^{+}$ 
satisfying the condition that 
$\beta \lhd \gamma$ for all $\beta \in \Delta^{+} \setminus \Inv(w)$ 
and $\gamma \in \Inv(w)$ is in bijection with 
the set of reduced expressions of $\lng$ of the form 
\begin{equation*}
\lng = 
\underbrace{ s_{j_{p}}s_{j_{p-1}} \cdots s_{j_{a+2}}s_{j_{a+1}} }_{=v}
\underbrace{ s_{j_{a}}s_{j_{a-1}} \cdots s_{j_{2}}s_{j_{1}} }_{=w}; 
\end{equation*}
note that $\Inv(w)=
\bigl\{ s_{j_{1}} \cdots s_{j_{q-1}}\alpha_{j_{q}} \mid 1 \le q \le a\bigr\}$. 
Similarly, if $\lng = vw_{2}w_{1}$, with $v,\,w_{2},\,w_{1} \in W$, and 
$\ell(\lng) = \ell(v) + \ell(w_{2}) + \ell(w_{1})$, then 
$\Inv(w_{1}) \subset \Inv (w_{2}w_{1})$, and 
the set of reflection orders $\lhd$ on $\Delta^{+}$ 
satisfying the condition that 
$\beta \lhd \gamma_{1} \lhd \gamma_{2}$ 
for all $\beta \in \Delta^{+} \setminus \Inv(w_{2}w_{1})$, 
$\gamma_{1} \in \Inv(w_{2}w_{1}) \setminus \Inv(w_{1})$, and 
$\gamma_{2} \in \Inv(w_{1})$ is in bijection with 
the set of reduced expressions of $\lng$ of the form 
\begin{equation*}
\lng = 
\underbrace{ s_{j_{p}}s_{j_{p-1}} \cdots s_{j_{a+2}}s_{j_{a+1}} }_{=v}
\underbrace{ s_{j_{a}}s_{j_{a-1}} \cdots s_{j_{t+2}}s_{j_{b+1}} }_{=w_{2}}
\underbrace{ s_{j_{b}}s_{j_{t-1}} \cdots s_{j_{2}}s_{j_{1}} }_{=w_{1}}. 
\end{equation*}
%
%==============================%
%     START SUBSECTION 0203    %
%==============================%
%
\subsection{Quantum Bruhat graph.}
\label{subsec:QBG}
%
%%%%%%%%%%%%%%%
%%% dfn:QBG %%%
%%%%%%%%%%%%%%%
%
\begin{dfn}[{\cite[Definition 6.1]{BFP}}] \label{dfn:QBG}
The quantum Bruhat graph, denoted by $\QBG(W)$, is 
the $\Delta^{+}$-labeled
directed graph whose vertices are the elements of $W$, and 
whose directed edges are of the form: $w \edge{\beta} v$ 
for $w,v \in W$ and $\beta \in \Delta^{+}$ such that 
$v= ws_{\beta}$, and such that either of the following holds: 
(B)~$\ell(v) = \ell (w) + 1$; 
(Q)~$\ell(v) = \ell (w) + 1 - 2 \pair{\rho}{\beta^{\vee}}$.
An edge satisfying (B) (resp., (Q)) is called a Bruhat (resp., quantum) edge. 
The Bruhat graph, denoted by $\BG(W)$, is the 
$\Delta^{+}$-labeled directed graph obtained 
from $\QBG(W)$ by removing all quantum edges. 
\end{dfn}
%
%%%%%%%%%%%%%%%
%%% rem:qe1 %%%
%%%%%%%%%%%%%%%
%
\begin{rem} \label{rem:qe1}
For $w \in W$ and $\beta \in \Delta^{+}$, we see that 
$\ell(ws_{\beta}) \ge \ell(w) - \ell(s_{\beta}) \ge 
\ell(w) + 1 - 2 \pair{\rho}{\beta^{\vee}}$. 
Hence, if $w \edge{\beta} v$ is a quantum edge in $\QBG(W)$, 
then $\beta$ is a quantum root. 
Moreover, if $s_{\beta}=s_{j_{1}}s_{j_{2}} \cdots s_{j_{r}}$ is a reduced 
expression of $s_{\beta}$ (note that $r=2\pair{\rho}{\beta^{\vee}}-1$), then 
$\ell(ws_{j_{1}}s_{j_{2}} \cdots s_{j_{t}}) = 
\ell(w) - t$ for all $0 \le t \le r$.
\end{rem}

Let $\bp:y_{0} \edge{\beta_{1}} y_{1} \edge{\beta_{2}} \cdots \edge{\beta_{r}} y_{r}$ 
be a directed path in $\QBG(W)$. We set $\bgn(\bp):=y_{0}$ and $\ed(\bp)=y_{r}$. 
Also, we define the length $\ell(\bp)$ and the weight $\wt(\bp)$ of $\bp$ by 
\begin{equation*}
\ell(\bp):=r \qquad \text{and} \qquad 
\wt(\bp):=\sum_{ 
  \begin{subarray}{c}
  1 \le u \le r \\[1mm]
  \text{$y_{u-1} \edge{\beta_{u}} y_{u}$ is a quantum edge}
  \end{subarray} } \beta_{u}^{\vee}. 
\end{equation*}
For $x,y \in W$, we define $\wt(x \Rightarrow y)$ and $\ell(x \Rightarrow y)$ 
to be the weight $\wt(\bp)$ and the length $\ell(\bp)$ of 
a shortest directed path $\bp$ from $x$ to $y$ in $\QBG(W)$, respectively; 
we know that $\wt(x \Rightarrow y)$ does not depend on the choice of 
a shortest directed path $\bp$ (see, e.g., \cite[Sect.~4.1]{LNSSS2}). 
%
%%%%%%%%%%%%%%%%%%%%
%%% rem:shortest %%%
%%%%%%%%%%%%%%%%%%%%
%
\begin{rem} \label{rem:shortest}
Let $x,y \in W$. We see that 
$y \ge x$ in the Bruhat order if and only if 
all the edges in a shortest directed path $\bp$ from $x$ to $y$ 
are Bruhat edges, that is, if and only if $\bp$ is a directed path in $\BG(W)$. 
In this case, $\ell(x \Rightarrow y) = \ell(\bp) = \ell(y)-\ell(x)$, 
and $\wt(x \Rightarrow y) = \wt (\bp) = 0$. 
\end{rem}

%%%%%%%%
\iffalse
%%%%%%%%
%
%%%%%%%%%%%%%%%%%
%%% lem:coset %%%
%%%%%%%%%%%%%%%%%
%
\begin{lem} \label{lem:coset}
Let $\J$ be a subset of $I$. 
Let $w \in \WJ$, and $x,\,y \in w\WJs$. 
All the labels of edges in a shortest directed path 
from $x$ to $y$ in $\QBG(W)$ are contained in $\DJs$. 
\end{lem}
%
%%%
\fi
%%%

Let $\lhd$ be an arbitrary reflection (convex) order on $\Delta^{+}$ (see Section~\ref{subsec:ro}). 
A directed path $y_{0} \edge{\beta_{1}} y_{1} \edge{\beta_{2}} \cdots \edge{\beta_{r}} y_{r}$ 
in $\QBG(W)$ is said to be label-increasing (with respect to $\lhd$) 
if $\beta_{1} \lhd \beta_{2} \lhd \cdots \lhd \beta_{r}$. 
We know the following from \cite{BFP} 
(see also \cite[Theorem~7.3]{LNSSS1}). 
%
%%%%%%%%%%%%%%
%%% thm:LI %%%
%%%%%%%%%%%%%%
%
\begin{thm} \label{thm:LI}
For all $x,y \in W$, there exists a unique label-increasing directed path 
$\bp:x = y_{0} \edge{\beta_{1}} y_{1} \edge{\beta_{2}} \cdots \edge{\beta_{r}} y_{r} = y$ 
from $x$ to $y$ in $\QBG(W)$. Moreover, it is a shortest directed path from $x$ to $y$, 
and is lexicographically minimal among all shortest directed paths 
from $x$ to $y$ in the following sense\,{\rm:} 
for each shortest directed path 
$\bq:x = z_{0} \edge{\gamma_{1}} z_{1} \edge{\gamma_{2}} \cdots 
\edge{\gamma_{r}} z_{r} = y$, there exists $1 \le u \le r$ such that 
$\gamma_{t} = \beta_{t}$ for all $1 \le t \le u$ and 
$\gamma_{u+1} \rhd \beta_{u+1}$. 
\end{thm}
%
%%%%%%%%%%%%%%
%%% rem:LI %%%
%%%%%%%%%%%%%%
%
\begin{rem} \label{rem:LI}
Let $x,y \in W$ be such that $y \ge x$ in the Bruhat order. 
By Theorem~\ref{thm:LI} and Remark~\ref{rem:shortest}, 
the (unique) label-increasing directed path from $x$ to $y$ 
in $\QBG(W)$ is a directed path of length $\ell(y)-\ell(x)$ in $\BG(W)$. 
\end{rem}

The next lemma follows from \cite[Corollary 2.5.2]{BB}. 
%
%%%%%%%%%%%%%
%%% lem:B %%%
%%%%%%%%%%%%%
%
\begin{lem} \label{lem:B}
Let $\J$ be a subset of $I$. 
Let $y \in \WJ$ and $\gamma \in \DJp$ be 
such that $y \edge{\gamma} ys_{\gamma}$ is a directed edge in $\QBG(W)$. 
If the edge $y \edge{\gamma} ys_{\gamma}$ is a Bruhat edge, 
then $ys_{\gamma} \in \WJ$. 
\end{lem}

Let $\J$ be a subset of $I$. 
Let $\lhd$ be an arbitrary reflection (convex) order on $\Delta^{+}$ 
satisfying the condition that 
%
%%%%%%%%%%%%%
%%% eq:ro %%%
%%%%%%%%%%%%%
%
\begin{equation} \label{eq:ro}
\beta \lhd \gamma \quad 
\text{for all $\beta \in \DJs$ and $\gamma \in \DJp$}; 
\end{equation}
recall that $\Inv(\mcr{\lng}) = \DJp$ (see Section~\ref{subsec:ro}). 
For each $y \in \WJ$, denote by $\bBG{y}$ (resp., $\bQBG{y}$) the set of 
all label-increasing directed paths $\bp$ 
in the Bruhat graph $\BG(W)$ (resp., in the quantum Bruhat graph $\QBG(W)$) 
such that $\bgn(\bp)=y$, and such that 
all the labels of edges in $\bp$ are contained in $\DJp$: 
%
%%%%%%%%%%%%%%%
%%% eq:bBGy %%%
%%%%%%%%%%%%%%%
%
\begin{equation} \label{eq:bBGy}
\bp: 
\underbrace{y = y_{0} \edge{\beta_{1}} y_{1} \edge{\beta_{2}} \cdots \edge{\beta_{r}} y_{r},}_{
\text{directed path in $\BG(W)$ (resp., $\QBG(W)$)}} 
\quad \text{where} \quad 
\begin{cases}
r \ge 0, \\[1mm]
\text{$\beta_{u} \in \DJp$ for all $1 \le u \le r$}, \\[1mm]
\beta_{1} \lhd \beta_{2} \lhd \cdots \lhd \beta_{r}. 
\end{cases}
\end{equation}
Note that $\bBG{y} \subset \bQBG{y}$. 
%
%%%%%%%%%%%%%%%%
%%% rem:bBGy %%%
%%%%%%%%%%%%%%%%
%
\begin{rem} \label{rem:bBGy}
Keep the notation and setting above. 
\begin{enu}
\item By the uniqueness of a label-increasing directed path in Theorem~\ref{thm:LI}, 
the map $\ed:\bQBG{y} \rightarrow W$, $\bp \mapsto \ed(\bp)$, is injective. 
For a subset $\bB$ of $\bQBG{y}$, we set 
$\ed(\bB):=\bigl\{ \ed(\bp) \mid \bp \in \bB \bigr\}$. 

\item Let $\bp \in \bBG{y}$ be of the form \eqref{eq:bBGy}. 
We see by Lemma~\ref{lem:B} that $y_{u} \in \WJ$ for all $0 \le u \le s$. 
In particular, $\ed(\bp) \in \WJ$, and hence $\ed(\bBG{y}) \subset \WJ$. 
\end{enu}
\end{rem}
%
%%%%%%%%%%%%%%%%
%%% lem:bBGy %%%
%%%%%%%%%%%%%%%%
%
\begin{lem} \label{lem:bBGy}
Keep the notation and setting above. 
Neither $\ed(\bBG{y})$ nor $\ed(\bQBG{y})$ 
depends on the choice of a reflection order $\lhd$ satisfying 
condition \eqref{eq:ro}. Namely, if $\prec$ is also 
a reflection order on $\Delta^{+}$ satisfying condition \eqref{eq:ro}, 
then $\ed(\bBG{y}) = \ed(\mathbf{BG}_{y}^{\prec})$ and 
$\ed(\bQBG{y}) = \ed(\mathbf{QBG}_{y}^{\prec})$. 
\end{lem}

\begin{proof}
Let $\bp \in \bQBG{y}$, and let $\bq$ be the label-increasing directed path 
from $y$ to $w:=\ed(\bp)$ with respect to $\prec$. 
We claim that $\bq \in \mathbf{QBG}_{y}^{\prec}$. 
Recall from Theorem~\ref{thm:LI} that 
$\bp$ and $\bq$ are both shortest directed paths from $y$ to $w$; 
we write them as follows: 
\begin{align*}
& \bp : y = y_{0} \edge{\beta_{1}} y_{1} \edge{\beta_{2}} \cdots \edge{\beta_{r}} y_{r}=w, \\
& \bq : y = z_{0} \edge{\gamma_{1}} z_{1} \edge{\gamma_{2}} \cdots \edge{\gamma_{r}} z_{r}=w. 
\end{align*}
Because $\bp$ is lexicographically less than or equal to $\bq$ with respect to $\lhd$ 
in the sense of Theorem~\ref{thm:LI}, we have $\gamma_{1} \rhdeq \beta_{1}$. 
Since $\beta_{1} \in \DJp$, and $\lhd$ satisfies condition \eqref{eq:ro}, 
we deduce that $\gamma_{1} \in \DJp$. Since $\prec$ also satisfies condition \eqref{eq:ro}, 
it follows that $\gamma_{u} \in \DJp$ for all $1 \le u \le r$. 
Thus we obtain $\bq \in \mathbf{QBG}_{y}^{\prec}$, as desired. 
This proves that $\ed(\bQBG{y}) \subset \ed(\mathbf{QBG}_{y}^{\prec})$; 
the opposite inclusion can be shown similarly. 
If $\bp \in \bBG{y}$, then we have $w \ge y$. 
By Remark~\ref{rem:shortest}, the directed path $\bq$ is a directed path in $\BG(W)$, 
and hence $\bq \in \mathbf{BG}_{y}^{\prec}$. 
This proves that $\ed(\bBG{y}) \subset \ed(\mathbf{BG}_{y}^{\prec})$; 
the opposite inclusion can be shown similarly. 
This proves the lemma. 
\end{proof}

Finally, let us recall the following from \cite[Lemma~5.14]{LNSSS1}. 
%
%%%%%%%%%%%%%%
%%% lem:DL %%%
%%%%%%%%%%%%%%
%
\begin{lem} \label{lem:DL}
Let $u,\,w \in W$, and $\beta \in \Delta^{+}$. 
Assume that we have a directed edge $u \edge{\beta} w$ in $\QBG(W)$. Let $j \in I$. 
\begin{enu}
\item
If $w^{-1}\alpha_{j} \in \Delta^{-}$ and $u^{-1}\alpha_{j} \in \Delta^{+}$, then 
the directed edge $u \edge{\beta} w$ is a Bruhat edge, and 
$\beta = u^{-1}\alpha_{j}$, $w = s_{j}u$. 

\item If $w^{-1}\alpha_{j},\,u^{-1}\alpha_{j} \in \Delta^{-}$, 
or if $w^{-1}\alpha_{j},\,u^{-1}\alpha_{j} \in \Delta^{+}$, then 
we have a directed edge $s_{j}u \edge{\beta} s_{j}w$ in $\QBG(W)$. 
Moreover, $s_{j}u \edge{\beta} s_{j}w$ is a Bruhat (resp., quantum) edge 
if and only if $u \edge{\beta} w$ is a Bruhat (resp., quantum) edge.
\end{enu}
\end{lem}
%
%==============================%
%     START SUBSECTION 0204    %
%==============================%
%
\subsection{Dual tilted Bruhat order.}
\label{subsec:tilted}
%
%%%%%%%%%%%%%%%%%%%
%%% dfn:dtilted %%%
%%%%%%%%%%%%%%%%%%%
%
\begin{dfn}[{\cite[Definition~2.24]{NOS}}] \label{dfn:dtilted}
For each $v \in W$, we define the dual $v$-tilted Bruhat order 
$\dtb{v}$ on $W$ as follows: for $w_{1},w_{2} \in W$, 
%
%%%%%%%%%%%%%%%%%%
%%% eq:dtilted %%%
%%%%%%%%%%%%%%%%%%
%
\begin{equation} \label{eq:dtilted}
w_{1} \dtb{v} w_{2} \iff \ell(w_{1} \Rightarrow v) = 
 \ell(w_{1} \Rightarrow w_{2}) + \ell(w_{2} \Rightarrow v).
\end{equation}
Namely, $w_{1} \dtb{v} w_{2}$ if and only if 
there exists a shortest directed path in $\QBG(W)$ 
from $w_{1}$ to $v$ passing through $w_{2}$; 
or equivalently, if and only if the concatenation of 
a shortest directed path from $w_{1}$ to $w_{2}$ and 
one from $w_{2}$ to $v$ is one from $w_{1}$ to $v$. 
\end{dfn}
%
%%%%%%%%%%%%%%%%%%
%%% prop:tbmax %%%
%%%%%%%%%%%%%%%%%%
%
\begin{prop}[{\cite[Proposition~2.25]{NOS}}] \label{prop:tbmax}
Let $v \in W$, and let $\J$ be a subset of $I$. 
Then each coset $u\WJs$ for $u \in W$ has a unique maximal element
with respect to $\dtb{v}$\,{\rm;} 
we denote it by $\tbmax{u}{\J}{v}$.
\end{prop}

%%%%%%%%%%%%%%%%%
%%% lem:tbmax %%%
%%%%%%%%%%%%%%%%%
%
\begin{lem} \label{lem:tbmax}
Let $\lhd$ be a reflection order on $\Delta^{+}$ 
satisfying condition \eqref{eq:ro}. 
Let $v,w \in W$, and $w' \in w\WJs$. Then, 
$w' = \tbmax{w}{\J}{v}$ if and only if 
all the labels in the label-increasing {\rm(}shortest{\rm)} directed path 
from $w'$ to $v$ in $\QBG(W)$ are contained in $\DJp$. 
\end{lem}

\begin{proof}
We first show the ``only if'' part. 
Assume that $w' = \tbmax{w}{\J}{v}$, and let 
\begin{equation} \label{eq:dp}
w' = y_{0} \edge{\beta_{1}} y_{1} \edge{\beta_{2}} \cdots \edge{\beta_{r}} y_{r} = v
\end{equation}
be the label-increasing (shortest) directed path from $w'$ to $v$ in $\QBG(W)$. 
By \eqref{eq:ro}, it suffices to show that $\beta_{1} \rhdeq \beta$, 
where $\beta$ is the largest element of $\DJp$ with respect to $\lhd$. 
Suppose, for a contradiction, that $\beta_{1} \lhd \beta$. 
We set $t:=\max \bigl\{ 1 \le u \le r \mid \beta_{u} \lhd \beta \bigr\}$; 
note that $t \ge 1$. Then we see that $y_{t} \in w\WJs$. 
Since $w' = \tbmax{w}{\J}{v}$, it follows that 
$\ell(y_{t} \Rightarrow v) = 
\ell(y_{t} \Rightarrow w') + 
\ell(w' \Rightarrow v) \ge \ell(w' \Rightarrow v)$. 
However, it is obvious by \eqref{eq:dp} that 
$\ell(w' \Rightarrow v) > \ell(y_{t} \Rightarrow v)$, 
which is a contradiction. 

We next show the ``if'' part. Let 
\begin{equation} \label{eq:dp2}
w' = y_{0} \edge{\beta_{1}} y_{1} \edge{\beta_{2}} \cdots \edge{\beta_{r}} y_{r} = v
\end{equation}
be the label-increasing (shortest) directed path from $w'$ to $v$ in $\QBG(W)$, 
where $\beta_{u} \in \DJp$ for all $1 \le u \le r$ by the assumption. 
Here we remark that the full subgraph of $\QBG(W)$ 
whose vertex set is $w\WJs$ is isomorphic, as a $\DJs$-labeled directed graph, 
to the quantum Bruhat graph $\QBG(\WJs)$ associated to the parabolic subgroup $\WJs$, 
under the map $\mcr{w}z \mapsto z$ for $z \in \WJs$. 
Also, we note that the restriction of 
the reflection order $\lhd$ on $\Delta^{+}$ to the subset $\DJs$ is 
a reflection order on $\DJs$. Therefore, it follows from Theorem~\ref{thm:LI} 
(applied to $\QBG(\WJs)$) that for an arbitrary element $w'' \in w\WJs$, 
there exists a directed path 
\begin{equation} \label{eq:dp3}
w'' = z_{0} \edge{\gamma_{1}} z_{1} \edge{\gamma_{2}} \cdots \edge{\gamma_{s}} z_{s} = w'
\end{equation}
in the full subgraph above (and hence in $\QBG(W)$) from $w''$ to $w'$ 
such that $\gamma_{u} \in \DJs$ for all $1 \le u \le s$, and 
$\gamma_{1} \lhd \cdots \lhd \gamma_{s}$; notice that 
this directed path is a shortest directed path from $w''$ to $w'$. 
Hence, by \eqref{eq:ro}, the concatenation of 
the directed paths \eqref{eq:dp3} and \eqref{eq:dp2} 
is the label-increasing (shortest) directed path from $w''$ to $v$ 
passing through $w'$. Thus, we have showm that $w'' \dtb{v} w'$. 
This proves the lemma. 
\end{proof}
%
%==============================%
%     START SUBSECTION 0205    %
%==============================%
%
\subsection{Lakshmibai-Seshadri paths.}
\label{subsec:LS}

In this subsection, we fix $\mu \in P^{+}$, and set
%
%%%%%%%%%%%%
%%% eq:J %%%
%%%%%%%%%%%%
%
\begin{equation} \label{eq:J}
\J=\J_{\mu}:= 
\bigl\{ i \in I \mid \pair{\mu}{\alpha_{i}^{\vee}}=0 \bigr\} \subset I.
\end{equation}
%
%%%%%%%%%%%%%%%%
%%% dfn:PQBG %%%
%%%%%%%%%%%%%%%%
%
\begin{dfn} \label{dfn:PQBG}
The parabolic quantum Bruhat graph, denoted by $\QBG(\WJ)$, is 
the $(\DJp)$-labeled directed graph 
whose vertices are the elements of $\WJ$, and 
whose directed edges are of the form: $w \edge{\beta} v$ 
for $w,v \in \WJ$ and $\beta \in \DJp$ 
such that $v= \mcr{ws_{\beta}}$, and such that either of 
the following holds: 
(B)~$\ell(v) = \ell (w) + 1$; 
(Q)~$\ell(v) = \ell (w) + 1 - 2 \pair{\rho-\rho_{\J}}{\beta^{\vee}}$.
An edge satisfying (B) (resp., (Q)) is called a Bruhat (resp., quantum) edge. 
The parabolic Bruhat graph, denoted by $\BG(\WJ)$, is the 
$(\DJp)$-labeled directed graph obtained from $\QBG(\WJ)$
by removing all quantum edges.
\end{dfn}%
%
%%%%%%%%%%%%%%%
%%% dfn:QBa %%%
%%%%%%%%%%%%%%%
%
\begin{dfn} \label{dfn:QBa}
Let $0 < \sigma < 1$ be a rational number. 
We define $\QBa$ (resp., $\Ba$) to be the subgraph of $\QBG(\WJ)$ (resp., $\BG(\WJ)$) 
with the same vertex set but having only those directed edges of 
the form $w \edge{\beta} v$ for which 
$\sigma \pair{\mu}{\beta^{\vee}} \in \BZ$ holds.
\end{dfn}
%
%%%%%%%%%%%%%%%
%%% dfn:QLS %%%
%%%%%%%%%%%%%%%
%
\begin{dfn}[{\cite[Definition~3.1]{LNSSS2}}] \label{dfn:QLS}
A quantum Lakshmibai-Seshadri path of shape $\mu$ 
(resp., a Lakshmibai-Seshadri path of shape $\mu$) is a pair 
%
%%%%%%%%%%%%%%
%%% eq:QLS %%%
%%%%%%%%%%%%%%
%
\begin{equation} \label{eq:QLS}
\eta = (\bx \,;\, \bsig) = 
(x_{1},\,\dots,\,x_{s} \,;\, \sigma_{0},\,\sigma_{1},\,\dots,\,\sigma_{s}), \quad s \ge 1, 
\end{equation}
of a sequence $x_{1},\,\dots,\,x_{s}$ 
of elements in $\WJ$, with $x_{u} \ne x_{u+1}$ 
for any $1 \le u \le s-1$, and an increasing sequence 
$0 = \sigma_0 < \sigma_1 < \cdots  < \sigma_s =1$ of rational numbers 
satisfying the condition that there exists a directed path 
in $\QBb{\sigma_{u}}$ (resp., $\Bb{\sigma_{u}}$) from $x_{u+1}$ to  $x_{u}$ 
for each $u = 1,\,2,\,\dots,\,s-1$. 
\end{dfn}

Denote by $\QLS(\mu)$ and $\LS(\mu)$
the sets of all quantum Lakshmibai-Seshadri paths and 
all Lakshmibai-Seshadri paths of shape $\mu$, respectively; 
note that $\LS(\mu) \subset \QLS(\mu)$. For $\eta \in \QLS(\mu)$ 
of the form \eqref{eq:QLS}, we set $\iota(\eta):=x_{1}$, $\kappa(\eta):=x_{s}$, and 
%
%%%%%%%%%%%%%
%%% eq:wt %%%
%%%%%%%%%%%%%
%
\begin{equation} \label{eq:wt}
\wt (\eta) := \sum_{u = 1}^{s} (\sigma_{u}-\sigma_{u-1})x_{u}\mu \in P, 
\end{equation}
%
%%%%%%%%%%%%%%
%%% eq:deg %%%
%%%%%%%%%%%%%%
%
\begin{equation} \label{eq:deg}
\Deg (\eta) := - \sum_{u=1}^{s-1} \sigma_{u} 
 \pair{\mu}{\wt(x_{u+1} \Rightarrow x_{u})} \in \BZ_{\le 0}. 
\end{equation}
For $\eta = (x_{1},\,\dots,\,x_{s} \,;\, 
\sigma_{0},\,\sigma_{1},\,\dots,\,\sigma_{s}) \in \QLS(\mu)$ and $v \in W$, 
define $\kap{\eta}{v} \in W$ by the following recursive formula: 
%
%%%%%%%%%%%%%%
%%% eq:haw %%%
%%%%%%%%%%%%%%
%
\begin{equation} \label{eq:haw}
\begin{cases}
\ha{x}_{0}:=v, & \\[2mm]
\ha{x}_{u}:=\tbmax{x_{u}}{\J}{\ha{x}_{u-1}} & \text{for $1 \le u \le s$}, \\[2mm]
\kap{\eta}{v}:=\ha{x}_{s}. 
\end{cases}
\end{equation}
We set
%
%%%%%%%%%%%%%%%
%%% eq:zeta %%%
%%%%%%%%%%%%%%%
%
\begin{equation} \label{eq:zeta}
\zet{\eta}{v}:=
\wt ( \ha{x}_{1} \Rightarrow v ) + 
\sum_{u=1}^{s-1} \wt (\ha{x}_{u+1} \Rightarrow \ha{x}_{u}).
\end{equation}
%
%==============================%
%     START SUBSECTION 0206    %
%==============================%
%
\subsection{Character identity of Chevalley type for general anti-dominant weights.}
\label{subsec:chevalley}
Let $\Fg_{\af} = \bigl(\BC[z,z^{-1}] \otimes \Fg\bigr) \oplus \BC c \oplus \BC d$ be 
the (untwisted) affine Lie algebra over $\BC$ associated to 
the finite-dimensional simple Lie algebra $\Fg$, 
where $c$ is the canonical central element and $d$ is 
the scaling element (or degree operator), 
with Cartan subalgebra $\Fh_{\af} = \Fh \oplus \BC c \oplus \BC d$. 
We regard an element $\mu \in \Fh^{\ast}:=\Hom_{\BC}(\Fh,\,\BC)$ as an element of 
$\Fh_{\af}^{\ast}$ by setting $\pair{\mu}{c}=\pair{\mu}{d}:=0$, where 
$\pair{\cdot\,}{\cdot}:\Fh_{\af}^{\ast} \times \Fh_{\af} \rightarrow \BC$ denotes
the canonical pairing of $\Fh_{\af}^{\ast}:=\Hom_{\BC}(\Fh_{\af},\,\BC)$ and $\Fh_{\af}$. 
Let $\{ \alpha_{i}^{\vee} \}_{i \in I_{\af}} \subset \Fh_{\af}$ and 
$\{ \alpha_{i} \}_{i \in I_{\af}} \subset \Fh_{\af}^{\ast}$ be the set of 
simple coroots and simple roots of $\Fg_{\af}$, respectively, 
where $I_{\af}:=I \sqcup \{0\}$; note that 
$\pair{\alpha_{i}}{c}=0$ and $\pair{\alpha_{i}}{d}=\delta_{i,0}$ 
for $i \in I_{\af}$. 
Denote by $\delta \in \Fh_{\af}^{\ast}$ the null root of $\Fg_{\af}$; 
recall that $\alpha_{0}=\delta-\theta$. 
Let $W_{\af}$ be the (affine) Weyl group of $\Fg_{\af}$, with $e$ the identity element. 
For each $\xi \in Q^{\vee}$, 
let $t_{\xi} \in W_{\af}$ denote the translation in $\Fh_{\af}^{\ast}$ 
by $\xi$ (see \cite[Sect.~6.5]{K}); recall that 
$W_{\af} \cong W \ltimes \bigl\{ t_{\xi} \mid \xi \in Q^{\vee} \bigr\} 
\cong W \ltimes Q^{\vee}$. 
Finally, let $U_{\q}(\Fg_{\af})$ denote the quantized universal 
enveloping algebra over $\BC(\q)$ associated to $\Fg_{\af}$, 
with $E_{i}$ and $F_{i}$ the Chevalley generators 
corresponding to $\alpha_{i}$ for $i \in I_{\af}$. 
We denote by $U_{\q}^{-}(\Fg_{\af})$ 
the negative part of $U_{\q}(\Fg_{\af})$, that is, 
the $\BC(\q)$-subalgebra of $U_{\q}(\Fg_{\af})$ 
generated by the $F_{i}$, $i \in I_{\af}$. 

We take an arbitrary $\lambda \in P^{+}=\sum_{i \in I}\BZ_{\ge 0}\vpi_{i}$. 
Let $V(\lambda)$ denote the (level-zero) extremal weight module of 
extremal weight $\lambda$ over $U_{\q}(\Fg_{\af})$, 
which is defined to be 
the integrable $U_{\q}(\Fg_{\af})$-module generated by 
a single element $v_{\lambda}$ 
with the defining relation that ``$v_{\lambda}$ is 
an extremal weight vector of weight $\lambda$''. 
Here, recall from \cite[Sect.~3.1]{Kas02} and \cite[Sect.~2.6]{Kas05} that 
$v_{\lambda}$ is an extremal weight vector of weight $\lambda$ 
if and only if ($v_{\lambda}$ is a weight vector of weight $\lambda$ and) 
there exists a family $\{ v_{x} \}_{x \in W_{\af}}$ 
of weight vectors in $V(\lambda)$ such that $v_{e}=v_{\lambda}$, 
and such that for each $i \in I_{\af}$ and $x \in W_{\af}$ with 
$n:=\pair{x\lambda}{\alpha_{i}^{\vee}} \ge 0$ (resp., $\le 0$),
the equalities $E_{i}v_{x}=0$ and $F_{i}^{(n)}v_{x}=v_{s_{i}x}$ 
(resp., $F_{i}v_{x}=0$ and $E_{i}^{(-n)}v_{x}=v_{s_{i}x}$) hold, 
where for $i \in I_{\af}$ and $k \in \BZ_{\ge 0}$, 
the $E_{i}^{(k)}$ and $F_{i}^{(k)}$ are the $k$-th divided powers of 
the Chevalley generators $E_{i}$ and $F_{i}$ of $U_{\q}(\Fg_{\af})$, respectively;
note that the weight of $v_{x}$ is $x\lambda$. 
Also, for each $x \in W_{\af}$, we define 
the Demazure submodule $V_{x}^{-}(\lambda)$ of $V(\lambda)$ by 
$V_{x}^{-}(\lambda):=U_{\q}^{-}(\Fg_{\af})v_{x}$. 
%
%%%%%%%%%%%%%%%
%%% rem:dem %%%
%%%%%%%%%%%%%%%
%
\begin{rem} \label{rem:dem}
Keep the notation and setting above. 
Take $\J=\J_{\lambda}$ as in \eqref{eq:J}. 
We deduce from \cite[Lemma~4.1.2]{NS16} that 
$V_{yt_{\xi}}^{-}(\lambda) = V_{\mcr{y}t_{[\xi]}}^{-}(\lambda)$ 
for $y \in W$ and $\xi \in Q^{\vee}$; for the notation $\mcr{y}=\mcr{y}^{\J}$ and 
$[\xi]=[\xi]^{\J}$, see Section~\ref{subsec:lie}.
\end{rem}

Following \cite[Sect.~2.4]{KNS}, 
we define the graded character $\gch V_{x}^{-}(\lambda)$ of 
$V_{x}^{-}(\lambda)$ by
\begin{equation} \label{eq:gch}
\gch V_{x}^{-}(\lambda) : = 
 \sum_{k \in \BZ}
\Biggl(
 \sum_{\gamma \in Q} 
 \dim \bigl( V_{x}^{-}(\lambda)_{\lambda+\gamma+k\delta} \bigr) 
 \be^{\lambda+\gamma}
\Biggr) q^{k}, \quad \text{where $q:=\be^{\delta}$}. 
\end{equation}

%%%%%%%%%%%%%%%
%%% thm:NOS %%%
%%%%%%%%%%%%%%%
%
\begin{thm}[{\cite[Corollary~3.15]{NOS}}] \label{thm:NOS}
Let $\mu \in P^{+}$ and $x \in W$. 
For all $\lambda \in P^{+}$ such that $\lambda-\mu \in P^{+}$, 
the following identity holds\,{\rm:}
%
%%%%%%%%%%%%%%
%%% eq:NOS %%%
%%%%%%%%%%%%%%
%
\begin{equation} \label{eq:NOS}
\gch V_{x}^{-}(\lambda-\mu) = 
\sum_{v \in W}\,
\sum_{
  \begin{subarray}{c}
  \eta \in \QLS(\mu) \\[1mm]
  \kap{\eta}{v} = x 
  \end{subarray}}
(-1)^{\ell(v)-\ell(x)} q^{-\Deg(\eta)}\be^{-\wt(\eta)}
\gch V_{vt_{\zeta(\eta,v)}}^{-}(\lambda). 
\end{equation}
\end{thm}
%
%=========================%
%     START SECTION 03    %
%=========================%
%
\section{Character identity of Chevalley type \\ for anti-dominant minuscule fundamental weights.}
\label{sec:main}

Assume that $\Fg$ is simply-laced or of type $B_{n}$. 
Let $k \in I$ be such that $\vpi_{k}$ is minuscule, 
that is, $\pair{\vpi_{k}}{\beta^{\vee}} \in \bigl\{-1,\,0,\,1\bigr\}$ 
for all $\beta \in \Delta$; the fundamental weights 
corresponding to black nodes of the Dynkin diagrams below are 
the minuscule fundamental weights: 
\begin{center}
\vspace{5mm}

%WinTpicVersion4.32a
{\unitlength 0.1in%
\begin{picture}(49.4500,23.1000)(0.9500,-30.4500)%
% CIRCLE 2 0 0 0 Black Black  
% 4 600 800 600 760 600 760 600 760
% 
\special{sh 1.000}%
\special{ia 600 800 40 40 0.0000000 6.2831853}%
\special{pn 8}%
\special{ar 600 800 40 40 0.0000000 6.2831853}%
% CIRCLE 2 0 0 0 Black Black  
% 4 1000 800 1000 760 1000 760 1000 760
% 
\special{sh 1.000}%
\special{ia 1000 800 40 40 0.0000000 6.2831853}%
\special{pn 8}%
\special{ar 1000 800 40 40 0.0000000 6.2831853}%
% CIRCLE 2 0 0 0 Black Black  
% 4 1400 800 1400 760 1400 760 1400 760
% 
\special{sh 1.000}%
\special{ia 1400 800 40 40 0.0000000 6.2831853}%
\special{pn 8}%
\special{ar 1400 800 40 40 0.0000000 6.2831853}%
% CIRCLE 2 0 0 0 Black Black  
% 4 2800 800 2800 760 2800 760 2800 760
% 
\special{sh 1.000}%
\special{ia 2800 800 40 40 0.0000000 6.2831853}%
\special{pn 8}%
\special{ar 2800 800 40 40 0.0000000 6.2831853}%
% CIRCLE 2 0 0 0 Black Black  
% 4 3200 800 3200 760 3200 760 3200 760
% 
\special{sh 1.000}%
\special{ia 3200 800 40 40 0.0000000 6.2831853}%
\special{pn 8}%
\special{ar 3200 800 40 40 0.0000000 6.2831853}%
% LINE 2 0 3 0 Black White  
% 2 650 800 950 800
% 
\special{pn 8}%
\special{pa 650 800}%
\special{pa 950 800}%
\special{fp}%
% LINE 2 0 3 0 Black White  
% 2 1050 800 1350 800
% 
\special{pn 8}%
\special{pa 1050 800}%
\special{pa 1350 800}%
\special{fp}%
% LINE 2 0 3 0 Black White  
% 2 1450 800 1750 800
% 
\special{pn 8}%
\special{pa 1450 800}%
\special{pa 1750 800}%
\special{fp}%
% LINE 2 0 3 0 Black White  
% 2 2450 800 2750 800
% 
\special{pn 8}%
\special{pa 2450 800}%
\special{pa 2750 800}%
\special{fp}%
% LINE 2 0 3 0 Black White  
% 2 2850 800 3150 800
% 
\special{pn 8}%
\special{pa 2850 800}%
\special{pa 3150 800}%
\special{fp}%
% LINE 2 2 3 0 Black White  
% 2 1750 800 2450 800
% 
\special{pn 8}%
\special{pa 1750 800}%
\special{pa 2450 800}%
\special{dt 0.045}%
% CIRCLE 2 0 3 0 Black Black  
% 4 600 1200 600 1160 600 1160 600 1160
% 
\special{pn 8}%
\special{ar 600 1200 40 40 0.0000000 6.2831853}%
% CIRCLE 2 0 3 0 Black Black  
% 4 1000 1200 1000 1160 1000 1160 1000 1160
% 
\special{pn 8}%
\special{ar 1000 1200 40 40 0.0000000 6.2831853}%
% CIRCLE 2 0 3 0 Black Black  
% 4 1400 1200 1400 1160 1400 1160 1400 1160
% 
\special{pn 8}%
\special{ar 1400 1200 40 40 0.0000000 6.2831853}%
% CIRCLE 2 0 3 0 Black Black  
% 4 2800 1200 2800 1160 2800 1160 2800 1160
% 
\special{pn 8}%
\special{ar 2800 1200 40 40 0.0000000 6.2831853}%
% CIRCLE 2 0 0 0 Black Black  
% 4 3200 1200 3200 1160 3200 1160 3200 1160
% 
\special{sh 1.000}%
\special{ia 3200 1200 40 40 0.0000000 6.2831853}%
\special{pn 8}%
\special{ar 3200 1200 40 40 0.0000000 6.2831853}%
% LINE 2 0 3 0 Black White  
% 2 650 1200 950 1200
% 
\special{pn 8}%
\special{pa 650 1200}%
\special{pa 950 1200}%
\special{fp}%
% LINE 2 0 3 0 Black White  
% 2 1050 1200 1350 1200
% 
\special{pn 8}%
\special{pa 1050 1200}%
\special{pa 1350 1200}%
\special{fp}%
% LINE 2 0 3 0 Black White  
% 2 1450 1200 1750 1200
% 
\special{pn 8}%
\special{pa 1450 1200}%
\special{pa 1750 1200}%
\special{fp}%
% LINE 2 0 3 0 Black White  
% 2 2450 1200 2750 1200
% 
\special{pn 8}%
\special{pa 2450 1200}%
\special{pa 2750 1200}%
\special{fp}%
% LINE 2 2 3 0 Black White  
% 2 1750 1200 2450 1200
% 
\special{pn 8}%
\special{pa 1750 1200}%
\special{pa 2450 1200}%
\special{dt 0.045}%
% CIRCLE 2 0 0 0 Black Black  
% 4 600 2000 600 1960 600 1960 600 1960
% 
\special{sh 1.000}%
\special{ia 600 2000 40 40 0.0000000 6.2831853}%
\special{pn 8}%
\special{ar 600 2000 40 40 0.0000000 6.2831853}%
% CIRCLE 2 0 3 0 Black Black  
% 4 1000 2000 1000 1960 1000 1960 1000 1960
% 
\special{pn 8}%
\special{ar 1000 2000 40 40 0.0000000 6.2831853}%
% CIRCLE 2 0 3 0 Black Black  
% 4 1400 2000 1400 1960 1400 1960 1400 1960
% 
\special{pn 8}%
\special{ar 1400 2000 40 40 0.0000000 6.2831853}%
% CIRCLE 2 0 3 0 Black Black  
% 4 2800 2000 2800 1960 2800 1960 2800 1960
% 
\special{pn 8}%
\special{ar 2800 2000 40 40 0.0000000 6.2831853}%
% CIRCLE 2 0 0 0 Black Black  
% 4 3200 2200 3200 2160 3200 2160 3200 2160
% 
\special{sh 1.000}%
\special{ia 3200 2200 40 40 0.0000000 6.2831853}%
\special{pn 8}%
\special{ar 3200 2200 40 40 0.0000000 6.2831853}%
% LINE 2 0 3 0 Black White  
% 2 650 2000 950 2000
% 
\special{pn 8}%
\special{pa 650 2000}%
\special{pa 950 2000}%
\special{fp}%
% LINE 2 0 3 0 Black White  
% 2 1050 2000 1350 2000
% 
\special{pn 8}%
\special{pa 1050 2000}%
\special{pa 1350 2000}%
\special{fp}%
% LINE 2 0 3 0 Black White  
% 2 1450 2000 1750 2000
% 
\special{pn 8}%
\special{pa 1450 2000}%
\special{pa 1750 2000}%
\special{fp}%
% LINE 2 0 3 0 Black White  
% 2 2450 2000 2750 2000
% 
\special{pn 8}%
\special{pa 2450 2000}%
\special{pa 2750 2000}%
\special{fp}%
% LINE 2 2 3 0 Black White  
% 2 1750 2000 2450 2000
% 
\special{pn 8}%
\special{pa 1750 2000}%
\special{pa 2450 2000}%
\special{dt 0.045}%
% CIRCLE 2 0 0 0 Black Black  
% 4 3200 1800 3200 1760 3200 1760 3200 1760
% 
\special{sh 1.000}%
\special{ia 3200 1800 40 40 0.0000000 6.2831853}%
\special{pn 8}%
\special{ar 3200 1800 40 40 0.0000000 6.2831853}%
% LINE 2 0 3 0 Black White  
% 2 2840 1980 3160 1820
% 
\special{pn 8}%
\special{pa 2840 1980}%
\special{pa 3160 1820}%
\special{fp}%
% LINE 2 0 3 0 Black White  
% 2 2840 2020 3160 2180
% 
\special{pn 8}%
\special{pa 2840 2020}%
\special{pa 3160 2180}%
\special{fp}%
% LINE 2 0 3 0 Black White  
% 2 2800 1150 3200 1150
% 
\special{pn 8}%
\special{pa 2800 1150}%
\special{pa 3200 1150}%
\special{fp}%
% LINE 2 0 3 0 Black White  
% 2 2800 1250 3200 1250
% 
\special{pn 8}%
\special{pa 2800 1250}%
\special{pa 3200 1250}%
\special{fp}%
% LINE 2 0 3 0 Black White  
% 2 3050 1200 2950 1100
% 
\special{pn 8}%
\special{pa 3050 1200}%
\special{pa 2950 1100}%
\special{fp}%
% LINE 2 0 3 0 Black White  
% 2 3050 1200 2950 1300
% 
\special{pn 8}%
\special{pa 3050 1200}%
\special{pa 2950 1300}%
\special{fp}%
% CIRCLE 2 0 0 0 Black Black  
% 4 600 2600 600 2560 600 2560 600 2560
% 
\special{sh 1.000}%
\special{ia 600 2600 40 40 0.0000000 6.2831853}%
\special{pn 8}%
\special{ar 600 2600 40 40 0.0000000 6.2831853}%
% CIRCLE 2 0 3 0 Black Black  
% 4 1000 2600 1000 2560 1000 2560 1000 2560
% 
\special{pn 8}%
\special{ar 1000 2600 40 40 0.0000000 6.2831853}%
% CIRCLE 2 0 3 0 Black Black  
% 4 1400 2600 1400 2560 1400 2560 1400 2560
% 
\special{pn 8}%
\special{ar 1400 2600 40 40 0.0000000 6.2831853}%
% CIRCLE 2 0 3 0 Black Black  
% 4 1800 2600 1800 2560 1800 2560 1800 2560
% 
\special{pn 8}%
\special{ar 1800 2600 40 40 0.0000000 6.2831853}%
% CIRCLE 2 0 0 0 Black Black  
% 4 2200 2600 2200 2560 2200 2560 2200 2560
% 
\special{sh 1.000}%
\special{ia 2200 2600 40 40 0.0000000 6.2831853}%
\special{pn 8}%
\special{ar 2200 2600 40 40 0.0000000 6.2831853}%
% CIRCLE 2 0 3 0 Black Black  
% 4 1400 3000 1400 2960 1400 2960 1400 2960
% 
\special{pn 8}%
\special{ar 1400 3000 40 40 0.0000000 6.2831853}%
% LINE 2 0 3 0 Black White  
% 2 650 2600 950 2600
% 
\special{pn 8}%
\special{pa 650 2600}%
\special{pa 950 2600}%
\special{fp}%
% LINE 2 0 3 0 Black White  
% 2 1050 2600 1350 2600
% 
\special{pn 8}%
\special{pa 1050 2600}%
\special{pa 1350 2600}%
\special{fp}%
% LINE 2 0 3 0 Black White  
% 2 1450 2600 1750 2600
% 
\special{pn 8}%
\special{pa 1450 2600}%
\special{pa 1750 2600}%
\special{fp}%
% LINE 2 0 3 0 Black White  
% 2 1850 2600 2150 2600
% 
\special{pn 8}%
\special{pa 1850 2600}%
\special{pa 2150 2600}%
\special{fp}%
% LINE 2 0 3 0 Black White  
% 2 1400 2650 1400 2950
% 
\special{pn 8}%
\special{pa 1400 2650}%
\special{pa 1400 2950}%
\special{fp}%
% CIRCLE 2 0 3 0 Black Black  
% 4 3000 2605 3000 2565 3000 2565 3000 2565
% 
\special{pn 8}%
\special{ar 3000 2605 40 40 0.0000000 6.2831853}%
% CIRCLE 2 0 3 0 Black Black  
% 4 3400 2605 3400 2565 3400 2565 3400 2565
% 
\special{pn 8}%
\special{ar 3400 2605 40 40 0.0000000 6.2831853}%
% CIRCLE 2 0 3 0 Black Black  
% 4 3800 2605 3800 2565 3800 2565 3800 2565
% 
\special{pn 8}%
\special{ar 3800 2605 40 40 0.0000000 6.2831853}%
% CIRCLE 2 0 3 0 Black Black  
% 4 4200 2605 4200 2565 4200 2565 4200 2565
% 
\special{pn 8}%
\special{ar 4200 2605 40 40 0.0000000 6.2831853}%
% CIRCLE 2 0 3 0 Black Black  
% 4 4600 2605 4600 2565 4600 2565 4600 2565
% 
\special{pn 8}%
\special{ar 4600 2605 40 40 0.0000000 6.2831853}%
% CIRCLE 2 0 3 0 Black Black  
% 4 3800 3005 3800 2965 3800 2965 3800 2965
% 
\special{pn 8}%
\special{ar 3800 3005 40 40 0.0000000 6.2831853}%
% LINE 2 0 3 0 Black White  
% 2 3050 2605 3350 2605
% 
\special{pn 8}%
\special{pa 3050 2605}%
\special{pa 3350 2605}%
\special{fp}%
% LINE 2 0 3 0 Black White  
% 2 3450 2605 3750 2605
% 
\special{pn 8}%
\special{pa 3450 2605}%
\special{pa 3750 2605}%
\special{fp}%
% LINE 2 0 3 0 Black White  
% 2 3850 2605 4150 2605
% 
\special{pn 8}%
\special{pa 3850 2605}%
\special{pa 4150 2605}%
\special{fp}%
% LINE 2 0 3 0 Black White  
% 2 4250 2605 4550 2605
% 
\special{pn 8}%
\special{pa 4250 2605}%
\special{pa 4550 2605}%
\special{fp}%
% LINE 2 0 3 0 Black White  
% 2 3800 2655 3800 2955
% 
\special{pn 8}%
\special{pa 3800 2655}%
\special{pa 3800 2955}%
\special{fp}%
% CIRCLE 2 0 0 0 Black Black  
% 4 5000 2605 5000 2565 5000 2565 5000 2565
% 
\special{sh 1.000}%
\special{ia 5000 2605 40 40 0.0000000 6.2831853}%
\special{pn 8}%
\special{ar 5000 2605 40 40 0.0000000 6.2831853}%
% LINE 2 0 3 0 Black White  
% 2 4650 2605 4950 2605
% 
\special{pn 8}%
\special{pa 4650 2605}%
\special{pa 4950 2605}%
\special{fp}%
% STR 2 0 3 0 Black White  
% 4 600 1300 600 1400 5 0 0 0
% $1$
\put(6.0000,-14.0000){\makebox(0,0){$1$}}%
% STR 2 0 3 0 Black White  
% 4 1000 1300 1000 1400 5 0 0 0
% $2$
\put(10.0000,-14.0000){\makebox(0,0){$2$}}%
% STR 2 0 3 0 Black White  
% 4 1400 1300 1400 1400 5 0 0 0
% $3$
\put(14.0000,-14.0000){\makebox(0,0){$3$}}%
% STR 2 0 3 0 Black White  
% 4 2800 1300 2800 1400 5 0 0 0
% $n-1$
\put(28.0000,-14.0000){\makebox(0,0){$n-1$}}%
% STR 2 0 3 0 Black White  
% 4 3200 1300 3200 1400 5 0 0 0
% $n$
\put(32.0000,-14.0000){\makebox(0,0){$n$}}%
% STR 2 0 3 0 Black White  
% 4 300 700 300 800 5 0 0 0
% $A_{n}$
\put(3.0000,-8.0000){\makebox(0,0){$A_{n}$}}%
% STR 2 0 3 0 Black White  
% 4 300 1100 300 1200 5 0 0 0
% $B_{n}$
\put(3.0000,-12.0000){\makebox(0,0){$B_{n}$}}%
% STR 2 0 3 0 Black White  
% 4 300 1900 300 2000 5 0 0 0
% $D_{n}$
\put(3.0000,-20.0000){\makebox(0,0){$D_{n}$}}%
% STR 2 0 3 0 Black White  
% 4 300 2500 300 2600 5 0 0 0
% $E_{6}$
\put(3.0000,-26.0000){\makebox(0,0){$E_{6}$}}%
% STR 2 0 3 0 Black White  
% 4 2700 2505 2700 2605 5 0 0 0
% $E_{7}$
\put(27.0000,-26.0500){\makebox(0,0){$E_{7}$}}%
\end{picture}}%

\end{center}

\noindent
We set $\J:=\J_{\vpi_{k}}=I \setminus \{k\}$. 
Fix an (arbitrary) reflection order $\lhd$ satisfying condition \eqref{eq:ro}; 
recall from Lemma~\ref{lem:bBGy} that for each $y \in \WJ$, 
the set $\ed(\bBG{y})$ does not depend on 
the choice of a reflection order $\lhd$ satisfying condition \eqref{eq:ro}. 
In addition, we set 
%
%%%%%%%%%%%%%%%
%%% eq:gamq %%%
%%%%%%%%%%%%%%%
%
\begin{equation} \label{eq:gamq}
\gq = \gq(\Fg,\,k):=
 \begin{cases}
 \alpha_{k} & \text{if $\Fg$ is simply-laced}, \\[1mm]
 s_{n}\alpha_{n-1}=\alpha_{n-1}+2\alpha_{n} & \text{if $\Fg$ is of type $B_{n}$ and $k=n$}.
 \end{cases}
\end{equation}
%
%%%%%%%%%%%%%%%%
%%% thm:main %%%
%%%%%%%%%%%%%%%%
%
\begin{thm}\label{thm:main}
Assume that $\Fg$ is simply-laced or of type $B_{n}$. 
Let $k \in I$ be such that $\vpi_{k}$ is minuscule, 
and set $\J=\J_{\vpi_{k}}=I \setminus \{k\}$. 
Let $x \in \WJ$. For all $N \ge 1$, the following hold: 
\begin{enu}
\item If $x \ge \mcr{ s_{\theta} }$, then 
%
%%%%%%%%%%%%%%%%%
%%% eq:main_1 %%%
%%%%%%%%%%%%%%%%
%
\begin{equation}\label{eq:main_1}
\begin{split}
\gch V_{x}^{-}((N-1)\vpi_{k}) & = \be^{-x\vpi_{k}} 
\sum_{y \in \ed(\bBG{x})} (-1)^{\ell(y)-\ell(x)} \gch V_{y}^{-}(N\vpi_{k}) \\[3mm]
& \quad + 
\be^{-x\vpi_{k}} 
\sum_{y \in \ed(\bBG{x})} (-1)^{\ell(y)-\ell(x)+1} 
\gch V_{ \mcr{ys_{\gq} } t_{\alpha_{k}^{\vee}} }^{-}(N\vpi_{k}). 
\end{split}
\end{equation}
Moreover, in the second sum on the right-hand side of \eqref{eq:main_1}, 
no cancellations occur, or equivalently, 
$\mcr{ ys_{\gq} } \ne \mcr{ y's_{\gq} }$ for any 
$y,\,y' \in \ed(\bBG{x})$ with $y \ne y'$. 

\item If $x \not\ge \mcr{ s_{\theta} }$, then
%
%%%%%%%%%%%%%%%%%
%%% eq:main_2 %%%
%%%%%%%%%%%%%%%%%
%
\begin{equation} \label{eq:main_2}
\gch V_{x}^{-}((N-1)\vpi_{k}) = \be^{-x\vpi_{k}} 
\sum_{y \in \ed(\bBG{x})} (-1)^{\ell(y)-\ell(x)} \gch V_{y}^{-}(N\vpi_{k}). 
\end{equation}
\end{enu}
\end{thm}

Here we show Theorem~\ref{thm:main} in the case that $x = e$.
%
%%%%%%%%%%%%%%%%
%%% prop:x=e %%%
%%%%%%%%%%%%%%%%
%
\begin{prop} \label{prop:x=e}
Keep the notation and setting of Theorem~\ref{thm:main}. 
If $x = e$ (note that $x \not\ge \mcr{s_{\theta}}$), then 
the character identity \eqref{eq:main_2} holds. 
\end{prop}

\begin{proof}
Since the reflection order $\lhd$ satisfies \eqref{eq:ro} (with $\J=I \setminus \{k\}$), 
we see that $\alpha_{k}$ is the largest element in $\Delta^{+}$ with respect to $\lhd$. 
It is easily seen that 
$\bBG{x} = \bBG{e}$ is identical to the set 
$\bigl\{ \text{$e$ (directed path of length $0$)}, \, e \edge{\alpha_{k}} s_{k} \bigr\}$. 
Therefore, the right-hand side of \eqref{eq:main_2} can be written as: 
\begin{equation} \label{eq:x=e}
\be^{-\vpi_{k}}( \gch V_{e}^{-}(N\vpi_{k}) -  \gch V_{s_{k}}^{-}(N\vpi_{k})), 
\end{equation}
which is identical to $\gch V_{e}^{-}((N-1)\vpi_{k})$ 
on the left-hand side of \eqref{eq:main_2} 
by \cite[Proposition~5.3]{NOS}. This proves the proposition. 
\end{proof}

In the rest of this paper, we will prove Theorem~\ref{thm:main} in the case that $x \ne e$. 
We divide our proof as follows. In Section~\ref{sec:prfa}, 
we give a proof in simply-laced types. In Section~\ref{sec:prfb}, 
we give a proof in type $B_{n}$. 
Before giving these proofs, we show some technical lemmas 
in Section~\ref{sec:pre}, which are valid in both types. 
%
%=========================%
%     START SECTION 04    %
%=========================%
%
\section{Recurrence relations for coefficients in the character identity of Chevalley type.}
\label{sec:pre}

As in Section~\ref{sec:main}, 
assume that $\Fg$ is simply-laced or of type $B_{n}$. 
Let $k \in I$ be such that $\vpi_{k}$ is minuscule, 
and set $\J=\J_{\vpi_{k}} = I \setminus \{k\}$. 
%
%==============================%
%     START SUBSECTION 0401    %
%==============================%
%
\subsection{Quantum Lakshmibai-Seshadri paths of shape $\vpi_{k}$.}
\label{subsec:lem_QLS}

Since $\vpi_{k}$ is minuscule, we have 
$\pair{\vpi_{k}}{\beta^{\vee}} \in \bigl\{0,\,1\bigr\}$ 
for all $\beta \in \Delta^{+}$. 
Therefore, $\QBa$ (and hence $\Ba$) 
has no directed edges for any rational number $0 < \sigma < 1$. 
Hence we obtain the following. 
%
%%%%%%%%%%%%%%%%%%
%%% lem:QLS=LS %%%
%%%%%%%%%%%%%%%%%%
%
\begin{lem} \label{lem:QLS=LS}
It holds that $\QLS(\vpi_{k})=\LS(\vpi_{k})=
\bigl\{(w\,;\,0,\,1) \mid w \in \WJ \bigr\}$. 
Therefore, $\deg (\eta) = 0$ for all $\eta \in \QLS(\vpi_{k})=\LS(\vpi_{k})$. 
\end{lem}

Let $x \in \WJ$.  Let $\eta \in \QLS(\vpi_{k})=\LS(\vpi_{k})$ 
and $v \in W$ be such that $\kap{\eta}{v}=x$. 
By Lemma~\ref{lem:QLS=LS}, $\eta = (w\,;\,0,\,1)$ for some $w \in \WJ$. 
Hence we have $w = \kappa(\eta) = \mcr{ \kap{\eta}{v} } = \mcr{x} = x$ 
since $x \in \WJ$. Thus we obtain $\eta = (x\,;\,0,\,1)$. 
Therefore, by \eqref{eq:haw}, we see that $\tbmax{x}{\J}{v} = x$. 
Let $\lhd$ be a reflection order on $\Delta^{+}$ 
satisfying condition \eqref{eq:ro}. It follows from Lemma~\ref{lem:tbmax} that 
$\tbmax{x}{\J}{v} = x$ if and only if 
all the labels in the label-increasing directed path from $x$ to $v$ 
in $\QBG(W)$ are contained in $\DJp$. Therefore, 
by Theorem~\ref{thm:NOS}, Remarks~\ref{rem:bBGy}\,(1) and \ref{rem:dem}, 
we deduce that 
%
%%%%%%%%%%%%%%%%%%%%
%%% eq:munuscule %%%
%%%%%%%%%%%%%%%%%%%%
%
\begin{equation} \label{eq:minuscule}
\begin{split}
\gch V_{x}^{-} ((N-1)\vpi_{k}) & = 
 \be^{-x \vpi_{k}} \sum_{\bp \in \bQBG{x}} 
 (-1)^{\ell(\ed(\bp))-\ell(x)} \gch V_{ \mcr{\ed(\bp)} t_{[\wt(\bp)]} }^{-} (N\vpi_{k}) \\[3mm]
& = 
 \be^{-x \vpi_{k}} \sum_{y \in \ed(\bQBG{x})} 
 (-1)^{\ell(y)-\ell(x)} \gch V_{ \mcr{y} t_{[\wt(x \Rightarrow y)]} }^{-} (N\vpi_{k})
\end{split}
\end{equation}
for all $N \in \BZ_{\ge 1}$; for the notation $\mcr{\,\cdot\,}=\mcr{\,\cdot\,}^{\J}$ and 
$[\,\cdot\,]=[\,\cdot\,]^{\J}$, see Section~\ref{subsec:lie}. 
Recall from Lemma~\ref{lem:bBGy} that $\ed(\bQBG{y})$ does not depend on 
the choice of a reflection order $\lhd$ satisfying condition \eqref{eq:ro}. 
%
%==============================%
%     START SUBSECTION 0402    %
%==============================%
%
\subsection{Lemmas on Bruhat edges in the quantum Bruhat graph (1).}
\label{subsec:tech}

Recall that $\vpi_{k}$ is minuscule and $\J=I \setminus \{k\}$. 
We know the following (see, e.g., \cite[Lemma 11.1.16]{G}). 
%
%%%%%%%%%%%%%%%
%%% prop:lw %%%
%%%%%%%%%%%%%%%
%
\begin{prop} \label{prop:lw}
The restriction of the Bruhat order to $\WJ$ agrees with 
the restriction of the left weak Bruhat order to $\WJ$. 
Namely, if $y,\,w \in \WJ$ satisfy $w \ge y$, then 
there exist a sequence $y=y_{0},\,y_{1},\,\dots,\,y_{p}=w$ of elements of $\WJ$ 
and a sequence $j_{1},\,j_{2},\,\dots,\,j_{p}$ of elements of $I$ such that 
$y_{q} = s_{j_{q}}y_{q-1}$ and 
$\ell(y_{q}) = \ell(y_{q-1}) + 1$ for all $1 \le q \le p$. 
\end{prop}
%
%%%%%%%%%%%%%%%
%%% lem:tiy %%%
%%%%%%%%%%%%%%%
%
\begin{lem} \label{lem:tiy}
For each $y \in \WJe$, there exist a unique 
$\yJ{y} \in \WJ$ and $\yJs{y} \in \WJs$ 
such that $y = \yJ{y}\yJs{y}s_{k}$ and 
$\ell(y) = \ell(\yJ{y}) + \ell(\yJs{y}) + \ell(s_{k})$. 
\end{lem}
%
%%%%%%%%%%%%%%%
%%% rem:tiy %%%
%%%%%%%%%%%%%%%
%
\begin{rem} \label{rem:tiy}
Keep the notation and setting of Lemma~\ref{lem:tiy}. 
We see that $\yJs{y}s_{k} \in \WJ$ and 
$\ell(\yJs{y}s_{k}) = \ell(\yJs{y})+\ell(s_{k})$. 
Also, it is easily verified that 
$\ell(y) = \ell(\yJ{y}) + \ell(\yJs{y}s_{k})$. 
\end{rem}

\begin{proof}[Proof of Lemma~\ref{lem:tiy}]
Let $z_{0} \in \WJs$ be a longest element for which 
$y = y_{0}z_{0}s_{k}$ for some $y_{0} \in W$ such that
$\ell(y) = \ell(y_{0}) + \ell(z_{0}) + \ell(s_{k})$. 
Suppose, for a contradiction, that $y_{0} \notin \WJ$. 
In this case, there exist $y_{0}' \in W$ and 
$j \in \J = I \setminus \{k\}$ such that $y_{0}=y_{0}'s_{j}$ and 
$\ell(y_{0})=\ell(y_{0}') + 1$. Hence we have 
$y = y_{0}'s_{j}z_{0}s_{k}$, with $\ell(s_{j}z_{0}) = \ell(z_{0}) + 1$ and 
$\ell(y) = \ell(y_{0}') + \ell(s_{j}z_{0}) + \ell(s_{k})$. 
Since $s_{j}z_{0} \in \WJs$, this contradicts 
the maximality of the length of $z_{0} \in \WJs$. Thus we obtain $y_{0} \in \WJ$, as desired.

Assume now that $y = y_{1}z_{1}s_{k}$ for some $y_{1} \in \WJ$ and $z_{1} \in \WJs$
such that $\ell(y) = \ell(y_{1}) + \ell(z_{1}) + \ell(s_{k})$. Then, 
$y_{1}\vpi_{k} = y_{1}z_{1}\vpi_{k} = ys_{k}\vpi_{k} = 
 y_{0}z_{0}\vpi_{k} = y_{0}\vpi_{k}$.
Since $y_{1},\,y_{0} \in \WJ$, we deduce that $y_{0}=y_{1}$. 
Hence, from the equalities $y_{1}z_{1}s_{k} = y = y_{0}z_{0}s_{k}$, 
it follows that $z_{1} = z_{0}$. This proves the lemma. 
\end{proof}

Recall that $\ls$ is the longest element of $\WJs$. 
We write $\mcr{\ls s_{k}} \in \WJ$ as 
$\mcr{\ls s_{k}} = z_{k}s_{k}$ for some $z_{k} \in W$ such that 
$\ell(z_{k}s_{k}) = \ell(z_{k}) + 1$. Since $\ls \in \WJs$, 
we see that $\ell(\ls s_{k}) = \ell(\ls) + 1$. 
Also, since $z_{k} s_{k} = \mcr{\ls s_{k}} \leq \ls s_{k}$, 
we deduce by the Subword Property for the Bruhat order 
(see, e.g., \cite[Theorem 2.2.2]{BB})
that $z_{k} \in \WJs$; in the notation of Lemma~\ref{lem:tiy}, 
we have $\yJ{\mcr{\ls s_{k}}}=e$ and $\yJs{\mcr{\ls s_{k}}}=z_{k}$. 
%
%%%%%%%%%%%%%%%%
%%% lem:zksk %%%
%%%%%%%%%%%%%%%%
%
\begin{lem} \label{lem:zksk}
Let $y \in \WJe$, and set $z:=\yJs{y} \in \WJs$. 
Then, $z_{k}s_{k} \ge zs_{k}$ in the notation above, 
where the equality holds if and only if $y \ge z_{k}s_{k}$. 
\end{lem}

\begin{proof}
Recall that $zs_{k} = \yJs{y}s_{k} \in \WJ$ (see Remark~\ref{rem:tiy}). 
Since $z \in \WJs$, we have $\ls \ge z$, and hence $\ls s_{k} \ge zs_{k}$. 
Thus we obtain $z_{k}s_{k} = \mcr{ \ls s_{k} } \ge \mcr{ zs_{k} } = zs_{k}$, 
as desired (see, e.g., \cite[Proposition~2.5.1]{BB}). 

We set $y' :=\yJ{y} \in \WJ$. If $z_{k}s_{k} = zs_{k}$, then 
it is obvious that $y = y'zs_{k} =y'z_{k}s_{k} \ge z_{k}s_{k}$,
which proves the ``only if'' part; 
recall that $\ell(y) = \ell(y') + \ell(zs_{k}) = 
\ell(y') + \ell(z_{k}s_{k})$. Assume now that $y \ge z_{k}s_{k}$. 
Since $y,\,z_{k}s_{k} \in \WJ$, and $y \ge z_{k}s_{k}$, 
it follows from Proposition~\ref{prop:lw} that 
there exist $z_{k}s_{k}=y_{0},\,y_{1},\,\dots,\,y_{p}=y \in \WJ$ and 
$j_{1},\,j_{2},\,\dots,\,j_{p} \in I$ such that 
$y_{q} = s_{j_{q}}y_{q-1}$ and 
$\ell(y_{q}) = \ell(y_{q-1}) + 1$ for all $1 \le q \le p$; 
notice that 
$\pair{y_{q-1}\vpi_{k}}{\alpha_{j_{q}}^{\vee}} > 0$ 
for $1 \le q \le p$, since $y_{q}=s_{j_{q}}y_{q-1} \in \WJ$ and
$\ell(y_{q}) = \ell(y_{q-1}) + 1$. In view of Lemma~\ref{lem:tiy}, 
it suffices to show that $s_{j_p} \cdots s_{j_2}s_{j_1} \in \WJ$. 
Suppose, for a contradiction, that 
there exists $1 \le q \le p$ such that 
$s_{j_{q-1}} \cdots s_{j_2}s_{j_1} \in \WJ$ and 
$s_{j_{q}} \cdots s_{j_2}s_{j_1} \not\in \WJ$. 
In this case, $(s_{j_{q-1}} \cdots s_{j_2}s_{j_1})^{-1}\alpha_{j_{q}} \in \DJs$. 
Therefore, we see that
\begin{align*}
\pair{y_{q-1}\vpi_{k}}{\alpha_{j_{q}}^{\vee}}
& = 
\pair{s_{j_{q-1}} \cdots s_{j_2}s_{j_1}\ls s_{k}\vpi_{k}}{\alpha_{j_{q}}^{\vee}} 
\quad \text{since $\mcr{\ls s_{k}} = z_{k}s_{k}$} \\
& = 
\Bpair{s_{k}\vpi_{k}}{\bigl( 
 \underbrace{\ls (s_{j_{q-1}} \cdots s_{j_2}s_{j_1})^{-1} \alpha_{j_{q}}}_{\in - \DJs} \bigr)^{\vee}} \\
& = 
\pair{\vpi_{k}}{\bigl( 
 \underbrace{s_{k} \ls (s_{j_{q-1}} \cdots s_{j_2}s_{j_1})^{-1} \alpha_{j_{q}}}_{
   \in - s_{k}\DJs \subset - \Delta^{+} } \bigr)^{\vee}} \le 0, 
\end{align*}
which is a contradiction. 
Thus we have proved the ``if'' part.
This proves the lemma. 
\end{proof}
%
%%%%%%%%%%%%%%
%%% lem:EB %%%
%%%%%%%%%%%%%%
%
\begin{lem} \label{lem:EB}
Let $y \in \WJe$, and set $z:=\yJs{y} \in \WJs$. 
If $y \not\ge z_{k}s_{k}$, or equivalently, $z_{k}s_{k} > zs_{k}$, 
then there exists $\beta \in \Inv(z_{k}s_{k}) \setminus \{\alpha_{k}\}$ 
such that $y \edge{\beta} ys_{\beta}$ is a Bruhat edge. 
\end{lem}
\begin{proof}
Since $z_{k}s_{k},\,zs_{k} \in \WJ$ (see Remark~\ref{rem:tiy}) and 
$z_{k}s_{k} > zs_{k}$, it follows from Proposition~\ref{prop:lw} that 
there exist $zs_{k}=y_{0},\,y_{1},\,\dots,\,y_{p}=z_{k}s_{k} \in \WJ$ and 
$j_{1},\,j_{2},\,\dots,\,j_{p} \in I$ such that 
$y_{q} = s_{j_{q}}y_{q-1}$ and 
$\ell(y_{q}) = \ell(y_{q-1}) + 1$ for all $1 \le q \le p$; 
note that $p \ge 1$. We claim that $j_{1} \in \J$. 
Indeed, we see that $z_{k}s_{k}\vpi_{k} = \vpi_{k} - z_{k} \alpha_{k}$, 
where $z_{k} \alpha_{k}$ is contained in 
$\alpha_{k} + \sum_{i \in \J} \BZ \alpha_{i}$. 
Also, we see that 
\begin{align*}
z_{k}s_{k}\vpi_{k} 
 & = y_{p}\vpi_{k} = y_{p-1}\vpi_{k} - \alpha_{j_{p}} 
   = y_{p-2}\vpi_{k} - (\alpha_{j_{p-1}} + \alpha_{j_{p}}) \\
 & = \cdots = y_{0}\vpi_{k} - \sum_{q=1}^{p} \alpha_{j_{q}}
   = zs_{k}\vpi_{k} - \sum_{q=1}^{p} \alpha_{j_{q}} \\
 & = \vpi_{k} - z\alpha_{k} - \sum_{q=1}^{p} \alpha_{j_{q}}, 
\end{align*}
where $z \alpha_{k}$ is contained in 
$\alpha_{k} + \sum_{i \in \J} \BZ \alpha_{i}$. 
Combining these, we deduce that 
$\sum_{q=1}^{p} \alpha_{j_{q}} \in \sum_{i \in \J} \BZ \alpha_{i}$, 
which implies that $j_{1},\,\dots,\,j_{p} \in \J$; in particular, 
we obtain $j_{1} \in \J$, as desired. 
Next, we claim that 
$\beta:=y_{0}^{-1}\alpha_{j_{1}} = (zs_{k})^{-1}\alpha_{j_{1}} \in 
\Inv(z_{k}s_{k})$. Indeed, 
since $\ell(s_{j_{1}}zs_{k}) = \ell(zs_{k})+1$, we have 
$\beta=(zs_{k})^{-1}\alpha_{j_{1}} \in \Delta^{+}$. 
Also, it follows that 
\begin{equation*}
z_{k}s_{k}\beta = s_{j_{p}} \cdots s_{j_{2}}s_{j_{1}}zs_{k}\beta
= s_{j_{p}} \cdots s_{j_{2}}s_{j_{1}}\alpha_{j_{1}}.
\end{equation*}
Since $s_{j_{p}} \cdots s_{j_{2}}s_{j_{1}}$ is reduced, 
we obtain $z_{k}s_{k}\beta \in \Delta^{-}$, and hence 
$\beta \in \Inv(z_{k}s_{k})$, as desired. 
Here, we see that $ys_{\beta} = \yJ{y}zs_{k}s_{\beta} = 
\yJ{y}s_{j_{1}}zs_{k}$, which implies that 
$\ell(ys_{\beta}) \le \ell(\yJ{y}) + \ell(s_{j_{1}}) + 
\ell(zs_{k})=\ell(y)+1$ (see Remark~\ref{rem:tiy}). 
Since $y\beta = \yJ{y}\alpha_{j_{1}}$, and 
$\yJ{y} \in \WJ$, $j_{1} \in \J$, 
it follows from \eqref{eq:mcr} that 
$y\beta \in \Delta^{+}$. Hence $\ell(ys_{\beta}) > \ell(y)$. 
Therefore, we find that $\ell(ys_{\beta}) = \ell(y) + 1$, 
which implies that $\beta \ne \alpha_{k}$ (recall that $y \ne e$), 
and hence $y \edge{\beta} ys_{\beta}$ is a Bruhat edge. 
This proves the lemma. 
\end{proof}
%
%%%%%%%%%%%%%%%%
%%% lem:invc %%%
%%%%%%%%%%%%%%%%
%
\begin{lem} \label{lem:invc}
Let $y \in \WJe$ and $\beta \in \DJp$ be such that 
$y \edge{\beta} ys_{\beta}$ is a Bruhat edge in $\QBG(W)$. 
If $y \not\ge z_{k}s_{k}$ and $ys_{\beta} \ge z_{k}s_{k}$, then 
$\beta \in \Inv(z_{k}s_{k}) \setminus \{\alpha_{k}\}$. 
\end{lem}

\begin{proof}
Notice that $ys_{\beta} \in \WJe$ by Lemma~\ref{lem:B}. 
We have $ys_{\beta} = \yJ{ys_{\beta}}\yJs{ys_{\beta}}s_{k}$ by Lemma~\ref{lem:tiy}. 
Since $ys_{\beta} \ge z_{k}s_{k}$, we see that $\yJs{ys_{\beta}} = z_{k}$ by Lemma~\ref{lem:zksk}. 
Since $ys_{\beta} \ge y$ and $\ell(ys_{\beta})=\ell(y) + 1$, 
it follows from the Subword Property 
for the Bruhat order that $y$ has a reduced expression obtained 
from a reduced expression of $ys_{\beta}$ by removing one simple reflection. 
Here we recall that 
$\ell(ys_{\beta}) = \ell(\yJ{ys_{\beta}}) + \ell(\yJs{ys_{\beta}}) + \ell(s_{k})$. 
Suppose, for a contradiction, that 
$y = w\yJs{ys_{\beta}}s_{k} = wz_{k}s_{k}$, where $w$ is obtained 
from a reduced expression of $\yJ{ys_{\beta}}$ by removing one simple reflection. 
In this case, since $\ell(y) = \ell(w) + \ell(z_{k}s_{k})$, it follows that 
$y \ge z_{k}s_{k}$, which contradicts the assumption. 
Since $y \in \WJe$, the rightmost simple reflection of 
any reduced expression of $y$ must be $s_{k}$. 
Hence we deduce that $y = \yJ{ys_{\beta}}zs_{k}$ and 
$\ell(y) = \ell(\yJ{ys_{\beta}}) + \ell(z) + \ell(s_{k})$, 
where $z$ is obtained from a reduced expression of $\yJs{ys_{\beta}} = z_{k}$ 
by removing one simple reflection. 
Let $z_{k}s_{k} = s_{i_{a}}s_{i_{a-1}} \cdots s_{i_{2}}s_{i_{1}}$
be a reduced expression of $z_{k}s_{k}$ (note that $i_{1}=k$), 
and assume that $zs_{k} = s_{i_{a}}s_{i_{a-1}} \cdots s_{i_{b+1}}s_{i_{b-1}} \cdots s_{i_{2}}s_{i_{1}}$ 
for some $2 \le b \le a$. In this case, we have
\begin{align*}
s_{\beta} & = y^{-1} (y s_{\beta}) 
= (\yJ{y s_{\beta}} z s_{k})^{-1} (\yJ{y s_{\beta}} z_{k} s_{k}) 
= (z s_{k})^{-1} (z_{k} s_{k}) \\
& = s_{i_{1}}s_{i_{2}} \cdots s_{i_{b-1}}s_{i_{b}}s_{i_{b-1}} \cdots s_{i_{2}}s_{i_{1}},
\end{align*}
which implies that 
$\beta = s_{i_{1}}s_{i_{2}} \cdots s_{i_{b-1}}\alpha_{i_{b}} 
\in \Inv(z_{k}s_{k}) \setminus \{\alpha_{k}\}$. This proves the lemma. 
\end{proof}
%
%%%%%%%%%%%%%%
%%% lem:yJ %%%
%%%%%%%%%%%%%%
%
\begin{lem} \label{lem:yJ}
Let $y \in \WJe$ and $\beta \in \Inv(z_{k}s_{k})$ be such that 
$y \edge{\beta} ys_{\beta}$ is a Bruhat edge in $\QBG(W)$; 
note that $ys_{\beta} \in \WJe$ by Lemma~\ref{lem:B}.
Then, $\yJ{y}=\yJ{ys_{\beta}}$. 
\end{lem}

\begin{proof}
Recall that $y = \yJ{y}\yJs{y}s_{k}$ and 
$ys_{\beta} = \yJ{ys_{\beta}}\yJs{ys_{\beta}}s_{k}$. 
Also, notice that $\beta \ne \alpha_{k}$ since $y \in \WJe$. 
Since $\beta \in \Inv(z_{k}s_{k}) \setminus \{\alpha_{k}\}$, and since
$z_{k} \in \WJs$ and $\ell(z_{k}s_{k}) = \ell(z_{k})+1$, 
we deduce that $\beta = s_{k}z\alpha_{j}$ for some $z \in \WJs$ and $j \in \J$. 
Hence $ys_{\beta} = \yJ{y}\yJs{y}s_{k}s_{\beta} = \yJ{y}\yJs{y}zs_{j}z^{-1} s_{k}$; 
notice that $\yJs{y}zs_{j}z^{-1} \in \WJs$. Therefore, we see that 
\begin{equation*}
\yJ{ys_{\beta}}\vpi_{k} = \yJ{ys_{\beta}}\yJs{ys_{\beta}}\vpi_{k} = ys_{\beta}s_{k}\vpi_{k} = 
\yJ{y}\underbrace{\yJs{y}zs_{j}z^{-1}}_{\in \WJs}\vpi_{k} = \yJ{y}\vpi_{k}. 
\end{equation*}
Since $\yJ{ys_{\beta}},\,\yJ{y} \in \WJ$, 
we deduce that $\yJ{ys_{\beta}}=\yJ{y}$. 
This proves the lemma. 
\end{proof}
%
%==============================%
%     START SUBSECTION 0403    %
%==============================%
%
\subsection{Demazure operators.}
\label{subsec:demazure}

\begin{dfn} \label{def:Demazure_op}
For $i \in I$, we define a $\BC\pra{q^{-1}}$-linear operator 
$\SD_i=\SD_i^{-}$ on $\BC\pra{q^{-1}}[P]$ as follows: for $\xi \in P$, 
\begin{equation*}
\begin{split}
\SD_i \be^{\xi} & := \frac{\be^{\xi} - \be^{\alpha_i}\be^{s_i \xi}}{1-\be^{\alpha_i}} =
\begin{cases}
\be^{\xi}(1 + \be^{\alpha_i} + \be^{2\alpha_i} + \cdots \be^{-\pair{\xi}{\alpha_i^\vee}\alpha_i}) 
  & \text{if } \pair{\xi}{\alpha_i^\vee} \leq 0, \\[2mm]
0 & \text{if } \pair{\xi}{\alpha_i^\vee} = 1, \\[2mm]
-\be^{\xi} (\be^{-\alpha_i} + \be^{-2\alpha_i} + \cdots + \be^{(-\pair{\xi}{\alpha_i^\vee} + 1)\alpha_i}) 
  & \text{if } \pair{\xi}{\alpha_i^\vee} \geq 2. 
\end{cases}
\end{split}
\end{equation*}
\end{dfn}

We can easily verify the following. 
%
%%%%%%%%%%%%%%%%%%%
%%% lem:Leibniz %%%
%%%%%%%%%%%%%%%%%%%
%
\begin{lem} \label{lem:Leibniz}
For $\lambda,\,\mu \in P$ and $i \in I$, it holds that
\begin{align*}
\SD_i(\be^\lambda \be^\mu) = (\SD_i \be^{\lambda + \rho}) \be^{\mu-\rho} + \be^{s_i\lambda} (\SD_i \be^\mu).
\end{align*}
\end{lem}

We know the following from 
\cite[Proposition 6.6 and Remark 6.7]{NOS}. 
%
%%%%%%%%%%%%%%%%%%%%%%%%%%%
%%% prop:Demazure_chara %%%
%%%%%%%%%%%%%%%%%%%%%%%%%%%
%
\begin{prop} \label{prop:Demazure_chara}
Let $x \in W$, $i \in I$, and $\lambda \in P^+$.
\begin{enu}
\item If $s_i x < x$, then $\SD_i \gch V_{x}^-(\lambda) = \gch V_{s_i x}^-(\lambda)$.
In particular, if $\pair{x\lambda}{\alpha_i^\vee} < 0$, 
then $\SD_i \gch V_{x}^-(\lambda) = \gch V_{s_i x}^-(\lambda)$. 

\item If $s_i x > x$, then $\SD_i \gch V_{x}^-(\lambda) = \gch V_{x}^-(\lambda)$.
In particular, if $\pair{x\lambda}{\alpha_i^\vee} \geq 0$, 
then $\SD_i \gch V_{x}^-(\lambda) = \gch V_{x}^-(\lambda)$.
\end{enu}
\end{prop}
%
%==============================%
%     START SUBSECTION 0404    %
%==============================%
%
\subsection{Recurrence relations for coefficients in the character identity of Chevalley type.}
\label{subsec:ind}

Recall that $\vpi_{k}$ is minuscule and 
$\J=\J_{\vpi_{k}} = I \setminus \{k\}$. 
By \eqref{eq:minuscule}, $\gch V_{x}^{-}((N-1)\vpi_{k})$ can be written as:
%
%%%%%%%%%%%%%%%
%%% eq:NOSa %%%
%%%%%%%%%%%%%%%
%
\begin{equation} \label{eq:NOSa}
\gch V_{x}^{-}((N-1)\vpi_{k}) = 
\be^{-x\vpi_{k}} \sum_{v \in \WJ} \sum_{m \in \BZ_{\ge 0}}
c^{x}_{v,m} \gch V_{ v t_{m \alpha_{k}^{\vee}} }(N\vpi_{k}), 
\end{equation}
where $c^{x}_{v,m} \in \BZ$, and $c^{x}_{v,m}=0$ 
for all but finitely many $(v,m) \in \WJ \times \BZ_{\ge 0}$. 
\begin{lem} \label{lem:+1}
Let $j \in I$ be such that $\pair{x\vpi_{k}}{\alpha_{j}^{\vee}} = 1$. 
It holds that $c^{x}_{v,m} = 0$ for all $v \in \WJ$ such that 
$\pair{v\vpi_{k}}{ \alpha_{j}^{\vee} } = 0$ and for all $m \in \BZ_{\ge 0}$.
%
\iffalse
%
\begin{enu}
\item It holds that $c^{x}_{v,m} = 0$ for all $v \in \WJ$ such that 
      $\pair{v\vpi_{k}}{ \alpha_{j}^{\vee} } = 0$ and $m \in \BZ_{\ge 0}$.
\item It holds that $c^{x}_{v,m}+c^{x}_{s_{j}v,m}=0$ for all $v \in \WJ$ 
      such that $\pair{v\vpi_{k}}{ \alpha_{j}^{\vee} } > 0$ 
      (notice that $s_{j}v \in \WJ$) and $m \in \BZ_{\ge 0}$, 
      where we set $c^{x}_{w,-1}:=0$ for all $w \in \WJ$. 
\end{enu}
%
\fi
%
\end{lem}

\begin{proof}
For simplicity of notation, 
we set $c_{v,m}:=c^{x}_{v,m}$. 
Let $j \in I$ be such that $\pair{x\vpi_{k}}{\alpha_{j}^{\vee}} = 1$. 
By Proposition~\ref{prop:Demazure_chara}\,(2), 
we have
\begin{equation*}
\SD_{j}\underbrace{\gch V_{x}^{-}((N-1)\vpi_{k})}_{\text{(LHS) of \eqref{eq:NOSa}}} 
 = \gch V_{x}^{-}((N-1)\vpi_{k}).
\end{equation*}
Also, we see that
\begin{align*}
& \SD_{j} \underbrace{%
  \left( \be^{-x\vpi_{k}} \sum_{v \in \WJ}
  \sum_{m \in \BZ_{\ge 0}} c_{v,m} \gch V_{ vt_{m\alpha_{k}^{\vee}} }(N\vpi_{k}) \right)%
  }_{\text{(RHS) of \eqref{eq:NOSa}}} \\[5mm]
& \hspace*{10mm} = \be^{-x\vpi_{k}} 
  \sum_{v \in \WJ} \sum_{m \in \BZ_{\ge 0}} c_{v,m} \gch V_{ vt_{m\alpha_{k}^{\vee}} }(N\vpi_{k}) \\[3mm]
& \hspace*{15mm} + \be^{-s_{j}x\vpi_{k}} 
  \sum_{v \in \WJ} \sum_{m \in \BZ_{\ge 0}} c_{v,m} 
  \SD_{j}\Bigl( \gch V_{ vt_{m\alpha_{k}^{\vee}} }(N\vpi_{k}) \Bigr) 
  \quad \text{by Lemma~\ref{lem:Leibniz}}\\[5mm]
& \hspace*{10mm} = 
  \overbrace{
  \be^{-x\vpi_{k}} 
  \sum_{v \in \WJ} \sum_{m \in \BZ_{\ge 0}} c_{v,m} \gch V_{ vt_{m\alpha_{k}^{\vee}} }(N\vpi_{k})}^{%
  = \gch V_{x}^{-}((N-1)\vpi_{k}) } \\[3mm]
& \hspace*{15mm} + \be^{-s_{j}x\vpi_{k}} \sum_{ 
  \begin{subarray}{c} v \in \WJ \\ \pair{v\vpi_{k}}{\alpha_{j}^{\vee}} > 0 \end{subarray}}
  \sum_{m \in \BZ_{\ge 0}} c_{v,m} \gch V_{ vt_{m\alpha_{k}^{\vee}} }(N\vpi_{k}) \\[3mm]
& \hspace*{15mm} + \be^{-s_{j}x\vpi_{k}} \sum_{ 
  \begin{subarray}{c} v \in \WJ \\ \pair{v\vpi_{k}}{\alpha_{j}^{\vee}} = 0 \end{subarray}}
  \sum_{m \in \BZ_{\ge 0}} c_{v,m} \gch V_{ vt_{m\alpha_{k}^{\vee}} }(N\vpi_{k}) \\[3mm]
& \hspace*{15mm} + \be^{-s_{j}x\vpi_{k}} \sum_{
  \begin{subarray}{c} v \in \WJ \\ \pair{v\vpi_{k}}{ \alpha_{j}^{\vee} } < 0 \end{subarray}}
  \sum_{m \in \BZ_{\ge 0}} c_{v,m} \gch V_{ s_{j}vt_{m\alpha_{k}^{\vee}} }(N\vpi_{k})
  \quad \text{by Proposition~\ref{prop:Demazure_chara}}; 
\end{align*}
notice that $s_{j}v \in \WJ$ for $v \in \WJ$ such that 
$\pair{v\vpi_{k}}{ \alpha_{j}^{\vee} } < 0$. Therefore, we obtain 
\begin{align*}
& \sum_{ 
  \begin{subarray}{c} v \in \WJ \\ \pair{v\vpi_{k}}{\alpha_{j}^{\vee}} > 0 \end{subarray}}
  \sum_{m \in \BZ_{\ge 0}} c_{v,m} \gch V_{ vt_{m\alpha_{k}^{\vee}} }(N\vpi_{k}) \\[3mm]
& + \sum_{ 
  \begin{subarray}{c} v \in \WJ \\ \pair{v\vpi_{k}}{\alpha_{j}^{\vee}} = 0 \end{subarray}}
  \sum_{m \in \BZ_{\ge 0}} c_{v,m} \gch V_{ vt_{m\alpha_{k}^{\vee}} }(N\vpi_{k}) \\[3mm]
& + \sum_{
  \begin{subarray}{c} v \in \WJ \\ \pair{v\vpi_{k}}{ \alpha_{j}^{\vee} } < 0 \end{subarray}}
  \sum_{m \in \BZ_{\ge 0}} c_{v,m} \gch V_{ s_{j}v t_{m\alpha_{k}^{\vee}} }(N\vpi_{k}) = 0. 
\end{align*}
Because the graded characters $\gch V_{ wt_{m\alpha_{k}^{\vee}} }(N\vpi_{k})$ 
for $(w,m) \in \WJ \times \BZ_{\ge 0}$ are linearly independent 
(note that all the sums on the left-hand side of the equation above are finite sums), 
it follows that $c_{v,m} = 0$ for all $v \in \WJ$ such that 
$\pair{v\vpi_{k}}{ \alpha_{j}^{\vee} } = 0$ and for all $m \in \BZ_{\ge 0}$. 
This proves the lemma.
\end{proof}
%
%%%%%%%%%%%%%%
%%% lem:-1 %%%
%%%%%%%%%%%%%%
%
\begin{lem} \label{lem:-1}
Let $j \in I$ be such that $\pair{x\vpi_{k}}{\alpha_{j}^{\vee}} = - 1$; 
notice that $s_{j}x \in \WJ$. 
\begin{enu}
\item It holds that $c^{s_{j}x}_{v,m} = -c^{x}_{v,m}$ for all $v \in \WJ$ 
      such that $\pair{v\vpi_{k}}{\alpha_{j}^{\vee}} < 0$ and for all $m \in \BZ_{\ge 0}$. 
\item It holds that $c^{s_{j}x}_{v,m} = c^{x}_{s_{j}v,m}$ 
      for all $v \in \WJ$ such that $\pair{v\vpi_{k}}{\alpha_{j}^{\vee}} > 0$ 
      (note that $s_{j}v \in \WJ$) and for all $m \in \BZ_{\ge 0}$.
\end{enu}
\end{lem}

\begin{proof}
Let $j \in I$ be such that $\pair{x\vpi_{k}}{\alpha_{j}^{\vee}} = -1$. 
We see by Proposition~\ref{prop:Demazure_chara}\,(1) and Lemma~\ref{lem:Leibniz} that 
\begin{equation*}
\SD_{j}\Bigl( \be^{x\vpi_{k}} 
 \underbrace{\gch V_{x}^{-}((N-1)\vpi_{k})}_{\text{(LHS) of \eqref{eq:NOSa}}} \Bigr) = 
\be^{x\vpi_{k}}\gch V_{x}^{-}((N-1)\vpi_{k}) + 
\be^{s_{j}x\vpi_{k}}\gch V_{s_{j}x}^{-}((N-1)\vpi_{k}). 
\end{equation*}
Also, by Proposition~\ref{prop:Demazure_chara}, we deduce that
\begin{align*}
& \SD_{j} \underbrace{%
 \left( \sum_{v \in \WJ} \sum_{m \in \BZ_{\ge 0}} 
 c^{x}_{v,m} \gch V_{ vt_{m\alpha_{k}^{\vee}} }(N\vpi_{k}) \right)%
 }_{ \text{(RHS) of \eqref{eq:NOSa} multiplied by $\be^{x\vpi_{k}}$} } \\[5mm]
& = \sum_{ 
  \begin{subarray}{c} v \in \WJ \\ \pair{v\vpi_{k}}{\alpha_{j}^{\vee}} \ge 0 \end{subarray}}
  \sum_{m \in \BZ_{\ge 0}} c^{x}_{v,m} \gch V_{ vt_{m\alpha_{k}^{\vee}} }(N\vpi_{k}) + 
\sum_{ 
  \begin{subarray}{c} v \in \WJ \\ \pair{v\vpi_{k}}{\alpha_{j}^{\vee}} < 0 \end{subarray}}
  \sum_{m \in \BZ_{\ge 0}} c^{x}_{v,m} \gch V_{ s_{j}vt_{m\alpha_{k}^{\vee}} }(N\vpi_{k}); 
\end{align*}
note that $s_{j}v \in \WJ$ for $v \in \WJ$ such that 
$\pair{v\vpi_{k}}{ \alpha_{j}^{\vee} } < 0$. Therefore, we obtain 
\begin{align*}
& \be^{s_{j}x\vpi_{k}}\gch V_{s_{j}x}^{-}((N-1)\vpi_{k}) \\
& \hspace*{10mm} = -  
\sum_{ 
  \begin{subarray}{c} v \in \WJ \\ \pair{v\vpi_{k}}{\alpha_{j}^{\vee}} < 0 \end{subarray}}
  \sum_{m \in \BZ_{\ge 0}} c^{x}_{v,m} \gch V_{ vt_{m\alpha_{k}^{\vee}} }(N\vpi_{k}) \\[3mm]
& \hspace*{30mm} + 
\sum_{ 
  \begin{subarray}{c} v \in \WJ \\ \pair{v\vpi_{k}}{\alpha_{j}^{\vee}} < 0 \end{subarray}}
  \sum_{m \in \BZ_{\ge 0}} c^{x}_{v,m} \gch V_{ s_{j}v t_{m\alpha_{k}^{\vee}} }(N\vpi_{k});
\end{align*}
here, observe that the left-hand side of this equation is identical to 
\begin{equation*}
\be^{s_{j}x\vpi_{k}}\gch V_{s_{j}x}^{-}((N-1)\vpi_{k}) = 
\sum_{v \in \WJ} 
\sum_{m \in \BZ_{\ge 0}} c_{v,m}^{ s_{j}x } \gch V_{ v t_{m\alpha_{k}^{\vee}} }(N\vpi_{k}). 
\end{equation*}
Hence we obtain the equalities in parts (1) and (2), as desired. 
This proves the lemma.
\end{proof}
%
%=========================%
%     START SECTION 05    %
%=========================%
%
\section{Proof of Theorem~\ref{thm:main} in simply-laced types.}
\label{sec:prfa}

In this section, we assume that $\Fg$ is simply-laced. 
As in Section~\ref{sec:main}, 
let $k \in I$ be such that $\vpi_{k}$ is minuscule, and set 
$\J=\J_{\vpi_{k}} = I \setminus \{k\}$. 
We may assume that $x \ne e$ by Proposition~\ref{prop:x=e}. 
%
%==============================%
%     START SUBSECTION 0501    %
%==============================%
%
\subsection{Quantum edges in the quantum Bruhat graph (1).}
\label{subsec:qe1}
%
%%%%%%%%%%%%%%
%%% lem:Q1 %%%
%%%%%%%%%%%%%%
%
\begin{lem} \label{lem:Q1}
Let $y \in \WJ$ and $\gamma \in \DJp$. 
We have a quantum edge $y \edge{\gamma} ys_{\gamma}$ in $\QBG(W)$ 
if and only if $y \ne e$ and $\gamma=\alpha_{k}=\gq$ (see \eqref{eq:gamq}). 
\end{lem}

\begin{proof}
If $y \in \WJe$, then there exists a reduced expression of $y$ 
whose rightmost simple reflection is $s_{k}$. 
This fact immediately implies the ``if'' part. 
Let us show the ``only if'' part. Assume that $\gamma \ne \alpha_{k}$. 
Then, $s_{\gamma}$ has a reduced expression of the form
$s_{\gamma} = \cdots s_{p}s_{q}s_{p} \cdots$ for some $p,\,q \in I$, with $p \ne q$, 
such that $s_{p}s_{q}s_{p} = s_{q}s_{p}s_{q}$. 
From the equalities $\ell(ys_{\gamma}) = \ell(y) - \pair{2\rho}{\gamma^{\vee}} + 1 = 
\ell(y) - \ell(s_{\gamma})$ (see Remark~\ref{rem:qe1}), 
we see that if $ys_{\gamma} = s_{j_1}s_{j_2} \cdots s_{j_s}$ is 
a reduced expression of $ys_{\gamma}$, then 
\begin{equation*}
y = \underbrace{s_{j_1}s_{j_2} \cdots s_{j_s}}_{=ys_{\gamma}} 
    \underbrace{\cdots s_{p}s_{q}s_{p} \cdots}_{ = s_{\gamma} }
\end{equation*}
is a reduced expression of $y$. However, 
this contradicts the fact that every element in $\WJ$ is 
fully commutative (see \cite[Proposition 11.1.1\,(i)]{G}). 
Thus we have shown that $\gamma = \alpha_{k}$. 
This proves the lemma. 
\end{proof}
%
%==============================%
%     START SUBSECTION 0502    %
%==============================%
%
\subsection{Sets of label-increasing directed paths (1).}
\label{subsec:li1}

Let $\lhd$ be a reflection order on $\Delta^{+}$ 
satisfying condition \eqref{eq:ro}; remark that 
$\alpha_{k} \in \DJp$ is the largest element in $\Delta^{+}$ 
with respect to $\lhd$. Let $x \in \WJe$. 
Recall the notation $\bBG{x}$ and $\bQBG{x}$ from Section~\ref{subsec:QBG}; 
remark that $\ed(\bp) \ne e$ for any $\bp \in \bBG{x}$. 
For each $\bp \in \bBG{x}$, we define $\SE^{\SQ}_{\alpha_{k}}(\bp)$ 
to be the concatenation $\bp \edge{\alpha_{k}} \ed(\bp)s_{k}$ 
of the directed path $\bp$ with the quantum edge $\ed(\bp) \edge{\alpha_{k}} \ed(\bp)s_{k}$ 
(see Lemma~\ref{lem:Q1}). We see that 
$\SE^{\SQ}_{\alpha_{k}}(\bp) \in \bQBG{x}$, and 
%
%%%%%%%%%%%%%%%%
%%% eq:bQBG1 %%%
%%%%%%%%%%%%%%%%
%
\begin{equation} \label{eq:bQBG1}
\bQBG{x} = \bBG{x} \sqcup \bigl\{\SE^{\SQ}_{\alpha_{k}}(\bp) \mid \bp \in \bBG{x} \bigr\};
\end{equation}
note that $\wt(\bp)=0$ and $\wt(\SE^{\SQ}_{\alpha_{k}}(\bp)) = \alpha_{k}^{\vee}$
for all $\bp \in \bBG{x}$. Therefore, it follows from \eqref{eq:minuscule} that 
for all $N \in \BZ_{\ge 1}$, 
%
%%%%%%%%%%%%%%%%
%%% eq:NOS3a %%%
%%%%%%%%%%%%%%%%
%
\begin{equation} \label{eq:NOS3a}
\begin{split}
\gch V_{x}^{-}((N-1)\vpi_{k}) & = 
\be^{-x\vpi_{k}}
\sum_{ y \in \ed(\bBG{x}) } (-1)^{\ell(y)-\ell(x)} 
\gch V_{y}^{-}(N\vpi_{k}) \\[2mm]
& \hspace*{10mm} + 
\be^{-x\vpi_{k}}
\sum_{ y \in \ed(\bBG{x}) } (-1)^{\ell(y)-\ell(x)+1} 
\gch V_{\mcr{ys_{k}} t_{\alpha_{k}^{\vee}}}^{-}(N\vpi_{k}), 
\end{split}
\end{equation}
which proves the character identity \eqref{eq:main_1} 
in Theorem \ref{thm:main}\,(1) in simply-laced types. 
%
%==============================%
%     START SUBSECTION 0503    %
%==============================%
%
\subsection{Cancellations in equation \eqref{eq:NOS3a}.}
\label{subsec:can1}

Let $x \in \WJe$. We set 
\begin{equation}
\bG_{x}^{\lhd} := \bigl\{ \SE^{\SQ}_{\alpha_{k}}(\bp) \mid \bp \in \bBG{x} \bigr\} =
\bQBG{x} \setminus \bBG{x},
\end{equation}
and then $\bGx{x}{v}:=\bigl\{ \bq \in \bG_{x}^{\lhd} \mid \mcr{\ed(\bq)} = v \bigr\}$ 
for $v \in \WJ$. By \eqref{eq:NOS3a}, it follows that 
%
%%%%%%%%%%%%%%%%%
%%% eq:NOS3aa %%%
%%%%%%%%%%%%%%%%%
%
\begin{equation} \label{eq:NOS3aa}
\begin{split}
\gch V_{x}^{-}((N-1)\vpi_{k}) & = 
\be^{-x\vpi_{k}}
\sum_{ y \in \ed(\bBG{x}) } (-1)^{\ell(y)-\ell(x)} 
\gch V_{y}^{-}(N\vpi_{k}) \\ 
& \hspace*{5mm} + 
\be^{-x\vpi_{k}} \sum_{v \in \WJ} 
\underbrace{%
\left( \sum_{ \bq \in \bGx{x}{v} } (-1)^{\ell(\ed(\bq))-\ell(x)} \right)}_{%
\text{$=c_{v,1}^{x}$; see \eqref{eq:NOSa}} }
\gch V_{vt_{\alpha_{k}^{\vee}}}^{-}(N\vpi_{k}). 
\end{split}
\end{equation}
%
%%%%%%%%%%%%%%%
%%% lem:itv %%%
%%%%%%%%%%%%%%%
%
\begin{lem} \label{lem:itv}
Keep the notation and setting above. 
If $\# \bGx{x}{v} \ge 2$, then 
\begin{equation*}
c_{v,1}^{x} = \sum_{ \bq \in \bGx{x}{v} } (-1)^{\ell(\ed(\bq))-\ell(x)} = 0.
\end{equation*}
\end{lem}

\begin{proof}
Suppose, for a contradiction, that the assertion is false. 
Let $x$ be a maximal element (with respect to the Bruhat order) of the set 
\begin{equation*}
\bigl\{ w \in \WJe \mid 
\text{$\# \bGx{w}{v} \ge 2$ and $c_{v,1}^{w} \ne 0$ for some $v \in \WJ$} \bigr\};
\end{equation*}
since $\bG_{ \mcr{\lng} }^{\lhd} = \bigl\{ \mcr{\lng} \edge{\alpha_{k}} \mcr{\lng}s_{k} \bigr\}$, 
it follows that $x \ne \mcr{\lng}$. 
Take $j \in I$ such that $\pair{x\vpi_{k}}{\alpha_{j}^{\vee}} = 1 > 0$
(recall that $\vpi_{k}$ is minuscule); note that $s_{j}x \in \WJ$ and 
$\ell(s_{j}x) = \ell(x)+1$.
Let $v \in \WJ$ be such that $\# \bGx{x}{v} \ge 2$ and 
$c_{v,1}^{x} \ne 0$. By Lemma \ref{lem:+1}\,(1), 
we have $\pair{v\vpi_{k}}{\alpha_{j}^{\vee}} \ne 0$; 
note that $s_{j}v \in \WJ$. 
%
%%%%%%%%%%%%%%%%%%%%%%%
\paragraph{\bf Case 1.}
%%%%%%%%%%%%%%%%%%%%%%%
%
Assume that $\pair{v\vpi_{k}}{\alpha_{j}^{\vee}} > 0$. 
We define an injective map $\bGx{x}{v} \rightarrow \bGx{s_{j}x}{s_{j}v}$, 
$\bq \mapsto \ha{\bq}$, as follows: for $\bq \in \bGx{x}{v}$ with $y:=\ed(\bp)$, 
we define $\ha{\bq}$ to be the label-increasing 
(shortest) directed path from $s_{j}x$ to $s_{j}y$ in $\QBG(W)$ 
(see Theorem~\ref{thm:LI}). We claim that $\ha{\bq} \in \bGx{s_{j}x}{s_{j}v}$. 
Indeed, recall that $\bq$ is of the form:
%
%%%%%%%%%%%%%
%%% eq:qx %%%
%%%%%%%%%%%%%
%
\begin{equation} \label{eq:qx}
\bq:\underbrace{%
  x = y_{0} \edge{\gamma_{1}} y_{1} \edge{\gamma_{2}} \cdots 
  \edge{\gamma_{s}} y_{s}}_{\in \bBG{x}}
  \underbrace{\edge{\gamma_{s+1}=\alpha_{k}} y_{s+1}}_{\text{quantum edge}} = \ed(\bq) = y; 
\end{equation}
note that $\mcr{y} = v$.
Since $\pair{y\vpi_{k}}{\alpha_{j}^{\vee}} = 
\pair{v\vpi_{k}}{\alpha_{j}^{\vee}} > 0$ by the assumption in Case 1, 
we have $y^{-1}\alpha_{j} \in \Delta^{+}$. 
Similarly, we have $x^{-1}\alpha_{j} \in \Delta^{+}$. 
If $y_{u}^{-1}\alpha_{j} \in \Delta^{+}$ for all $1 \le u \le s$, 
then we see by Lemma~\ref{lem:DL}\,(2) that 
there exists a directed path $\bq'$ in $\QBG(W)$ 
from $s_{j}x$ to $s_{j}y$ of the following form: 
\begin{equation*}
\bq' : s_{j} x = s_{j}y_{0} \edge{\gamma_{1}} s_{j}y_{1} \edge{\gamma_{2}} \cdots 
  \edge{\gamma_{s+1}} s_{j}y_{s+1} = s_{j}y, 
\end{equation*}
with $\wt(\bq') = \wt (\bq) \ne 0$. Observe that 
$\bq' \in \bQBG{s_{j}x} \setminus \bBG{s_{j}x} = \bG_{s_{j}x}^{\lhd}$, 
and $\mcr{ \ed(\bq') } = \mcr{s_{j}y} = \mcr{s_{j}v} = s_{j}v$. 
Hence we obtain $\bq' \in \bGx{s_{j}x}{s_{j}v}$. 
Moreover, by the uniqueness of a label-increasing 
directed path from $s_{j}x$ to $s_{j}y$, we deduce that $\ha{\bq} = \bq'$, 
and hence $\ha{\bq} \in \bGx{s_{j}x}{s_{j}v}$ in this case. 

Assume now that $y_{u}^{-1}\alpha_{j} \in \Delta^{-}$ for some $1 \le u \le s$; 
remark that $s \ge 1$ in this case, since $y_{0}^{-1}\alpha_{j} \in \Delta^{+}$ and 
$y_{s+1}^{-1}\alpha_{j} \in \Delta^{+}$. If we set
$a:=\min \bigl\{ 1 \le u \le s \mid y_{u}^{-1}\alpha_{j} \in \Delta^{-} \bigr\}$, 
then we deduce from Lemma~\ref{lem:DL} that 
$\gamma_{a}=y_{a-1}^{-1}\alpha_{j}$, and 
that there exists a directed path $\bq''$ in $\QBG(W)$ from 
$s_{j}x$ to $y = \ed(\bp)$ of the following form: 
\begin{equation*}
\bq'': s_{j}x = s_{j}y_{0} \edge{\gamma_{1}} \cdots 
  \edge{\gamma_{a-1}} s_{j}y_{a-1}=y_{a} \edge{\gamma_{a+1}} \cdots
  \edge{\gamma_{s+1}} y_{s+1} = y; 
\end{equation*}
notice that $\bq'' \in \bQBG{s_{j}x}$. 
Here, since $x^{-1}\alpha_{j} \in \Delta^{+}$ and 
$y^{-1}\alpha_{j} \in \Delta^{+}$, it follows from 
\cite[Lemma~7.7\,(4)]{LNSSS1} that 
$\ell(\ha{\bq}) = \ell(s_{j}x \Rightarrow s_{j}y) = 
 \ell(x \Rightarrow y) = \ell(\bq) = s+1 \ge 2$, and 
$\wt(\ha{\bq}) = \wt(s_{j}x \Rightarrow s_{j}y) = 
 \wt(x \Rightarrow y) = \wt(\bq) \ne 0$. 
Let us write $\ha{\bq}$ as:
\begin{equation*}
\ha{\bq} : s_{j}x = x_{0} \edge{\beta_{1}} x_{1} \edge{\beta_{2}} \cdots 
  \edge{\beta_{s+1}} x_{s+1} = s_{j}y,
\end{equation*}
where $\beta_{1} \lhd \beta_{2} \lhd \cdots \lhd \beta_{s+1}$. 
We will show that $\beta_{1} \in \DJp$.
Notice that $x_{0}^{-1}\alpha_{j} \in \Delta^{-}$ and 
$x_{s+1}^{-1}\alpha_{j} \in \Delta^{-}$. 
Suppose, for a contradiction, that 
$x_{u}^{-1}\alpha_{j} \in \Delta^{-}$ for all $1 \le u \le s$.
Then, we see by Lemma~\ref{lem:DL}\,(2) that 
there exists a directed path $\ha{\bq}'$ in $\QBG(W)$ 
from $x$ to $y$ of the following form: 
\begin{equation*}
\ha{\bq}' : x = s_{j}x_{0} \edge{\beta_{1}} s_{j}x_{1} \edge{\beta_{2}} \cdots 
  \edge{\beta_{s+1}} s_{j}x_{s+1} = y. 
\end{equation*}
By the uniqueness of a label-increasing directed path from $x$ to $y$, 
we deduce that $\ha{\bq}'=\bq$; in particular, $s_{j}x_{a} = y_{a}$. 
However, $\Delta^{+} \ni (s_{j}x_{a})^{-1}\alpha_{j} = y_{a}^{-1}\alpha_{j} \in \Delta^{-}$, 
which is a contradiction. Thus there exists $1 \le u \le s$ such that 
$x_{u}^{-1}\alpha_{j} \in \Delta^{+}$. 
If we set $b:=\max \bigl\{1 \le u \le s \mid x_{u}^{-1}\alpha_{j} \in \Delta^{+} \bigr\}$, 
then we see by Lemma~\ref{lem:DL} that 
there exists a directed path $\ha{\bq}''$ in $\QBG(W)$ 
from $s_{j}x$ to $y$ of the following form: 
\begin{equation*}
\ha{\bq}'' : s_{j}x = x_{0} \edge{\beta_{1}}  \cdots 
  \edge{\beta_{b}} x_{b}=s_{j}x_{b+1} \edge{\beta_{b+2}} 
  \cdots \edge{\beta_{s+1}} s_{j}x_{s+1}  = y. 
\end{equation*}
By the uniqueness of a label-increasing directed path from $s_{j}x$ to $y$, 
we deduce that $\ha{\bq}'' = \bq''$. Hence $\beta_{1}$ is either $\gamma_{1}$ (if $a \ge 2$) 
or $\gamma_{2}$ (if $a=1$). Thus we obtain $\beta_{1} \in \DJp$, as desired. 
Since the reflection order $\lhd$ satisfies condition \eqref{eq:ro}, 
it follows that $\beta_{u} \in \DJp$ for all $1 \le u \le s+1$, 
which implies that $\ha{\bq} \in \bQBG{s_{j}x}$. 
Also, since $\wt(\ha{\bq}) \ne 0$ as seen above, 
we find that $\ha{\bq} \notin \bBG{s_{j}x}$, 
and hence $\ha{\bq} \in \bG_{s_{j}x}^{\lhd}$. It is easily seen that 
$\mcr{\ed(\ha{\bq})} = \mcr{s_{j}y} = \mcr{s_{j}v} = s_{j}v$. 
Therefore, it follows that $\ha{\bq} \in \bGx{s_{j}x}{s_{j}v}$. 

It remains to show that the map $\bGx{x}{v} \rightarrow \bGx{s_{j}x}{s_{j}v}$, 
$\bq \mapsto \ha{\bq}$, is injective. 
Let $\bq_{1},\,\bq_{2} \in \bGx{x}{v}$, with $\bq_{1} \ne \bq_{2}$. 
Note that $\ed(\bq_{1}) \ne \ed(\bq_{2})$ by Remark~\ref{rem:bBGy}\,(1). 
Since $\ed(\ha{\bq}_{1}) = s_{j}\ed(\bq_{1})$ and 
$\ed(\ha{\bq}_{2}) = s_{j}\ed(\bq_{2})$ by the definitions above, 
we deduce that $\ed(\ha{\bq}_{1}) \ne \ed(\ha{\bq}_{2})$. 
Hence, by Remark~\ref{rem:bBGy}\,(1), 
it follows that $\ha{\bq}_{1} \ne \ha{\bq}_{2}$, as desired. 
By the injectivity of the map above, 
we obtain $\# \bGx{s_{j}x}{s_{j}v} \ge \# \bGx{x}{v} \ge 2$. 
Also, by the maximality of $x$, we have $c^{s_{j}x}_{s_{j}v,1}=0$. 
However, by Lemma~\ref{lem:-1}\,(3), we see that 
$c^{x}_{v,1} = c^{s_{j}x}_{s_{j}v,1} = 0$, 
which contradicts the assumption that $c^{x}_{v,1} \ne 0$. 

%%%%%%%%%%%%%%%%%%%%%%%
\paragraph{\bf Case 2.}
%%%%%%%%%%%%%%%%%%%%%%%
%
Assume that $\pair{v\vpi_{k}}{\alpha_{j}^\vee} < 0$. 
We define an injective map $\bGx{x}{v} \rightarrow \bGx{s_{j}x}{v}$, 
$\bq \mapsto \ha{\bq}$, as follows. 
Assume that $\bq \in \bGx{x}{v}$ is of the form \eqref{eq:qx}. 
Note that $x^{-1}\alpha_{j} \in \Delta^{+}$ and 
$y^{-1}\alpha_{j} \in \Delta^{-}$ in this case. If we set
$a:=\min \bigl\{ 1 \le u \le s+1 \mid y_{u}^{-1}\alpha_{j} \in \Delta^{-} \bigr\}$, 
then it follows from Lemma~\ref{lem:DL} that 
there exists a directed path in $\QBG(W)$ from $s_{j}x$ to $y$ of the form:
\begin{equation*}
\ha{\bq} : 
  s_{j}x = s_{j}y_{0} \edge{\gamma_{1}} \cdots 
  \edge{\gamma_{a-1}} s_{j}y_{a-1}=y_{a} \edge{\gamma_{a+1}} \cdots
  \edge{\gamma_{s}} y_{s} \edge{\gamma_{s+1}} \cdots \edge{\gamma_{s+1}} y_{s+1} = y, 
\end{equation*}
with $\wt(\ha{\bq}) = \wt (\bq) \ne 0$. 
Observe that $\ha{\bq} \in \bQBG{s_{j}x} \setminus \bBG{s_{j}x}=\bG_{x}^{\lhd}$, 
and $\mcr{ \ed(\ha{\bq}) } = \mcr{ y } = v$. 
Hence we have $\ha{\bq} \in \bGx{s_{j}x}{v}$. 
By the same argument as in Case 1, we can show that 
the map $\bGx{x}{v} \rightarrow \bGx{s_{j}x}{v}$, 
$\bq \mapsto \ha{\bq}$, is injective. Therefore, 
we obtain $\# \bGx{s_{j}x}{v} \ge \# \bGx{x}{v} \ge 2$. 
We see by Lemma \ref{lem:-1}\,(2) and 
the maximality of $x$ that $c^{x}_{v,1} = - c^{s_{j}x}_{v,1} = 0$. 
However, this contradicts the assumption that $c^{x}_{v,1} \ne 0$. 
This completes the proof of the lemma. 
\end{proof}

Let $x \in \WJe$. 
For $v \in \WJ$ such that $\# \bGx{x}{v}=1$, 
let $\bq_{x,v}$ denote the (unique) element of $\bGx{x}{v}$. 
By Lemma~\ref{lem:itv} and \eqref{eq:NOS3aa}, we deduce that
%
%%%%%%%%%%%%%%%%%
%%% eq:NOS3ab %%%
%%%%%%%%%%%%%%%%%
%
\begin{equation} \label{eq:NOS3ab}
\begin{split}
\gch V_{x}^{-}((N-1)\vpi_{k}) & = 
\be^{-x\vpi_{k}}
\sum_{ y \in \ed(\bBG{x}) } (-1)^{\ell(y)-\ell(x)} 
\gch V_{y}^{-}(N\vpi_{k}) \\ 
& \hspace*{5mm} + 
\be^{-x\vpi_{k}} \sum_{ 
 \begin{subarray}{c} v \in \WJ \\[1mm] 
 \# \bGx{x}{v}=1  \end{subarray} } 
(-1)^{\ell(\ed(\bq_{x,v}))-\ell(x)}
\gch V_{v t_{\alpha_{k}^{\vee}}}^{-}(N\vpi_{k}).
\end{split}
\end{equation}
In order to prove the assertion on cancellations (following \eqref{eq:main_1}) 
and the character identity \eqref{eq:main_2} in Theorem~\ref{thm:main}\,(2) (in simply-laced types), 
it suffices to show the following proposition; its proof is given in the next subsection. 
%
%%%%%%%%%%%%%%%%%
%%% prop:fin1 %%%
%%%%%%%%%%%%%%%%%
%
\begin{prop} \label{prop:fin1}
Let $x \in \WJe$. 
\begin{enu}
\item If $x \ge \mcr{s_{\theta}}$, then $\# \bGx{x}{v}=0$ or $1$ for all $v \in \WJ$. 
\item If $x \not\ge \mcr{s_{\theta}}$, then $\# \bGx{x}{v} \ne 1$ for any $v \in \WJ$. 
\end{enu}
\end{prop}
%
%==============================%
%     START SUBSECTION 0504    %
%==============================%
%
\subsection{Proof of Proposition~\ref{prop:fin1}.}
\label{subsec:fin1}

As in Proposition~\ref{prop:fin1}, we assume that $x \in \WJe$. 
Recall that $\vpi_{k}$ is minuscule and $\J=I \setminus \{k\}$. 
Also, we recall from Section~\ref{subsec:tech} that 
$\mcr{\ls s_{k}} = z_{k}s_{k} \in \WJ$, with $z_{k} \in \WJs$.  
%
%%%%%%%%%%%%%%
%%% lem:zk %%%
%%%%%%%%%%%%%%
%
\begin{lem} \label{lem:zk}
If $\Fg$ is simply-laced and $\vpi_{k}$ is minuscule, 
then the element $\mcr{\ls s_{k}} = z_{k}s_{k}$ is identical to $\mcr{s_{\theta}}$. 
\end{lem}

\begin{proof}
It is easy to verify that $\ls \alpha_{k} = \theta$. From this, we see that 
$\mcr{\ls s_{k}} = \mcr{\ls s_{k} \ls^{-1}} = \mcr{s_{\ls \alpha_{k}}} = \mcr{s_{\theta}}$, 
as desired. 
\end{proof}
%
%%%%%%%%%%%%%%%%%%
%%% prop:fin1a %%%
%%%%%%%%%%%%%%%%%%
%
\begin{prop}[=Proposition~\ref{prop:fin1}\,(1)] \label{prop:fin1a}
Let $x \in \WJe$. 
If $x \ge z_{k}s_{k} = \mcr{s_{\theta}}$, then 
$\# \bGx{x}{v} = 0$ or $1$ for each $v \in \WJ$. 
\end{prop}

\begin{proof}
Let $v \in \WJ$ be such that $\# \bGx{x}{v} \ne 0$. 
Let $\bq \in \bGx{x}{v}$; recall that $\bq$ is of the form:
\begin{equation*}
\bq : \underbrace{x = y_{0} \edge{\gamma_{1}} y_{1} \edge{\gamma_{2}} \cdots 
  \edge{\gamma_{s}} y_{s}}_{ \in \bBG{x} } 
  \underbrace{ \edge{\gamma_{s+1} = \alpha_{k}} y_{s+1} }_{\text{quantum edge}} = \ed(\bq), 
\end{equation*}
where $y_{s} \in \WJ$, and $\mcr{\ed(\bq)} = \mcr{y_{s+1}}=v$. 
Since $y_{s} \ge x \ge z_{k}s_{k}=\mcr{s_{\theta}}$ in the Bruhat order, 
it follows from Lemmas~\ref{lem:tiy} and \ref{lem:zksk} that 
$y_{s} = \yJ{y_{s}}z_{k}s_{k}$, with $\yJ{y_{s}} \in \WJ$. 
Since $\mcr{y_{s}s_{k}} = \mcr{\ed(\bq)} =  v$ by the assumption, 
we deduce that $\yJ{y_{s}} = v$, and hence $y_{s}=vz_{k}s_{k}$. 
Thus, $y_{s}$ is uniquely determined by $v$. 
By the uniqueness of a label-increasing directed path 
from $x$ to $v z_{k} s_{k}$ (see Theorem~\ref{thm:LI}), 
we obtain $\# \bGx{x}{v} = 1$, as desired. This proves the proposition.
\end{proof}
%
%%%%%%%%%%%%%%%%%%
%%% prop:fin1b %%%
%%%%%%%%%%%%%%%%%%
%
\begin{prop}[=Proposition~\ref{prop:fin1}\,(2)] \label{prop:fin1b}
Let $x \in \WJe$. If $x \not\ge z_{k}s_{k}=\mcr{s_{\theta}}$, 
then $\# \bGx{x}{v} \ne 1$ for any $v \in \WJ$. 
\end{prop}

\begin{proof}
It is easily verified by Lemma~\ref{lem:bBGy} that 
$\# \bGx{x}{v}$ does not depend on the choice of 
a reflection order $\lhd$ satisfying \eqref{eq:ro}. 
In this proof, we take a reflection order $\lhd$ satisfying 
condition \eqref{eq:ro} and the additional condition that 
%
%%%%%%%%%%%%%%
%%% eq:ro2 %%%
%%%%%%%%%%%%%%
%
\begin{equation} \label{eq:ro2}
\beta \lhd \gamma \quad 
\text{for all $\beta \in (\DJp) \setminus \Inv (z_{k}s_{k})$ 
and $\gamma \in \Inv (z_{k}s_{k})$}; 
\end{equation}
the existence of a reflection order satisfying these conditions 
follows from Proposition~\ref{prop:lw} and the fact that $\mcr{\lng} \ge z_{k}s_{k}$ 
(see also Section~\ref{subsec:ro}). 

Now, let $v \in \WJ$ be such that $\# \bGx{x}{v} \ne 0$. Let 
\begin{equation*}
\bq:
 \underbrace{x = y_{0} \edge{\gamma_{1}} y_{1} \edge{\gamma_{2}} \cdots \edge{\gamma_{s-1}} y_{s-1}
  \edge{\gamma_{s}} y_{s}}_{\in \bBG{x}}
  \underbrace{\edge{\gamma_{s+1} = \alpha_{k}} y_{s+1}}_{\text{quantum edge}} = \ed(\bq)
\end{equation*}
be an element of $\bGx{x}{v}$; note that $y_{s} \in \WJe$ 
and $\mcr{\ed(\bq)}=\mcr{y_{s+1}}=v$. 
First we assume that $s \ge 1$, and $\gamma_{s} \in \Inv(z_{k}s_{k})$. 
Since $s \ge 1$, we see that $y_{s-1} \in \WJe$. 
Hence it follows from Lemma~\ref{lem:Q1} and \eqref{eq:bQBG1} that
\begin{equation*}
\bq': x = y_{0} \edge{\gamma_{1}} y_{1} \edge{\gamma_{2}} \cdots 
 \edge{\gamma_{s-1}} y_{s-1} \edge{\alpha_{k}} y_{s-1}s_{k}
\end{equation*}
is an element of $\bG_{x}^{\lhd}$. Since $y_{s},\,y_{s-1} \in \WJe$, 
we have by Lemma~\ref{lem:tiy}
\begin{equation*}
y_{s} = \yJ{y_{s}}\yJs{y_{s}}s_{k} \quad \text{and} \quad
y_{s-1} = \yJ{y_{s-1}}\yJs{y_{s-1}}s_{k}, 
\end{equation*}
with $\yJ{y_{s}},\,\yJ{y_{s-1}} \in \WJ$ and 
$\yJs{y_{s}},\,\yJs{y_{s-1}} \in \WJs$. Also, since 
$y_{s-1} \edge{\gamma_{s}} y_{s}$ is a Bruhat edge with label
$\gamma_{s} \in \Inv(z_{k}s_{k})$, 
we have $\yJ{y_{s-1}}=\yJ{y_{s}}$ by Lemma~\ref{lem:yJ}. 
Therefore,
\begin{equation*}
\begin{split}
\mcr{\ed(\bq')} & = 
\mcr{y_{s-1}s_{k}} = \mcr{\yJ{y_{s-1}}\yJs{y_{s-1}}} = 
\yJ{y_{s-1}} = \yJ{y_{s}}=\mcr{\yJ{y_{s}}\yJs{y_{s}}} \\ 
& = \mcr{y_{s}s_{k}} = \mcr{\ed(\bq)} = v,
\end{split}
\end{equation*}
and hence $\bq' \in \bGx{x}{v}$. 
Hence we obtain $\# \bGx{x}{v} \ge \#\bigl\{\bq,\,\bq'\bigr\}=2$. 

Next, we assume that $\gamma_{s} \not \in \Inv(z_{k}s_{k})$. 
In this case, we see by \eqref{eq:ro2} that 
$\gamma_{u} \not\in \Inv(z_{k}s_{k})$ for any $1 \le u \le s$. 
Also, since $x \not\ge z_{k}s_{k}$ by the assumption, 
we deduce from Lemma~\ref{lem:invc} that $y_{s} \not\ge z_{k}s_{k}$. 
Since $y_{s} \in \WJe$, it follows from Lemma~\ref{lem:EB} that 
there exists $\beta \in \Inv(z_{k}s_{k}) \setminus \{\alpha_{k}\}$ 
such that $y_{s} \edge{\beta} y_{s}s_{\beta}$ is a Bruhat edge. Therefore, 
\begin{equation*}
\bq'': 
 \underbrace{ x = y_{0} \edge{\gamma_{1}} y_{1} \edge{\gamma_{2}} \cdots 
 \edge{\gamma_{s-1}} y_{s-1} \edge{\gamma_{s}} y_{s} \edge{\beta} y_{s}s_{\beta} }_{\in \bBG{x}}
 \underbrace{ \edge{\alpha_{k}} y_{s}s_{\beta} s_{k} }_{\text{quantum edge}}
\end{equation*}
is an element of $\bG_{x}^{\lhd}$. 
Applying Lemmas~\ref{lem:tiy} and \ref{lem:yJ} 
to $y_{s},\,y_{s}s_{\beta} \in \WJe$, 
we can show by exactly the same argument as above that 
$\mcr{\ed(\bq'')} = \mcr{\ed(\bq)} = v$. 
Thus we obtain $\bq'' \in \bGx{x}{v}$, and hence 
$\# \bGx{x}{v} \ge \#\bigl\{\bq,\,\bq''\bigr\}=2$. 
This proves the proposition. 
\end{proof}

This completes the proof of Theorem~\ref{thm:main} 
in simply-laced types. 
%
%=========================%
%     START SECTION 06    %
%=========================%
%
\section{Proof of Theorem~\ref{thm:main} in type $B_{n}$.}
\label{sec:prfb}

In this section, we assume that $\Fg$ is of type $B_{n}$, and $k=n$, 
which is a unique element in $I$ such that $\vpi_{k}$ is minuscule. 
We set $\J=\J_{\vpi_{n}} = I \setminus \{n\}$. 
We may assume that $x \ne e$ by Proposition~\ref{prop:x=e}. 

%==============================%
%     START SUBSECTION 0601    %
%==============================%
%
\subsection{Lemmas on Bruhat edges in the quantum Bruhat graph (2).}
\label{subsec:tech2}
Recall from \eqref{eq:gamq} that $\gq = s_{n}\alpha_{n-1}$; 
note that $s_{\gq} = s_{n}s_{n-1}s_{n} \in \WJ$. We set
\begin{equation}
\WJx:=\bigl\{ y \in \WJ \mid y \ge s_{n}s_{n-1}s_{n} \bigr\}; 
\end{equation}
note that $\WJs=W_{\{1,2,\dots,n-1\}}$ is the Weyl group of type $A_{n-1}$; 
we denote by $(\WJs)^{\J \setminus \{n-2\}}$ the set of minimal coset representatives 
for the cosets in $\WJs/W_{\J \setminus \{n-2\}}$. 
%
%%%%%%%%%%%%%%%
%%% lem:tiw %%%
%%%%%%%%%%%%%%%
%
\begin{lem} \label{lem:tiw}
For each $y \in \WJx$, there exist a unique 
$\xJ{y} \in \WJ$ and $\xJs{y} \in \WJs$ 
such that $y = \xJ{y}\xJs{y}s_{n}s_{n-1}s_{n}$ and 
$\ell(y) = \ell(\xJ{y}) + \ell(\xJs{y}) + \ell(s_{n}s_{n-1}s_{n})$. 
Moreover, $\xJs{y} \in (\WJs)^{\J \setminus \{n-2\}}$. 
\end{lem}

\begin{proof}
Since $\vpi_{n}$ is minuscule, it follows from Proposition~\ref{prop:lw} 
that for each $y \in \WJx$, 
there exists a (unique) $w \in W$ such that $y = ws_{n}s_{n-1}s_{n}$
and $\ell(y) = \ell(w) + \ell(s_{n}s_{n-1}s_{n})$. 
The existence and uniqueness of $\xJ{y} \in \WJ$ and $\xJs{y} \in \WJs$ 
can be shown by exactly the same argument as for Lemma~\ref{lem:tiy}; 
replace $s_{k}$ in the proof of Lemma~\ref{lem:tiy} by $s_{n}s_{n-1}s_{n}$. 

It remains to show that $\xJs{y} \in (\WJs)^{\J \setminus \{n-2\}}$; 
for this, it suffices to verify that if $\xJs{y} \ne e$, 
then the rightmost simple reflection in any reduced expression of 
$\xJs{y}$ is $s_{n-2}$. Assume that $\xJs{y} \ne e$, and 
write $\xJs{y}$ as $\xJs{y} = ws_{j}$ for some $w \in \WJs$ and $j \in \J$ 
such that $\ell(\xJs{y}) = \ell(w) + \ell(s_{j})$. Suppose, for a contradiction, that 
$j \ne n-2$. Then we have 
\begin{equation*}
y = \xJ{y}\xJs{y}s_{n}s_{n-1}s_{n} =
 \begin{cases}
 \xJ{y}ws_{n-1}s_{n}s_{n-1}s_{n} = 
 \xJ{y}ws_{n}s_{n-1}s_{n}s_{n-1} & \text{if $j = n-1$}, \\[2mm]
 \xJ{y}ws_{j}s_{n}s_{n-1}s_{n} = 
 \xJ{y}ws_{n}s_{n-1}s_{n}s_{j} & \text{if $1 \le j \le n-1$},
 \end{cases}
\end{equation*}
which contradicts the assumption that $y \in \WJ$. 
Thus we obtain $j = n-2$, as desired. 
This proves the lemma. 
\end{proof}

Recall that $\ls \in \WJs$ is the longest element of $\WJs$; 
also, recall from Section~\ref{subsec:tech} that $\mcr{\ls s_{n}} = z_{n}s_{n}$. 
%
%%%%%%%%%%%%%%
%%% lem:wk %%%
%%%%%%%%%%%%%%
%
\begin{lem} \label{lem:wk}
In type $B_{n}$, the element $\mcr{\ls s_{n}} = z_{n}s_{n}$ 
is identical to $s_{1}s_{2} \cdots s_{n-1}s_{n}$. Moreover, 
%
%%%%%%%%%%%%%%%%%%
%%% eq:sthetab %%%
%%%%%%%%%%%%%%%%%%
%
\begin{equation} \label{eq:sthetab}
\begin{split}
\mcr{s_{\theta}} & = s_{2}s_{3} \cdots s_{n-1}s_{n}
  \overbrace{s_{1}s_{2} \cdots s_{n-2}s_{n-1}s_{n}}^{=z_{n}s_{n}} \\
& = 
\underbrace{s_{2}s_{3} \cdots s_{n-1}s_{1}s_{2} \cdots s_{n-2}}_{=:w_{n} \in \WJs}
\underbrace{s_{n}s_{n-1}s_{n}}_{=s_{\gq}}.
\end{split}
\end{equation}
\end{lem}

\begin{proof}
Since $\ls = (s_{1}s_{2} \cdots s_{n-1})(s_{1}s_{2} \cdots s_{n-2}) \cdots (s_{1}s_{2})s_{1}$, 
we can show the equalities
$z_{n}s_{n} = \mcr{\ls s_{n}} = s_{1}s_{2} \cdots s_{n-1}s_{n}$ by direct calculation. 
Let us show \eqref{eq:sthetab}. 
Recall that $\theta = \alpha_{1}+2\alpha_{2}+\cdots+2\alpha_{n}$ 
and $\pair{\vpi_{n}}{\theta^{\vee}} = 1$, and that 
$\gq = \alpha_{n-1}+2\alpha_{n}$ and $\pair{\vpi_{n}}{\gq^{\vee}} = 1$. 
We have $\mcr{s_{\theta}} \vpi_{n} = \vpi_{n} - \theta$, and 
\begin{align*}
w_{n}s_{\gq}\vpi_{n} & = 
  s_{2}s_{3} \cdots s_{n-1}s_{1}s_{2} \cdots s_{n-2} ( \vpi_{n} - \gq ) \\
& = s_{2}s_{3} \cdots s_{n-1} \bigl(\vpi_{n}-(\alpha_{1}+\cdots+\alpha_{n-1}+2\alpha_{n})\bigr) \\
& = \vpi_{n}-(\underbrace{ \alpha_{1}+2\alpha_{2} + \cdots+2\alpha_{n-1}+2\alpha_{n} }_{=\theta});
\end{align*}
from this, 
we see that $w_{n}s_{\gq} \in \WJ$. Also, 
since $\mcr{s_{\theta}} \vpi_{n} = w_{n}s_{\gq}\vpi_{n}$, and since 
$\mcr{s_{\theta}},\,w_{n}s_{\gq} \in \WJ$, we obtain  
$\mcr{s_{\theta}} = w_{n}s_{\gq}$, as desired. 
This proves the lemma. 
\end{proof}

%%%%%%%%%%%%%%%%
%%% lem:wnsn %%%
%%%%%%%%%%%%%%%%
%
\begin{lem} \label{lem:wnsn}
Let $y \in \WJx$, and set $w:=\xJs{y} \in \WJs$. 
Then, $w_{n}s_{n}s_{n-1}s_{n} \ge ws_{n}s_{n-1}s_{n}$, 
where the equality holds if and only if $y \ge \mcr{s_{\theta}} = w_{n}s_{n}s_{n-1}s_{n}$. 
\end{lem}

\begin{proof}
Notice that $w_{n}$ is the longest element of $(\WJs)^{\J \setminus \{n-2\}}$. 
Since $w \in (\WJs)^{\J \setminus \{n-2\}}$ by Lemma~\ref{lem:tiw}, 
we have $w_{n} \ge w$. 
Since $\ell(w_{n}s_{n}s_{n-1}s_{n}) = \ell(w_{n}) + \ell(s_{n}s_{n-1}s_{n})$
and $\ell(ws_{n}s_{n-1}s_{n})=\ell(w)+\ell(s_{n}s_{n-1}s_{n})$, 
we deduce by the Subword Property for the Bruhat order 
that $w_{n}s_{n}s_{n-1}s_{n} \ge ws_{n}s_{n-1}s_{n}$. 
Also, we can show by exactly the same argument as for Lemma~\ref{lem:zksk} 
that the equality holds if and only if  $y \ge \mcr{s_{\theta}} = w_{n}s_{n}s_{n-1}s_{n}$. 
This proves the lemma. 
\end{proof}

The following lemma can be shown in exactly the same way as Lemma~\ref{lem:EB}; 
recall from Lemma~\ref{lem:wk} that $\mcr{s_{\theta}} = w_{n}s_{n}s_{n-1}s_{n}$. 
%
%%%%%%%%%%%%%%%
%%% lem:EB2 %%%
%%%%%%%%%%%%%%%
%
\begin{lem} \label{lem:EB2}
Let $y \in \WJx$, and set $w:=\yJ{y} \in \WJ$. 
If $y \not\ge \mcr{s_{\theta}}$, or equivalently, 
if $w_{n}s_{n}s_{n-1}s_{n} > ws_{n}s_{n-1}s_{n}$, then 
there exists $\beta \in \Inv(\mcr{s_{\theta}}) \setminus \{\alpha_{n},\,s_{n}\alpha_{n-1}=\gq,\,
s_{n}s_{n-1}\alpha_{n} \}$ such that $y \edge{\beta} ys_{\beta}$ is a Bruhat edge. 
\end{lem}

Using Lemma~\ref{lem:tiw}, 
we can show the following 
by the same argument as for Lemma~\ref{lem:invc}. 
%
%%%%%%%%%%%%%%%%
%%% lem:invy %%%
%%%%%%%%%%%%%%%%
%
\begin{lem} \label{lem:invy}
Let $y \in \WJ$ and $\beta \in \DJp$ be such that 
$y \edge{\beta} ys_{\beta}$ is a Bruhat edge in $\QBG(W)$. 
If $y \not\ge \mcr{s_{\theta}}$ and $ys_{\beta} \ge \mcr{s_{\theta}}$, 
then $\beta \in \Inv(\mcr{ s_{\theta} })$. 
\end{lem}

The proof of the next lemma is similar to 
that of Lemma~\ref{lem:yJ}; 
remark that $\beta \ne \alpha_{n}$, $s_{n}\alpha_{n-1} (= \gq)$, 
$s_{n}s_{n-1}\alpha_{n}$. 
%
%%%%%%%%%%%%%%
%%% lem:xJ %%%
%%%%%%%%%%%%%%
%
\begin{lem} \label{lem:xJ}
Let $y \in \WJx$ and $\beta \in \Inv(\mcr{ s_{\theta} })$ be such that 
$y \edge{\beta} ys_{\beta}$ is a Bruhat edge in $\QBG(W)$; 
note that $ys_{\beta} \in \WJx$ by Lemma~\ref{lem:B}.
Then, $\xJ{y}=\xJ{ys_{\beta}}$. 
\end{lem}

%%%%%%%%%%%%%%%%%
%%% lem:gesgq %%%
%%%%%%%%%%%%%%%%%
%
\begin{lem} \label{lem:gesgq}
Let $y \in \WJe$. 
If $y \not\ge s_{n}s_{n-1}s_{n}$, then 
$y = s_{p}s_{p+1} \cdots s_{n-1}s_{n}$ for some $1 \le p \le n$. 
\end{lem}

\begin{proof}
Recall from Lemma~\ref{lem:tiy} that 
$y = \yJ{y}\yJs{y}s_{n}$, with $\yJ{y} \in \WJ$ and $\yJ{y} \in \WJs$. 
Assume that $\yJs{y} = e$, and hence $y = \yJ{y}s_{n}$. 
Suppose, for a contradiction, that $\yJ{y} \in \WJe$. 
Since the rightmost simple reflection of any reduced expression of $\yJ{y}$ must be $s_{n}$, 
it follows that $\ell(y) < \ell(\yJ{y}) + \ell(s_{n})$, which is a contradiction. 
Thus we obtain $\yJ{y} = e$, and hence $y = s_{n}$. 

Assume that $\yJs{y} \ne e$. 
By exactly the same argument as for Lemma~\ref{lem:tiw}, we deduce that 
$\yJs{y} \in \WJs$ is the minimal coset representative for a coset in 
$\WJs/W_{\J \setminus \{n-1\}}$. Hence we have 
$\yJs{y} = s_{p}s_{p+1} \cdots s_{n-1}$ for some $1 \le p \le n-1$. 
Suppose, for a contradiction, that $\yJ{y} \in \WJe$. 
Since the rightmost simple reflection of 
any reduced expression of $\yJ{y}$ must be $s_{n}$, 
it follows from the Subword Property for the Bruhat order that 
$y = \yJ{y}s_{p}s_{p+1} \cdots s_{n-1}s_{n} \ge s_{n}s_{n-1}s_{n}$, 
which contradicts the assumption. 
Therefore, we obtain $\yJ{y} = e$, and hence 
$y = s_{p}s_{p+1} \cdots s_{n-1}s_{n}$. This proves the lemma. 
\end{proof}
%
%%%%%%%%%%%%%%%%
%%% lem:invx %%%
%%%%%%%%%%%%%%%%
%
\begin{lem} \label{lem:invx}
Let $y \in \WJ$ and $\beta \in \DJp$ be such that 
$y \edge{\beta} ys_{\beta}$ is a Bruhat edge in $\QBG(W)$. 
If $y \not\ge s_{n}s_{n-1}s_{n}$ and $ys_{\beta} \ge s_{n}s_{n-1}s_{n}$, 
then $\beta = s_{n}s_{n-1}\alpha_{n}$. 
\end{lem}

\begin{proof}
Notice that $y \ne e$. By Lemma~\ref{lem:gesgq}, 
$y = s_{p}s_{p+1} \cdots s_{n-1}s_{n}$ for some $1 \le p \le n$. 
Since $ys_{\beta} \ge y$ and $\ell(ys_{\beta})=\ell(y)+1$, 
it follows from Proposition~\ref{prop:lw} that 
there exists $j \in I$ such that 
$ys_{\beta} = s_{j}y = s_{j}s_{p}s_{p+1} \cdots s_{n-1}s_{n}$; 
note that this is a reduced expression of $ys_{\beta}$. 
Since $ys_{\beta} \ge s_{n}s_{n-1}s_{n}$, 
we deduce that $j=n$. Thus we obtain 
$s_{\beta} = s_{n}s_{n-1}s_{n}s_{n-1}s_{n}$, and hence 
$\beta = s_{n}s_{n-1}\alpha_{n}$. This proves the lemma. 
\end{proof}
%
%==============================%
%     START SUBSECTION 0602    %
%==============================%
%
\subsection{Quantum edges in the quantum Bruhat graph (2).}
\label{subsec:qe2}
%
%%%%%%%%%%%%%%
%%% lem:Q2 %%%
%%%%%%%%%%%%%%
%
\begin{lem} \label{lem:Q2}
Let $y \in \WJ$ and $\gamma \in \DJp$. 
We have a quantum edge $y \edge{\gamma} ys_{\gamma}$ in $\QBG(W)$ if and only if 
$y \ne e$ and $\gamma=\alpha_{n}$, or $y \ge s_{n}s_{n-1}s_{n}$ and 
$\gamma=s_{n}\alpha_{n-1} = 
\alpha_{n-1}+2\alpha_{n}=\gq$ (see \eqref{eq:gamq}).
\end{lem}

\begin{proof}
We first show the ``if'' part. It is easily shown that 
if $y \ne e$ and $\gamma=\alpha_{n}$, 
then $y \edge{\gamma} ys_{\gamma}$ is a quantum edge. 
Also, we deduce from Lemma~\ref{lem:tiw} that 
if $y \ge s_{n}s_{n-1}s_{n}$ and $\gamma=s_{n}\alpha_{n-1}$, then 
$y \edge{\gamma} ys_{\gamma}$ is a quantum edge. 
Thus we have shown the ``if'' part. 

We next show the ``only if'' part. 
Since $y \edge{\gamma} ys_{\gamma}$ is a quantum edge, 
$\gamma$ is a quantum root (see Remark~\ref{rem:qe1}). Because 
\begin{equation}
\begin{split}
\Delta^{+}_{\lo} = & \left\{ \alpha_{i} + \cdots + \alpha_{j-1} \mid 1 \leq i < j \leq n \right\} \\
& \sqcup \left\{ \alpha_{i} + \cdots + \alpha_{j-1} + 2(\alpha_{j} + \cdots + \alpha_{n}) \mid 1 \leq i < j \leq n \right\}, \\
\Delta^{+}_{\sh} = & \left\{ \alpha_{i} + \cdots + \alpha_{n} \mid 1 \leq i \leq n \right\}, 
\end{split}
\end{equation}
it follows from Lemma~\ref{lem:qr} (together with the assumption $\gamma \in \DJp$) that 
$\gamma = \alpha_{i}+ \cdots + \alpha_{j-1} + 2(\alpha_{j}+ \cdots + \alpha_{n})$ 
for some $1 \le i < j \le n$, or $\gamma = \alpha_{n}$. 
Assume that $\gamma = \alpha_{i}+ \cdots + \alpha_{j-1} + 2(\alpha_{j}+ \cdots + \alpha_{n})$ 
for some $1 \le i < j \le n$; we see by direct calculation that 
$\ell(s_{\gamma}) = 2\pair{\rho}{\gamma^{\vee}}-1 = 4n-2i-2j+1$ and 
\begin{equation*}
\gamma = \alpha_{i}+ \cdots + \alpha_{j-1} + 2(\alpha_{j}+ \cdots + \alpha_{n}) = 
(s_{j} \cdots s_{n-1})(s_{i}s_{i+1} \cdots s_{n-2}s_{n})\alpha_{n-1},  
\end{equation*}
which implies that 
\begin{equation*}
s_{\gamma} = (s_{j} \cdots s_{n-1})(s_{i}s_{i+1} \cdots s_{n-2}s_{n})s_{n-1}
(s_{n}s_{n-2} \cdots s_{i+1}s_{i})(s_{n-1} \cdots s_{j})
\end{equation*}
is a reduced expression of $s_{\gamma}$. Here, since $y \in \WJ$, and 
$y \edge{\gamma} ys_{\gamma}$ is a quantum edge, 
it follows from Remark~\ref{rem:qe1} and \eqref{eq:mcr} that 
the leftmost simple reflection of any reduced expression of 
$s_{\gamma}$ is $s_{n}$. Hence we have $j = n$ and $i = n-1$. 
Therefore, we deduce that $\gamma=\alpha_{n-1}+2\alpha_{n} = s_{n}\alpha_{n-1}$. 
Thus we have shown that $\gamma$ is either $\alpha_{n}$ or $s_{n}\alpha_{n-1}$. 

Let $y \in \WJ$ be such that 
$y \edge{\gamma} ys_{\gamma}$ is a quantum edge. 
We can easily verify that if $\gamma=\alpha_{n}$, then $y \ne e$. 
Assume that $\gamma=s_{n}\alpha_{n-1}$. 
Recall that $\ell(s_{\gamma}) = 2\pair{\rho}{\gamma^{\vee}}-1$ 
since $\gamma$ is a quantum root. 
We have $y = (ys_{\gamma})s_{\gamma}$, with 
$\ell(y) = \ell(ys_{\gamma}) + \ell(s_{\gamma})$. 
Hence, by the Subword Property for the Bruhat order, 
we deduce that $y \ge s_{\gamma} = s_{n}s_{n-1}s_{n}$. 
Thus we have shown the ``only if'' part. This proves the lemma. 
\end{proof}
%
%==============================%
%     START SUBSECTION 0603    %
%==============================%
%
\subsection{Sets of label-increasing directed paths (2).}
\label{subsec:li2}

Let $\lhd$ be a reflection order on $\Delta^{+}$ satisfying condition \eqref{eq:ro}; 
observe that $s_{n}\alpha_{n-1}$ and $\alpha_{n}$ are the second largest element 
and the largest element of $\Delta^{+}$ with respect to $\lhd$, respectively. 
Let $x \in \WJ = W^{I \setminus \{n\}}$, with $x \ne e$. 
Recall the notation $\bBG{x}$ and $\bQBG{x}$ from Section~\ref{subsec:QBG}; 
remark that $\ed(\bp) \in \WJe$ for all $\bp \in \bBG{x}$. 
For each $\bp \in \bBG{x}$, we define $\SE^{\SQ}_{\alpha_{n}}(\bp)$ 
to be the concatenation $\bp \edge{\alpha_{n}} \ed(\bp)s_{n}$ 
of the directed path $\bp$ with the quantum edge 
$\ed(\bp) \edge{\alpha_{n}} \ed(\bp)s_{n}$ (see Lemma~\ref{lem:Q2}). 
Then we deduce that $\SE^{\SQ}_{\alpha_{n}}(\bp) \in \bQBG{x}$. 

We define $\bA_{x}^{\lhd}$ to be the subset of $\bBG{x}$ 
consisting of all those $\bp \in \bBG{x}$ such that 
$\ed(\bp) \ge s_{n}s_{n-1}s_{n}$. 
Recall that $\gamma_{\SQ}=s_{n}\alpha_{n-1}$ (see \eqref{eq:gamq}). 
For each $\bp \in \bA_{x}^{\lhd}$, we define $\SE^{\SQ}_{\gamma_{\SQ}}(\bp)$ 
to be the concatenation $\bp \edge{\gamma_{\SQ}} \ed(\bp)s_{\gamma_{\SQ}}$ 
of the directed path $\bp$ with the quantum edge 
$\ed(\bp) \edge{\gamma_{\SQ}} \ed(\bp)s_{\gamma_{\SQ}}$ (see Lemma~\ref{lem:Q2}). 
We claim that $\SE^{\SQ}_{\gamma_{\SQ}}(\bp) \in \bQBG{x}$. Indeed, recall that 
$\gamma_{\SQ}=s_{n}\alpha_{n-1}$ and $\alpha_{n}$ are the second largest element 
and the largest element of $\Delta^{+}$ with respect to $\lhd$, respectively. 
Hence it suffices to show that 
if $\bp$ is of the form $\bp:x = y_{0} \edge{\gamma_{1}} y_{1} \edge{\gamma_{2}} \cdots 
 \edge{\gamma_{s}} y_{s}$, with $s \ge 1$, then the final label $\gamma_{s}$ is neither 
 $\gamma_{\SQ}=s_{n}\alpha_{n-1}$ nor $\alpha_{n}$. 
Since $y_{s-1} \in \WJe$, it is easily seen that $\gamma_{s} \ne \alpha_{n}$. 
Also, since $\bp \in \bA_{x}^{\lhd}$, we have $y_{s} = \ed(\bp) \ge s_{n}s_{n-1}s_{n}$. 
By Lemma~\ref{lem:tiw}, $y_{s} = \xJ{y_{s}}\xJs{y_{s}}s_{n}s_{n-1}s_{n}$ and 
$\ell(y_{s})= \ell(\xJ{y_{s}}) + \ell(\xJs{y_{s}}) + \ell(s_{n}s_{n-1}s_{n})$. 
Hence it follows that $\ell(y_{s}s_{\gq}) = \ell(y_{s}) - 3$, 
which implies that $y_{s}s_{\gq} \edge{\gq} y_{s}$ is not a Bruhat edge. 
Therefore, we deduce that $\gamma_{s} \ne \gamma_{\SQ},\,\alpha_{n}$, and hence 
$\SE^{\SQ}_{\gamma_{\SQ}}(\bp) \in \bQBG{x}$. 
Here we note that $\ed(\SE^{\SQ}_{\gamma_{\SQ}}(\bp)) = \xJ{y_{s}}\xJs{y_{s}}$ 
and $\ell(\xJ{y_{s}}\xJs{y_{s}}s_{n}) = \ell(\xJ{y_{s}}\xJs{y_{s}})+1$. 
Hence we have a Bruhat edge 
$\ed(\SE^{\SQ}_{\gamma_{\SQ}}(\bp)) \edge{\alpha_{n}} 
\ed(\SE^{\SQ}_{\gamma_{\SQ}}(\bp))s_{n}$. 
We define $\SE^{\SB}_{\alpha_{n}}(\SE^{\SQ}_{\gamma_{\SQ}}(\bp))$ 
to be the concatenation $\SE^{\SQ}_{\gamma_{\SQ}}(\bp) \edge{\alpha_{n}} 
\ed(\SE^{\SQ}_{\gamma_{\SQ}}(\bp))s_{n}$ of the directed path 
$\SE^{\SQ}_{\gamma_{\SQ}}(\bp)$ with the Bruhat edge 
$\ed(\SE^{\SQ}_{\gamma_{\SQ}}(\bp)) \edge{\alpha_{n}} 
\ed(\SE^{\SQ}_{\gamma_{\SQ}}(\bp))s_{n}$. 

We see that 
%
%%%%%%%%%%%%%%%%
%%% eq:bQBG2 %%%
%%%%%%%%%%%%%%%%
%
\begin{equation} \label{eq:bQBG2}
\begin{split}
\bQBG{x} & = \bBG{x} \sqcup 
\bigl\{\SE^{\SQ}_{\alpha_{n}}(\bp) \mid \bp \in \bBG{x} \bigr\} \\ 
& \qquad \sqcup 
\bigl\{\SE^{\SQ}_{\gamma_{\SQ}}(\bp) 
   \mid \bp \in \bA_{x}^{\lhd} \bigr\} \sqcup 
\bigl\{\SE^{\SB}_{\alpha_{n}}(\SE^{\SQ}_{\gamma_{\SQ}}(\bp))
   \mid \bp \in \bA_{x}^{\lhd} \bigr\}; 
\end{split}
\end{equation}
remark that $\wt(\bp)=0$ and $\wt(\SE^{\SQ}_{\alpha_{n}}(\bp)) = \alpha_{n}^{\vee}$
for all $\bp \in \bBG{x}$, and that 
$\wt(\SE^{\SQ}_{\gamma_{\SQ}}(\bp)) = \wt (\SE^{\SB}_{\alpha_{n}}(\SE^{\SQ}_{\gamma_{\SQ}}(\bp))) = 
\gamma_{\SQ}^{\vee} = \alpha_{n-1}^{\vee}+\alpha_{n}^{\vee}$ for all $\bp \in \bA_{x}^{\lhd}$. 
Therefore, it follows from \eqref{eq:minuscule} that for all $N \in \BZ_{\ge 1}$, 
%
%%%%%%%%%%%%%%%%
%%% eq:NOS4x %%%
%%%%%%%%%%%%%%%%
%
\begin{equation} \label{eq:NOS4x}
\begin{split}
\gch V_{x}^{-}((N-1)\vpi_{n}) = \, & 
\be^{-x\vpi_{n}}
\sum_{ y \in \ed(\bBG{x}) } (-1)^{\ell(y)-\ell(x)} 
\gch V_{y}^{-}(N\vpi_{n}) \\[2mm]
& + 
\be^{-x\vpi_{n}}
\sum_{ y \in \ed(\bBG{x}) } (-1)^{\ell(y)-\ell(x)+1} 
\gch V_{\mcr{ys_{n}} t_{\alpha_{n}^{\vee}}}^{-}(N\vpi_{n}) \\[2mm]
& + 
\be^{-x\vpi_{n}}
\sum_{ y \in \ed(\bA_{x}^{\lhd}) } (-1)^{\ell(y)-\ell(x)+1} 
\gch V_{\mcr{ys_{\gamma_{\SQ}}} t_{\alpha_{n}^{\vee}}}^{-}(N\vpi_{n}) \\[2mm]
& + 
\be^{-x\vpi_{n}}
\sum_{ y \in \ed(\bA_{x}^{\lhd}) } (-1)^{\ell(y)-\ell(x)+2} 
\gch V_{\mcr{ys_{\gamma_{\SQ}s_{n}}} t_{\alpha_{n}^{\vee}}}^{-}(N\vpi_{n}). 
\end{split}
\end{equation}
Since $s_{\gamma_{\SQ}} = s_{n}s_{n-1}s_{n}$, 
it is easily seen that $\mcr{\ed(\bp)s_{n}} = \mcr{\ed(\bp)s_{\gamma_{\SQ}}s_{n}}$ 
for $\bp \in \bA_{x}^{\lhd} \subset \bBG{x}$. Hence we deduce that
%
%%%%%%%%%%%%%%%%
%%% eq:NOS4a %%%
%%%%%%%%%%%%%%%%
%
\begin{equation} \label{eq:NOS4a}
\begin{split}
\gch V_{x}^{-}((N-1)\vpi_{n}) = \, & 
\be^{-x\vpi_{n}}
\sum_{ y \in \ed(\bBG{x}) } (-1)^{\ell(y)-\ell(x)} 
\gch V_{y}^{-}(N\vpi_{n}) \\[2mm]
& + 
\be^{-x\vpi_{n}}
\sum_{ y \in \ed(\bBG{x} \setminus \bA_{x}^{\lhd})} (-1)^{\ell(y)-\ell(x)+1} 
\gch V_{\mcr{ys_{n}} t_{\alpha_{n}^{\vee}}}^{-}(N\vpi_{n}) \\[2mm]
& + 
\be^{-x\vpi_{n}}
\sum_{ y \in \ed(\bA_{x}^{\lhd}) } (-1)^{\ell(y)-\ell(x)+1} 
\gch V_{\mcr{ys_{\gamma_{\SQ}}} t_{\alpha_{n}^{\vee}}}^{-}(N\vpi_{n}). 
\end{split}
\end{equation}
Here we remark that if $x \ge s_{n}s_{n-1}s_{n}$, then $\bBG{x} = \bA_{x}^{\lhd}$; 
in particular, if $x \ge \mcr{s_{\theta}}$ (see Lemma~\ref{lem:wk}), then 
$\bBG{x} = \bA_{x}^{\lhd}$. Thus we have proved the character identity \eqref{eq:main_1} 
in Theorem \ref{thm:main}\,(1) in type $B_{n}$; we will prove 
the assertion on cancellations (following \eqref{eq:main_1}) in type $B_{n}$ 
in the next subsection (see the comment preceding Proposition~\ref{prop:fin2a}). 
%
%==============================%
%     START SUBSECTION 0604    %
%==============================%
%
\subsection{Cancellations in equation \eqref{eq:NOS4a}.}
\label{subsec:can2}

Let $x \in \WJe$. We set 
\begin{equation} \label{eq:bHx}
\bH_{x}^{\lhd} : = 
\bigl\{ \SE^{\SQ}_{\alpha_{n}}(\bp) \mid \bp \in \bBG{x} \setminus \bA_{x}^{\lhd} \bigr\} \sqcup 
\bigl\{ \SE^{\SQ}_{\gq}(\bp) \mid \bp \in \bA_{x}^{\lhd} \bigr\},
\end{equation}
and then $\bHx{x}{v}:=\bigl\{ \bq \in \bH_{x}^{\lhd} \mid \mcr{\ed(\bq)} = v \bigr\}$ for $v \in \WJ$. 
By \eqref{eq:NOS4a}, 
%
%%%%%%%%%%%%%%%%%
%%% eq:NOS4aa %%%
%%%%%%%%%%%%%%%%%
%
\begin{equation} \label{eq:NOS4aa}
\begin{split}
\gch V_{x}^{-}((N-1)\vpi_{n}) & = 
\be^{-x\vpi_{n}}
\sum_{ y \in \ed(\bBG{x}) } (-1)^{\ell(y)-\ell(x)} 
\gch V_{y}^{-}(N\vpi_{n}) \\ 
& \hspace*{5mm} + 
\be^{-x\vpi_{n}} \sum_{v \in \WJ}
\underbrace{
\left( \sum_{ \bq \in \bHx{x}{v} } (-1)^{\ell(\ed(\bq))-\ell(x)} \right)}_{%
\text{$=c_{v,1}^{x}$; see \eqref{eq:NOSa}} }
\gch V_{vt_{\alpha_{n}^{\vee}}}^{-}(N\vpi_{n}). 
\end{split}
\end{equation}
The assertion on cancellations (following \eqref{eq:main_1}) 
in type $B_{n}$ follows from the next proposition. 
%
%%%%%%%%%%%%%%%%%%
%%% prop:fin2a %%%
%%%%%%%%%%%%%%%%%%
%
\begin{prop} \label{prop:fin2a}
Let $x \in \WJe$. If $x \ge \mcr{s_{\theta}}$, then 
$\# \bHx{x}{v} = 0$ or $1$ for each $v \in \WJ$. 
\end{prop}

\begin{proof}
Recall from Lemma~\ref{lem:wk} that 
$\mcr{ s_{\theta} } = w_{n}s_{n}s_{n-1}s_{n} \in \WJ$, with $w_{n} \in \WJs$. 
Also, since $x \ge \mcr{s_{\theta}} \ge s_{n}s_{n-1}s_{n}$, 
we have $\bBG{x} = \bA_{x}^{\lhd}$, and hence $\bH_{x}^{\lhd}=
\bigl\{ \SE_{\gq}^{\SQ}(\bp) \mid \bp \in \bBG{x} = \bA_{x}^{\lhd} \bigr\}$, 
where $\gq = s_{n}\alpha_{n-1}$ (see \eqref{eq:gamq}). 
Now, let $v \in \WJ$ be such that $\# \bHx{x}{v} \ne 0$, 
and let $\bq \in \bHx{x}{v}$ be of the form:
\begin{equation*}
\bq : \underbrace{x = y_{0} \edge{\gamma_{1}} y_{1} \edge{\gamma_{2}} \cdots 
  \edge{\gamma_{s}} y_{s}}_{ = :\bp \in \bBG{x} = \bA_{x}^{\lhd} } 
  \underbrace{ \edge{\gq} y_{s+1} }_{\text{quantum edge}} = \ed(\bq); 
\end{equation*}
note that $y_{s} \in \WJx$ and $\mcr{\ed(\bq)} = \mcr{y_{s+1}}=v$. 
Since $y_{s} \ge x \ge \mcr{ s_{\theta} } = w_{n}s_{n}s_{n-1}s_{n}$, 
it follows from Lemmas~\ref{lem:tiw} and \ref{lem:wnsn} that 
$y_{s} = \xJ{y_{s}}w_{n}s_{n}s_{n-1}s_{n}$, with $\xJ{y_{s}} \in \WJ$. 
Since $\mcr{y_{s}s_{\gq}} = \mcr{y_{s+1}} = \mcr{\ed(\bq)} =  v$ by the assumption, 
we deduce that $\xJ{y_{s}} = v$, and hence $y_{s}=vw_{n}s_{n}s_{n-1}s_{n}$. 
Thus, $y_{s}=\ed(\bp)$ is determined uniquely by $v$. 
By the uniqueness of a label-increasing directed path 
from $x$ to $vw_{n}s_{n}s_{n-1}s_{n}$ (see Theorem~\ref{thm:LI}), 
we obtain $\# \bHx{x}{v} = 1$, as desired. This proves the proposition.
\end{proof}
%
%==============================%
%     START SUBSECTION 0605    %
%==============================%
%
\subsection{Proof of the character identity \eqref{eq:main_2} in type $B_{n}$.}
\label{subsec:fin2b}

Let $x \in \WJe$. We set
%
%%%%%%%%%%%%%%
%%% eq:bXx %%%
%%%%%%%%%%%%%%
%
\begin{equation} \label{eq:bXx}
\bX_{x}^{\lhd} : = 
\bigl\{ \SE^{\SQ}_{\alpha_{n}}(\bp) \mid \bp \in \bBG{x} \bigr\} \sqcup 
\bigl\{ \SE^{\SQ}_{\gq}(\bp), \SE^{\SB}_{\alpha_{n}}(\SE^{\SQ}_{\gq}(\bp)) \mid \bp \in \bA_{x}^{\lhd} \bigr\} = 
\bQBG{x} \setminus \bBG{x}, 
\end{equation}
and then $\bXx{x}{v}:=\bigl\{ \bq \in \bX_{x}^{\lhd} \mid \mcr{\ed(\bq)} = v \bigr\}$ for $v \in \WJ$. 
By \eqref{eq:NOS4x}, 
%
%%%%%%%%%%%%%%%%%
%%% eq:NOS4xa %%%
%%%%%%%%%%%%%%%%%
%
\begin{equation} \label{eq:NOS4xa}
\begin{split}
\gch V_{x}^{-}((N-1)\vpi_{n}) & = 
\be^{-x\vpi_{n}}
\sum_{ y \in \ed(\bBG{x}) } (-1)^{\ell(y)-\ell(x)} 
\gch V_{y}^{-}(N\vpi_{n}) \\ 
& \hspace*{5mm} + 
\be^{-x\vpi_{n}} \sum_{v \in \WJ}
\underbrace{
\left( \sum_{ \bq \in \bXx{x}{v} } (-1)^{\ell(\ed(\bq))-\ell(x)} \right)}_{%
\text{$=c_{v,1}^{x}$; see \eqref{eq:NOSa}} }
\gch V_{vt_{\alpha_{n}^{\vee}}}^{-}(N\vpi_{n}). 
\end{split}
\end{equation}
The character identity \eqref{eq:main_2} in Theorem~\ref{thm:main}\,(2) in type $B_{n}$ follows from 
Lemma~\ref{lem:itvB} and Proposition~\ref{prop:fin2b} below.
%
%%%%%%%%%%%%%%%%
%%% lem:itvB %%%
%%%%%%%%%%%%%%%%
%
\begin{lem} \label{lem:itvB}
Let $x \in \WJe$ and $v \in \WJ$. 
If $\# \bXx{x}{v} \ge 2$, then 
\begin{equation*}
c_{v,1}^{x} = \sum_{ \bq \in \bXx{x}{v} } (-1)^{\ell(\ed(\bq))-\ell(x)} = 0.
\end{equation*}
\end{lem}

\begin{proof}
We prove the assertion by descending induction on $\ell(x)$; 
the proof is similar to that of Lemma~\ref{lem:itv}. 
If $x=\mcr{\lng}$, then we see that 
\begin{equation*}
\bX_{\mcr{\lng}}^{\lhd}=\bigl\{ 
\lng \edge{\alpha_{n}} \lng s_{n}, \ 
\lng \edge{\gq} \lng s_{\gq}, \ 
\lng \edge{\gq} \lng s_{\gq} \edge{\alpha_{n}} \lng s_{\gq}s_{n} \bigr\}. 
\end{equation*}
Hence we can show the assertion by direct calculation. 
Assume that $x < \mcr{\lng}$, and let $j \in I$ be such that $s_{j}x > x$, 
or equivalently, $x^{-1}\alpha_{j} \in \Delta^{+}$; 
note that $\pair{x\vpi_{n}}{\alpha_{j}^{\vee}} = 1 > 0$ 
since $\vpi_{n}$ is minuscule, and that $s_{j}x \in \WJ$. 
Let $v \in \WJ$ be such that $\# \bXx{x}{v} \ge 2$. 
If $\pair{v\vpi_{n}}{\alpha_{j}^{\vee}} = 0$, 
then $c_{v,1}^{x} = 0$ by Lemma \ref{lem:+1}. 
Hence we may assume that 
$\pair{v\vpi_{n}}{\alpha_{j}^{\vee}} \ne 0$; 
note that $s_{j}v \in \WJ$ in this case.

%%%%%%%%%%%%%%%%%%%%%%%
\paragraph{\bf Case 1.}
%%%%%%%%%%%%%%%%%%%%%%%
%
Assume that $\pair{v\vpi_{n}}{\alpha_{j}^{\vee}} > 0$. 
We define an injective map $\bXx{x}{v} \rightarrow \bXx{s_{j}x}{s_{j}v}$, 
$\bq \mapsto \ti{\bq}$, as follows: for $\bq \in \bXx{x}{v}$ with $y:=\ed(\bp)$, 
we define $\ti{\bq}$ to be the label-increasing 
(shortest) directed path from $s_{j}x$ to $s_{j}y$ in $\QBG(W)$ 
(see Theorem~\ref{thm:LI}). We claim that $\ti{\bq} \in \bXx{s_{j}x}{s_{j}v}$. 
Indeed, recall that $\bq$ is either of the following forms: 
%
%%%%%%%%%%%%%
%%% eq:qa %%%
%%%%%%%%%%%%%
%
\begin{equation} \label{eq:qa}
\bq:\underbrace{%
  x = y_{0} \edge{\gamma_{1}} y_{1} \edge{\gamma_{2}} \cdots 
  \edge{\gamma_{s}} y_{s}}_{=: \bp \in \bBG{x}}
  \underbrace{ \edge{\gamma_{s+1}} y_{s+1} }_{\text{quantum edge}} = \ed(\bq) =y,
\end{equation}
where $\gamma_{s+1}$ is $\alpha_{n}$ (resp., $\gq$) 
if $\bp \in \bBG{x} \setminus \bA_{x}^{\lhd}$ (resp., $\bp \in \bA_{x}^{\lhd}$), or 
%
%%%%%%%%%%%%%
%%% eq:qb %%%
%%%%%%%%%%%%%
%
\begin{equation} \label{eq:qb}
\bq:\underbrace{%
  x = y_{0} \edge{\gamma_{1}} y_{1} \edge{\gamma_{2}} \cdots 
  \edge{\gamma_{s}} y_{s}}_{=: \bp \in \bA_{x}^{\lhd} }
  \underbrace{ \edge{ \gamma_{s+1}=\gq } y_{s+1} }_{\text{quantum edge}} 
  \underbrace{ \edge{ \gamma_{s+2} = \alpha_{n} } y_{s+2} }_{\text{Bruhat edge}} 
  = \ed(\bq) = y;
\end{equation}
notice that $x^{-1}\alpha_{j} \in \Delta^{+}$ and 
$y^{-1}\alpha_{j} \in \Delta^{+}$. We set $t:=s+1$ (resp., $=s+2$) 
if $\bq$ is of the form \eqref{eq:qa} (resp., \eqref{eq:qb}). 
If $y_{u}^{-1}\alpha_{j} \in \Delta^{+}$ for all $1 \le u \le t-1$, 
then we see by Lemma~\ref{lem:DL}\,(2) that 
there exists a directed path $\bq'$ in $\QBG(W)$ 
from $s_{j}x$ to $s_{j}y$ of the following form: 
\begin{equation*}
\bq' : s_{j} x = s_{j}y_{0} \edge{\gamma_{1}} s_{j}y_{1} \edge{\gamma_{2}} \cdots 
  \edge{\gamma_{t}} s_{j}y_{t} = s_{j}y, 
\end{equation*}
with $\wt(\bq') = \wt (\bq) \ne 0$. Observe that 
$\bq' \in \bQBG{s_{j}x} \setminus \bBG{s_{j}x}$, and 
$\mcr{s_{j}y} = \mcr{ \ed(\bq') } = s_{j}v$. 
Hence we obtain $\bq' \in \bXx{s_{j}x}{s_{j}v}$. 
Moreover, by the uniqueness of a label-increasing 
directed path from $s_{j}x$ to $s_{j}y$, we deduce that $\ti{\bq} = \bq'$, 
and hence $\ti{\bq} \in \bXx{s_{j}x}{s_{j}v}$ in this case. 

Assume now that $y_{u}^{-1}\alpha_{j} \in \Delta^{-}$ for some $1 \le u \le t-1$; 
remark that $t \ge 2$ in this case, since $y_{0}^{-1}\alpha_{j} \in \Delta^{+}$ and 
$y_{t}^{-1}\alpha_{j} \in \Delta^{+}$. If we set
$a:=\min \bigl\{ 1 \le u \le t-1 \mid y_{u}^{-1}\alpha_{j} \in \Delta^{-} \bigr\}$, 
then we deduce from Lemma~\ref{lem:DL} that 
$\gamma_{a}=y_{a-1}^{-1}\alpha_{j}$, and 
that there exists a directed path $\bq''$ in $\QBG(W)$ from 
$s_{j}x$ to $y = \ed(\bq)$ of the following form: 
\begin{equation*}
\bq'': s_{j}x = s_{j}y_{0} \edge{\gamma_{1}} \cdots 
  \edge{\gamma_{a-1}} s_{j}y_{a-1}=y_{a} \edge{\gamma_{a+1}} \cdots
  \edge{\gamma_{t}} y_{t} = y; 
\end{equation*}
notice that $\bq'' \in \bQBG{s_{j}x}$. 
Here, since $x^{-1}\alpha_{j} \in \Delta^{+}$ and 
$y^{-1}\alpha_{j} \in \Delta^{+}$, it follows from 
\cite[Lemma~7.7\,(4)]{LNSSS1} that 
$\ell(\ti{\bq}) = \ell(s_{j}x \Rightarrow s_{j}y) = 
 \ell(x \Rightarrow y) = \ell(\bq) = t \ge 1$, and 
$\wt(\ti{\bq}) = \wt(s_{j}x \Rightarrow s_{j}y) = 
 \wt(x \Rightarrow y) = \wt(\bq) \ne 0$. 
Let us write $\ti{\bq}$ as:
\begin{equation*}
\ti{\bq} : s_{j}x = x_{0} \edge{\beta_{1}} x_{1} \edge{\beta_{2}} \cdots 
  \edge{\beta_{t}} x_{t} = s_{j}y,
\end{equation*}
where $\beta_{1} \lhd \beta_{2} \lhd \cdots \lhd \beta_{t}$. 
We will show that $\beta_{1} \in \DJp$.
Notice that $x_{0}^{-1}\alpha_{j} \in \Delta^{-}$ and 
$x_{t}^{-1}\alpha_{j} \in \Delta^{-}$. 
Suppose, for a contradiction, that 
$x_{u}^{-1}\alpha_{j} \in \Delta^{-}$ for all $1 \le u \le t-1$.
Then, we see by Lemma~\ref{lem:DL}\,(2) that 
there exists a directed path $\ti{\bq}'$ in $\QBG(W)$ 
from $x$ to $y$ of the following form: 
\begin{equation*}
\ti{\bq}' : x = s_{j}x_{0} \edge{\beta_{1}} s_{j}x_{1} \edge{\beta_{2}} \cdots 
  \edge{\beta_{t}} s_{j}x_{t} = y. 
\end{equation*}
By the uniqueness of a label-increasing directed path from $x$ to $y$, 
we deduce that $\ti{\bq}'=\bq$; in particular, $s_{j}x_{a} = y_{a}$. 
However, $\Delta^{+} \ni (s_{j}x_{a})^{-1}\alpha_{j} = y_{a}^{-1}\alpha_{j} \in \Delta^{-}$, 
which is a contradiction. Thus there exists $1 \le u \le t-1$ such that 
$x_{u}^{-1}\alpha_{j} \in \Delta^{+}$. 
If we set $b:=\max \bigl\{1 \le u \le t-1 \mid x_{u}^{-1}\alpha_{j} \in \Delta^{+} \bigr\}$, 
then we see by Lemma~\ref{lem:DL} that 
there exists a directed path $\ti{\bq}''$ in $\QBG(W)$ 
from $s_{j}x$ to $y$ of the following form: 
\begin{equation*}
\ti{\bq}'' : s_{j}x = x_{0} \edge{\beta_{1}}  \cdots 
  \edge{\beta_{b}} x_{b}=s_{j}x_{b+1} \edge{\beta_{b+2}} 
  \cdots \edge{\beta_{t}} s_{j}x_{t}  = y. 
\end{equation*}
By the uniqueness of a label-increasing directed path from $s_{j}x$ to $y$, 
we deduce that $\ti{\bq}'' = \bq''$. Hence $\beta_{1}$ is either $\gamma_{1}$ (if $a \ge 2$) 
or $\gamma_{2}$ (if $a=1$). Thus we obtain $\beta_{1} \in \DJp$, as desired. 
Since the reflection order $\lhd$ satisfies condition \eqref{eq:ro}, 
it follows that $\beta_{u} \in \DJp$ for all $1 \le u \le t$, 
which implies that $\ti{\bq} \in \bQBG{s_{j}x}$. 
Also, since $\wt(\ti{\bq}) \ne 0$ as seen above, 
we find that $\ti{\bq} \notin \bBG{s_{j}x}$, 
and hence $\ti{\bq} \in \bX_{s_{j}x}^{\lhd}$. It is easily seen that 
$\mcr{\ed(\ti{\bq})} = \mcr{s_{j}y} = s_{j}v$. Therefore, it follows that 
$\ti{\bq} \in \bXx{s_{j}x}{s_{j}v}$. 

The injectivity of the map $\bXx{x}{v} \rightarrow \bXx{s_{j}x}{s_{j}v}$, 
$\bq \mapsto \ti{\bq}$, can be shown by exactly the same argument as 
for the map $\bGx{x}{v} \rightarrow \bGx{s_{j}x}{s_{j}v}$ 
in Case 1 in the proof of Lemma~\ref{lem:itv}. Hence we obtain
$\# \bXx{s_{j}x}{s_{j}v} \ge \# \bXx{x}{v} \ge 2$. 
From our induction hypothesis and Lemma~\ref{lem:-1}\,(2), we conclude that 
$c^{x}_{v,1} = c^{s_{j}x}_{s_{j}v,1} = 0$, as desired. 

%%%%%%%%%%%%%%%%%%%%%%%
\paragraph{\bf Case 2.}
%%%%%%%%%%%%%%%%%%%%%%%
%
Assume that $\pair{v\vpi_{n}}{\alpha_{j}^\vee} < 0$. 
We define an injective map $\bXx{x}{v} \rightarrow \bXx{s_{j}x}{v}$, 
$\bq \mapsto \ti{\bq}$, as follows. 
Assume that $\bq \in \bXx{x}{v}$ is 
either of the forms \eqref{eq:qa} or \eqref{eq:qb}, 
and define $t \in \{s+1,\,s+2\}$ as in Case 1. 
Note that $x^{-1}\alpha_{j} \in \Delta^{+}$ and 
$y^{-1}\alpha_{j} \in \Delta^{-}$ in this case. If we set
$a:=\min \bigl\{ 1 \le u \le t \mid y_{u}^{-1}\alpha_{j} \in \Delta^{-} \bigr\}$, 
then it follows from Lemma~\ref{lem:DL} that 
there exists a directed path in $\QBG(W)$ from $s_{j}x$ to $y$ of the form:
\begin{equation*}
\ti{\bq} : 
  s_{j}x = s_{j}y_{0} \edge{\gamma_{1}} \cdots 
  \edge{\gamma_{a-1}} s_{j}y_{a-1}=y_{a} \edge{\gamma_{a+1}} \cdots
  \edge{\gamma_{s}} y_{s} \edge{\gamma_{s+1}} \cdots \edge{\gamma_{t}} y_{t} = y, 
\end{equation*}
with $\wt(\ti{\bq}) = \wt (\bq) \ne 0$. 
Observe that $\ti{\bq} \in \bQBG{s_{j}x} \setminus \bBG{s_{j}x}$, 
and $\mcr{ \ed(\ti{\bq}) } = \mcr{ y } = v$. 
Hence we have $\ti{\bq} \in \bXx{s_{j}x}{v}$. 
By the same argument as in Case 1 in the proof of Lemma~\ref{lem:itv},
we can show that the map $\bXx{x}{v} \rightarrow \bXx{s_{j}x}{v}$, 
$\bq \mapsto \ti{\bq}$, is injective. 
Therefore, we obtain 
$\# \bXx{s_{j}x}{v} \ge \# \bXx{x}{v} \ge 2$. 
From our induction hypothesis and 
Lemma \ref{lem:-1}\,(1), we conclude that $c^{x}_{v,1} = - c^{s_{j}x}_{v,1} = 0$, as desired. 

This completes the proof of the lemma. 
\end{proof}
%
%%%%%%%%%%%%%%%%%%
%%% prop:fin2b %%%
%%%%%%%%%%%%%%%%%%
%
\begin{prop} \label{prop:fin2b}
Let $x \in \WJe$. If $x \not\ge \mcr{s_{\theta}}$, 
then $\# \bXx{x}{v} \ne 1$ for any $v \in \WJ$. 
\end{prop}

\begin{proof}
It is easily verified by Lemma~\ref{lem:bBGy} and \eqref{eq:bXx} that 
$\# \bXx{x}{v}$ does not depend on the choice of 
a reflection order $\lhd$ satisfying \eqref{eq:ro}. 
In this proof, we take a reflection order $\lhd$ satisfying 
condition \eqref{eq:ro} and the additional condition that 
%
%%%%%%%%%%%%%%
%%% eq:ro3 %%%
%%%%%%%%%%%%%%
%
\begin{equation} \label{eq:ro3}
\beta \lhd \gamma \quad 
\text{for all $\beta \in (\DJp) \setminus \Inv (\mcr{s_{\theta}})$ 
and $\gamma \in \Inv (\mcr{ s_{\theta} })$}; 
\end{equation}
the existence of a reflection order satisfying these conditions 
follows from Proposition~\ref{prop:lw} and 
the fact that $\mcr{\lng} \ge \mcr{s_{\theta}}$ (see also Section~\ref{subsec:ro}). 

Let $v \in \WJ$ be such that $\# \bXx{x}{v} \ne 0$. 
We will show that $\# \bXx{x}{v} \ge 2$. 
Let $\bq \in \bXx{x}{v}$. If $\bq = \SE_{\alpha_{n}}^{\SQ}(\bp)$ 
for some $\bp \in \bA_{x}^{\lhd}$, then we deduce that 
$\bq' = \SE_{\alpha_{n}}^{\SB}(\SE_{\gq}^{\SQ}(\bp)) \in \bXx{x}{v}$, and hence 
$\# \bXx{x}{v} \ge 2$. 
Similarly, if $\bq = \SE_{\alpha_{n}}^{\SB}(\SE_{\gq}^{\SQ}(\bp))$ 
for some $\bp \in \bA_{x}^{\lhd}$, then 
we deduce that $\bq':=\SE_{\alpha_{n}}^{\SQ}(\bp) \in \bXx{x}{v}$, and hence 
$\# \bXx{x}{v} \ge 2$.

Next, assume that $\bq = \SE_{\gq}^{\SQ}(\bp)$ for some $\bp \in \bA_{x}^{\lhd}$, 
and write it as: 
\begin{equation*}
\bq:\underbrace{%
  x = y_{0} \edge{\gamma_{1}} y_{1} \edge{\gamma_{2}} \cdots 
  \edge{\gamma_{s}} y_{s}}_{ = \bp \in \bA_{x}^{\lhd} }
  \underbrace{\edge{\gq} y_{s+1}}_{\text{quantum edge}} = \ed(\bq);
\end{equation*}
note that $y_{s} \ge s_{n}s_{n-1}s_{n}$. 
If $s \ge 1$ and $\gamma_{s} \in \Inv (\mcr{s_{\theta}})$, then 
we define $\bp'$ to be 
\begin{equation*}
\bp' : \underbrace{%
  x = y_{0} \edge{\gamma_{1}} y_{1} \edge{\gamma_{2}} \cdots 
  \edge{\gamma_{s-1}} y_{s-1}}_{\in \bBG{x}}. 
\end{equation*}
If $y_{s-1} \ge s_{n}s_{n-1}s_{n}$, or equivalently, 
if $\bp' \in \bA_{x}^{\lhd}$, then we define $\bq':=\SE_{\gq}^{\SQ}(\bp')$, that is, 
\begin{equation*}
\bq' : \underbrace{%
  x = y_{0} \edge{\gamma_{1}} y_{1} \edge{\gamma_{2}} \cdots 
  \edge{\gamma_{s-1}} y_{s-1}}_{= \bp' \in \bA_{x}^{\lhd}} 
  \underbrace{ \edge{\gq} y_{s-1}s_{\gq} }_{\text{quantum edge}}. 
\end{equation*}
It is easily seen that $\bq' \in \bX_{x}^{\lhd}$. 
Moreover, we see from Lemma~\ref{lem:xJ} 
(applied to the Bruhat edge $y_{s-1} \edge{\gamma_{s}} y_{s}$) that 
$\mcr{ y_{s-1}s_{\gq} } = \mcr{ \xJ{y_{s-1}}\xJs{y_{s-1}} } = 
\xJ{y_{s-1}}=\xJ{y_{s}}= \mcr{ y_{s}s_{\gq} } = \mcr{ y_{s+1} } = \mcr{ \ed(\bq) }= v$, 
which implies that $\bq' \in \bXx{x}{v}$. Hence we obtain 
$\# \bXx{x}{v} \ge \# \bigl\{ \bq,\,\bq' \bigr\} = 2$. 
Assume that $y_{s-1} \not \ge s_{n}s_{n-1}s_{n}$, or equivalently, 
$\bp' \in \bBG{x} \setminus \bA_{x}^{\lhd}$. 
In this case, we deduce from Lemma~\ref{lem:invx} that 
$\gamma_{s}=s_{n} s_{n-1}\alpha_{n}$, 
which implies that $\mcr{y_{s-1}s_{n}} = \mcr{ y_{s} s_{\gq} } 
= \mcr{ y_{s+1} } = \mcr{ \ed(\bq) }= v$. 
Hence, if we define $\bq':=\SE_{\alpha_{n}}^{\SQ}(\bp')$, that is, 
\begin{equation*}
\bq' : \underbrace{%
  x = y_{0} \edge{\gamma_{1}} y_{1} \edge{\gamma_{2}} \cdots 
  \edge{\gamma_{s-1}} y_{s-1}}_{= \bp' \in \bBG{x} \setminus \bA_{x}^{\lhd} } 
  \underbrace{ \edge{\alpha_{n}} y_{s-1}s_{n} }_{\text{quantum edge}}, 
\end{equation*}
then $\bq' \in \bXx{x}{v}$, and hence 
$\# \bXx{x}{v} \ge \# \bigl\{ \bq,\,\bq' \bigr\} = 2$. 

Assume that $\gamma_{s} \not \in \Inv (\mcr{s_{\theta}})$. 
By \eqref{eq:ro3}, we see that 
$\gamma_{u} \notin \Inv (\mcr{s_{\theta}})$ for any $1 \le u \le s$. 
Since $x \not\ge \mcr{s_{\theta}}$ by the assumption, 
we deduce by Lemma~\ref{lem:invy} 
that $y_{s} \not\ge \mcr{s_{\theta}}$. 
Recall that $y_{s} \ge s_{n}s_{n-1}s_{n}$. 
Hence, by Lemma~\ref{lem:EB2}, 
there exists $\gamma \in \Inv(\mcr{s_{\theta}}) \setminus 
\bigl\{ \alpha_{n},\,\gq,\,s_{n}s_{n-1}\alpha_{n} \bigr\}$ such that 
$y_{s} \edge{\gamma} y_{s}s_{\gamma}$ is a Bruhat edge in $\QBG(W)$; 
remark that $\gamma_{s} \lhd \gamma$ by \eqref{eq:ro3}. 
Since $y_{s}s_{\gamma} > y_{s} \ge s_{n}s_{n-1}s_{n}$, 
we have a quantum edge 
$y_{s}s_{\gamma} \edge{\gq} y_{s}s_{\gamma}s_{\gq}$ by Lemma~\ref{lem:Q2}. 
Now we define $\bq'$ to be 
\begin{equation*}
\bq' : \underbrace{%
  x = y_{0} \edge{\gamma_{1}} y_{1} \edge{\gamma_{2}} \cdots 
  \edge{\gamma_{s-1}} y_{s-1} \edge{\gamma_{s}} y_{s} 
  \edge{\gamma} y_{s}s_{\gamma} }_{ \in \bA_{x}^{\lhd} } 
  \underbrace{\edge{\gq} y_{s}s_{\gamma}s_{\gq}}_{\text{quantum edge}}. 
\end{equation*}
It is easily verified by Lemma~\ref{lem:xJ} (applied to 
the Bruhat edge $y_{s} \edge{\gamma} y_{s}s_{\gamma}$)
that $\bq' \in \bXx{x}{v}$, 
and hence $\# \bXx{x}{v} \ge \# \bigl\{\bq,\,\bq'\bigr\} =2$. 

Finally, assume that $\bq = \SE_{\alpha_{n}}^{\SQ}(\bp)$ 
for some $\bp \in \bBG{x} \setminus \bA_{x}^{\lhd}$, 
and write $\bq$ as: 
\begin{equation*}
\bq:\underbrace{%
  x = y_{0} \edge{\gamma_{1}} y_{1} \edge{\gamma_{2}} \cdots 
  \edge{\gamma_{s}} y_{s}}_{= \bp \in \bBG{x} \setminus \bA_{x}^{\lhd}}
  \underbrace{\edge{\alpha_{n}} y_{s+1}}_{\text{quantum edge}} = \ed(\bq);
\end{equation*}
note that $y_{s} \not\ge s_{n}s_{n-1}s_{n}$, and hence 
$y_{s-1} \not\ge s_{n}s_{n-1}s_{n}$ (if $s \ge 1$). 
Remark that these elements are of the form in Lemma~\ref{lem:gesgq}. 
If $s \ge 1$ and $\gamma_{s} \in \Inv (\mcr{s_{\theta}})$, then 
we define $\bq'$ to be 
\begin{equation*}
\bq' : \underbrace{%
  x = y_{0} \edge{\gamma_{1}} y_{1} \edge{\gamma_{2}} \cdots 
  \edge{\gamma_{s-1}} y_{s-1}}_{= \bp' \in \bBG{x} \setminus \bA_{x}^{\lhd}} 
  \underbrace{ \edge{\alpha_{n}} y_{s-1}s_{n} }_{\text{quantum edge}}. 
\end{equation*}
Then it is easily seen by Lemma~\ref{lem:gesgq} 
that $\bq' \in \bXx{x}{v}$, and hence 
$\# \bXx{x}{v} \ge \# \bigl\{ \bq,\,\bq' \bigr\} = 2$.

Assume that $\gamma_{s} \not \in \Inv (\mcr{s_{\theta}})$. 
We deduce by \eqref{eq:ro3} that 
$\gamma_{u} \notin \Inv (\mcr{s_{\theta}})$ for any $1 \le u \le s$. 
Since $y_{s} \ne e$ and $y_{s} \not\ge s_{n}s_{n-1}s_{n}$, 
it follows from Lemma~\ref{lem:gesgq} that 
$y_{s} = s_{p}s_{p+1} \cdots s_{n-1}s_{n}$ for some $1 \le p \le n$. 
If $p < n$ (resp., $p=n$), then we set $\gamma := s_{n}s_{n-1}\alpha_{n}$
(resp., $\gamma := s_{n}\alpha_{n-1} = \gq$). 
In either case, $\gamma \in \Inv(\mcr{s_{\theta}})$, and 
we have a Bruhat edge $y_{s} \edge{\gamma} y_{s}s_{\gamma}$; 
note that $y_{s}s_{\gamma} \ge s_{n}s_{n-1}s_{n}$ (resp., 
$\not\ge s_{n}s_{n-1}s_{n}$) if $p < n$ (resp., $p=n$).
Now we define $\bq'$ to be 
\begin{equation*}
\bq' : \underbrace{%
  x = y_{0} \edge{\gamma_{1}} y_{1} \edge{\gamma_{2}} \cdots 
  \edge{\gamma_{s-1}} y_{s-1} \edge{\gamma_{s}} y_{s} 
  \edge{\gamma} y_{s}s_{\gamma} }_{ \in \bBG{x}} 
  \edge{\beta} y_{s}s_{\gamma}s_{\beta}, 
\end{equation*}
where $\beta:=\gq$ (resp., $\alpha_{n}$) if $p < n$ (resp., $p=n$). 
By \eqref{eq:ro3}, we have $\gamma_{s} \lhd \gamma \lhd \beta$. 
It is easily verified by Lemma~\ref{lem:gesgq} that $\bq' \in \bXx{x}{v}$, 
and hence $\# \bXx{x}{v} \ge \# \bigl\{\bq,\,\bq'\bigr\} = 2$. 
This proves the proposition. 
\end{proof}

This completes the proof of Theorem~\ref{thm:main} in type $B_{n}$. 

%%%%%%%%%
\appendix
%%%%%%%%%
%
%========================%
%     START SECTION A    %
%========================%
%
\section{An Example in type $A_{6}$.}
\label{sec:ex}

In this appendix, we assume that $\Fg$ is of type $A_{6}$ and $k =3$ 
in Theorem~\ref{thm:introThm1} in the Introduction; 
we know that $G/\PrJ$ is the Grassmannian $\Gr(3,7)$, where 
$\J=I \setminus \{3\} = \bigl\{1,2,4,5,6\bigr\}$. 
Observe that $\mcr{\lng} = s_4 s_3 s_2 s_1 s_5 s_4 s_3 s_2 s_6 s_5 s_4 s_3$ 
is a reduced expression of $\mcr{\lng}$. 
We take a reflection order $\lhd$ on $\Delta^{+}$ 
satisfying condition \eqref{eq:ro} such that on $\DJp = \Inv(\mcr{ \lng })$: 
\begin{align*}
& \alpha_1 + \cdots + \alpha_6 \lhd 
  \alpha_1 + \cdots + \alpha_5 \lhd 
  \alpha_1 + \cdots + \alpha_4 \lhd 
  \alpha_1 + \alpha_2 + \alpha_3 \\
\lhd & \ 
  \alpha_2 + \cdots + \alpha_6 \lhd 
  \alpha_2 + \cdots + \alpha_5 \lhd 
  \alpha_2 + \alpha_3 + \alpha_4 \lhd 
  \alpha_2 + \alpha_3 \\
\lhd & \ 
  \alpha_3 + \cdots + \alpha_6 \lhd 
  \alpha_3 + \alpha_4 + \alpha_5 \lhd 
  \alpha_3 + \alpha_4 \lhd 
  \alpha_3.
\end{align*}
We set $x = s_1 s_4 s_3 s_2 s_6 s_5 s_4 s_3 \in \WJ$; 
remark that $x \geq s_1 s_2 s_6 s_5 s_4 s_3 = \mcr{s_{\theta}}$. 
Then we see that $\ed(\bBG{x}) = \bigl\{ x_{1},\,x_{2},\,x_{3},\,x_{4} \bigr\}$, where 
\begin{align*}
& x_1 := x = s_1 s_4 s_3 s_2 s_6 s_5 s_4 s_3, & & x_2 := s_2 s_1 s_4 s_3 s_2 s_6 s_5 s_4 s_3, \\
& x_3 := s_1 s_5 s_4 s_3 s_2 s_6 s_5 s_4 s_3, & & x_4:= s_2 s_1 s_5 s_4 s_3 s_2 s_6 s_5 s_4 s_3.
\end{align*}
From \eqref{eq:introThm_Chevalley1}, we see that
\begin{align*}
& [\CO_{\Bv{x}}] \star [\CO_{G/\PrJ}(- \vpi_3)] \\
& \hspace{10mm} = \be^{-x\vpi_3} \left(
  \sum_{ y \in \ed(\bBG{x}) } (-1)^{\ell(y) - \ell(x)}[\CO_{ \Bv{y} }] +
  \sum_{ y \in \ed(\bBG{x}) } (-1)^{\ell(y) - \ell(x)+1}[\CO_{ \Bv{ \mcr{ys_{3}} } }] Q_{3}\right) \\[2mm]
& \hspace{10mm} = \be^{- x\vpi_3} 
  \bigl( [\CO_{\Bv{x_1}}] - [\CO_{\Bv{x_2}}] - [\CO_{\Bv{x_3}}] + [\CO_{\Bv{x_4}}] \\
& \hspace{20mm} 
  - [\CO_{ \Bv{ \mcr{x_1 s_3} } }]Q_{3} 
  + [\CO_{ \Bv{ \mcr{x_2 s_3} } }]Q_{3} 
  + [\CO_{ \Bv{ \mcr{x_3 s_3} } }]Q_{3} 
  - [\CO_{ \Bv{ \mcr{x_4 s_3} } }]Q_{3} \bigr) \nonumber \\
& \hspace{10mm} = \be^{- x\vpi_3} 
  \bigl( [\CO_{\Bv{x_1}}] - [\CO_{\Bv{x_2}}] - [\CO_{\Bv{x_3}}] + [\CO_{\Bv{x_4}}] \\
& \hspace{20mm} 
  - [\CO_{\Bv{y_1}}]Q_{3} 
  + [\CO_{\Bv{y_2}}]Q_{3} 
  + [\CO_{\Bv{y_3}}]Q_{3} 
  - [\CO_{\Bv{y_4}}]Q_{3} \bigr), 
\end{align*}
where
\begin{equation*}
y_1 := s_4 s_3, \qquad y_2 := s_2 s_4 s_3, \qquad y_3 := s_5 s_4 s_3, \qquad y_4 := s_2 s_5 s_4 s_3.
\end{equation*}
Therefore, by \eqref{rem:introRem}, we deduce that 
\begin{align}
& [\CO_{\Bv{x}}] \star [\CO_{\Bv{s_3}}] = 
  [\CO_{\Bv{s_3}}] - \be^{-\vpi_3} [\CO_{\Bv{x}}] \star [\CO_{G/P^i}(-\vpi_3)] \nonumber \\
& \hspace{10mm} = 
  [\CO_{\Bv{s_3}}] - \be^{x\vpi_3 -\vpi_3} 
  \bigl( [\CO_{\Bv{x_1}}] - [\CO_{\Bv{x_2}}] - [\CO_{\Bv{x_3}}] + [\CO_{\Bv{x_4}}] \nonumber \\
& \hspace{20mm} 
  - [\CO_{\Bv{y_1}}]Q_{3} 
  + [\CO_{\Bv{y_2}}]Q_{3} 
  + [\CO_{\Bv{y_3}}]Q_{3} 
  - [\CO_{\Bv{y_4}}]Q_{3} \bigr). \label{eq:exa6}
\end{align}

Let us compare \eqref{eq:exa6} with 
the equation in \cite[Corollary 3.10]{BCMP}. 
Recall from \cite[Sect.~3.1]{BCMP} the definition of $w_{\mu}$ for a Young diagram $\mu$; 
we can verify that if we take 
$\mu := \raisebox{0.6em}{\ytableausetup{boxsize = 0.5em}\ydiagram{4,3,1}}$, 
then $w_\mu = x$. 
For simplicity of notation, we denote by $\mu$ the Schubert class $[\CO_{\Bv{w_\mu}}]$ 
for a Young diagram $\mu$. By \cite[Example~3.12]{BCMP}, we have
\begin{align*}
\ytableausetup{smalltableaux}
& [\CO_{\Bv{s_3}}] \star [\CO_{\Bv{w_\mu}}] \\
& \hspace{10mm} = [\CO_{\Bv{w_\mu}}] - 
 J_{w_\mu} \biggl( \left(\, 
 \raisebox{0.7em}{\ydiagram{4,3,1}} - 
 \raisebox{0.7em}{\ydiagram{4,3,2}}- 
 \raisebox{0.7em}{\ydiagram{4,4,1}} + 
 \raisebox{0.7em}{\ydiagram{4,4,2}}\, \right) \\
& \hspace{20mm} - Q_{3} 
  \left(\,
  \raisebox{-0.1em}{\ydiagram{2}} - 
  \raisebox{0.3em}{\ydiagram{2,1}} - 
  \raisebox{-0.1em}{\ydiagram{3}} + 
  \raisebox{0.3em}{\ydiagram{3,1}} \, \right) \biggr); 
\end{align*}
for the definition of $J_{w_\mu}$, see \cite[Sect.~3.4]{BCMP}. 
Note that the variable $q$ in \cite{BCMP} is identical to 
$Q_{3}$ in this paper. Since 
\begin{align*}
\ytableausetup{boxsize = 0.4em}
& w_{\raisebox{0.4em}{\ydiagram{4,3,1}}} = x_1, \qquad 
  w_{\raisebox{0.4em}{\ydiagram{4,4,2}}} = x_2, \qquad 
  w_{\raisebox{0.4em}{\ydiagram{4,4,1}}} = x_3, \qquad 
  w_{\raisebox{0.4em}{\ydiagram{4,4,2}}} = x_4, \\
& w_{\ydiagram{2}} = y_1, \qquad 
  w_{\raisebox{0.2em}{\ydiagram{2,1}}} = y_2, \qquad 
  w_{\ydiagram{3}} = y_3, \qquad 
  w_{\raisebox{0.2em}{\ydiagram{3,1}}} = y_4,
\end{align*}
it follows that
\begin{align}
& [\CO_{\Bv{s_3}}] \star [\CO_{\Bv{w_\mu}}] \nonumber \\
& \hspace{10mm} = 
  [\CO_{\Bv{s_3}}] - J_{w_\mu} \bigl( 
   [\CO_{\Bv{x_1}}] - [\CO_{\Bv{x_2}}] - 
   [\CO_{\Bv{x_3}}] + [\CO_{\Bv{x_4}}] \nonumber \\
& \hspace{20mm} 
  - [\CO_{\Bv{y_1}}]Q_{3} 
  + [\CO_{\Bv{y_2}}]Q_{3} 
  + [\CO_{\Bv{y_3}}]Q_{3} 
  - [\CO_{\Bv{y_4}}]Q_{3} \bigr) \label{eq:exa6a}. 
\end{align}
By \cite[Lemma~3.4]{BCMP}, we have $J_{w_\mu} [\CO_{\Bv{y}}] = 
\be^{w_\mu \vpi_3 - \vpi_3} [\CO_{\Bv{y}}] = 
\be^{x \vpi_3 - \vpi_3} [\CO_{\Bv{y}}]$ for $y \in \WJ$. 
Thus, equation \eqref{eq:exa6} agrees with equation \eqref{eq:exa6a}. 

%===================%
%     REFERENCES    %
%===================%
%

\end{document}